\documentclass[10pt]{article}
\usepackage{mathrsfs}
\usepackage{latexsym}
\usepackage{graphicx}
\usepackage{amscd,amssymb,amsmath,amsbsy,amsthm,amsfonts}
\usepackage[abbrev,lite,alphabetic]{amsrefs}
\usepackage[mathscr]{eucal}
\usepackage[all]{xy}
\usepackage{xcolor}
\definecolor{reference}{rgb}{0.20,0.36,0.74}
\definecolor{citation}{rgb}{0,.40,.80}
\usepackage[colorlinks,plainpages,urlcolor=blue,linkcolor=reference,citecolor=citation]{hyperref}
\usepackage{verbatim}
\usepackage{enumerate}
\usepackage{cleveref}
\usepackage{tikz}
\usepackage{textcomp}
\usepackage{mathtools}
\usepackage{graphicx}
\usepackage[left=2cm, top=2cm, right=2cm, bottom=2cm]{geometry}

\usepackage{enumitem}

\hypersetup{linktocpage}

\newcommand{\fcat}{\mathscr}
\newcommand{\inj}{\hookrightarrow}
\newcommand{\surj}{\twoheadrightarrow}
\newcommand{\areq}{\mathbin{{\xrightarrow{\,\sim\,}}}}

\newcommand{\Type}{ {\fcat{S}} }

\DeclareMathOperator{\Sch}{Sch}

\renewcommand{\phi}{\varphi}
\renewcommand{\epsilon}{\varepsilon}

\DeclareMathOperator{\Vect}{Vect}
\DeclareMathOperator{\rank}{rank}
\DeclareMathOperator{\Sym}{Sym}
\newcommand{\Mod}{\mathrm{Mod}}

\DeclareMathOperator{\SL}{SL}
\DeclareMathOperator{\gr}{gr}

\DeclareMathOperator{\Rep}{Rep}

\renewcommand{\Im}{\mathrm{Im}}

\DeclareMathOperator{\Spec}{Spec}

\DeclareMathOperator{\Pic}{Pic}

\DeclareMathOperator{\Coh}{Coh}

\DeclareMathOperator{\QCoh}{QCoh}

\DeclareMathOperator{\Aff}{Aff}
\DeclareMathOperator{\Stk}{{\mathscr{S}tk}}
\DeclareMathOperator{\PStk}{{\mathscr{PS}tk}}

\DeclareMathOperator{\Hom}{Hom}
\DeclareMathOperator{\Ext}{Ext}
\DeclareMathOperator{\Tor}{Tor}
\newcommand{\HHom}{ {\fcat Hom} }

\DeclareMathOperator{\Map}{Map}
\DeclareMathOperator{\Fun}{Fun}

\DeclareMathOperator{\Ind}{Ind}

\DeclareMathOperator*{\colim}{colim}
\DeclareMathOperator*{\fib}{fib}

\newcommand{\SSet}{ {\mathrm{SSet}} }

\DeclareMathOperator{\CAlg}{CAlg}

\DeclareMathOperator{\Lan}{Lan}
\DeclareMathOperator{\Ran}{Ran}

\DeclareMathOperator{\Sh}{{\mathcal Shv}}
\DeclareMathOperator{\ev}{ev}

\newcommand{\prolim}[1][]{\lim\limits_{\stackrel{\longleftarrow}{#1}}}
\newcommand{\indlim}[1][]{\lim\limits_{\stackrel{\longrightarrow}{#1}}}

\DeclareMathOperator{\Tot}{Tot}

\DeclareMathOperator{\Fil}{Fil}
\DeclareMathOperator{\red}{red}

\mathchardef\mdef="2D

\usepackage{tocloft}
\DeclareFontFamily{OT1}{pzc}{}
\DeclareFontShape{OT1}{pzc}{m}{it}{<-> s * [1.200] pzcmi7t}{}
\DeclareMathAlphabet{\mathpzc}{OT1}{pzc}{m}{it}

\DeclareMathOperator{\Poly}{Poly}

\usepackage[english]{babel}
\usepackage[autostyle, english = american]{csquotes}
\MakeOuterQuote{"}
\usepackage{color}

\newtheorem{thm}{Theorem}[subsection]
\newtheorem*{thm*}{Theorem}
\newtheorem{lem}[thm]{Lemma}
\newtheorem{cor}[thm]{Corollary}
\newtheorem*{cor*}{Corollary}
\newtheorem{prop}[thm]{Proposition}
\newtheorem*{prop*}{Proposition}

\theoremstyle{definition}
\newtheorem{defn}[thm]{Definition}
\newtheorem{construction}[thm]{Construction}
\newtheorem{notation}[thm]{Notation}
\newtheorem{ex}[thm]{Example}

\newtheorem{rem}[thm]{Remark}
\newtheorem{question}[thm]{Question}

\newtheorem*{conj*}{Conjecture}

\definecolor{note_color}{rgb}{0.0,0.9,0.0}

\newcommand{\clause}[1]{\smallskip\noindent\underline{\textit{#1}}}

\usepackage{bbm}
\usepackage{color}
\usepackage{enumerate}
\usepackage{array}
\usepackage{graphicx}
\usepackage{systeme}
\usepackage{amssymb}

\newcommand{\Grad}{\mr{Grad}}
\newcommand{\Filt}{\mr{Filt}}
\newcommand{\ol}{\overline}
\newcommand{\Ker}{\mr{Ker}}

\newcommand{\mf}{\mathfrak}
\newcommand{\mc}{\mathcal}
\newcommand{\mbb}{\mathbb}
\newcommand{\mr}{\mathrm}
\newcommand{\mcr}{\mathscr}

\newcommand{\ul}{\underline}
\newcommand{\ra}{\rightarrow}
\newcommand{\xra}{\xrightarrow}

\newcommand{\RG}{{R\Gamma}}
\newcommand{\Zar}{\mathrm{Zar}}
\newcommand{\op}{{\mathrm{op}}}
\newcommand{\fin}{{\mathrm{fin}}}
\newcommand{\perf}{{\mathrm{perf}}}

\newcommand{\Hdg}{{\mathrm{H}}}
\newcommand{\dR}{{\mathrm{dR}}}
\newcommand{\crys}{{\mathrm{crys}}}

\newcommand{\fp}{{\mathrm{fp}}}
\newcommand{\pr}{{\mathrm{pr}}}

\newcommand{\Art}{{\mathrm{Art}}}

\newcommand{\sm}{{\mathrm{sm}}}
\newcommand{\cnj}{{\mathrm{conj}}}
\newcommand{\mstack}{\mathpzc}

\newcommand{\DMod}[1]{D(\Mod_{#1})}
\newcommand{\UMod}[1]{\Mod_{#1}}

\usepackage{tikz}
\usetikzlibrary{matrix}

\numberwithin{thm}{subsection}

\title{Hodge-to-de Rham degeneration for stacks}
\author{Dmitry Kubrak, Artem Prikhodko}
\date{}

\begin{document}
\maketitle

\begin{abstract}
We introduce a notion of a Hodge-proper stack and apply the strategy of Deligne-Illusie to prove the Hodge-to-de Rham degeneration in this setting. In order to reduce the statement in characteristic $0$ to characteristic $p$, we need to find a good integral model of a stack (namely, a Hodge-proper spreading), which, unlike in the case of proper schemes, need not to exist in general. To address this problem we investigate the property of spreadability in more detail by generalizing standard spreading out results for schemes to higher Artin stacks and showing that all proper and some global quotient stacks are Hodge-properly spreadable. As a corollary we deduce a (non-canonical) Hodge decomposition of the equivariant cohomology for certain classes of varieties with an algebraic group action.
\end{abstract}

\tableofcontents

\section{Introduction}
\subsection{Deligne-Illusie method for schemes}
Let $X$ be a smooth scheme over $\mathbb C$ and let $X(\mathbb C)$ be the topological space of its complex points. Grothendieck has shown that there is a formula for the singular cohomology of $X(\mathbb C)$ in purely algebraic terms, namely 

$$
H^n_{\mr{sing}}(X(\mathbb C),\mathbb C)\simeq H^n_{\mr{dR}}(X/\mathbb C),
$$
where the de Rham cohomology $H^n_{\mr{dR}}(X/\mathbb C)$ is defined as the $n$-th hypercohomology of the algebraic de Rham complex of $X$. If, moreover, $X$ is projective, using Hodge theory one obtains the Hodge decomposition
$$
H^n_{\mr{sing}}(X(\mathbb C),\mathbb C)\simeq \bigoplus_{p+q=n}H^q(X,\Omega^p_X).
$$
Unfortunately, it is only possible to get such a decomposition utilizing some transcendental methods (like Hodge theory). However, for $X$ proper, using just algebraic geometry we still obtain a functorial filtration $F^\bullet H^n_{\mr{dR}}(X/\mathbb C)$ whose associated graded is given by the sum above. Namely, the de Rham complex has a natural cellular (also called "stupid") filtration $\Omega_{X,\dR}^{\ge p}$ given by subcomplexes
$$
\Omega_{X,\dR}^{\ge p} \coloneqq \xymatrix{\ldots \ar[r] &0 \ar[r] & \Omega_{X}^p \ar[r]^-d & \Omega_{X}^{p+1} \ar[r]^-d & \ldots \ar[r]^-d & \Omega_X^{\dim X}}.
$$
This filtration induces a filtration on the complex of global sections $\RG_\dR(X/\mathbb C)\coloneqq\RG(X, \Omega_{X,\dR}^\bullet)$ whose associated graded pieces are $\RG(X, \Omega_X^p[-p])$. As a consequence one gets the so-called \emph{Hodge-to-de Rham spectral sequence}
$$E_1^{p,q} = H^q(X, \Omega_{X}^p) \ \Rightarrow\ H^{p+q}_\dR(X/\mathbb C).$$
As was shown by Deligne and Illusie \cite{DeligneIllusie} there is a purely algebraic proof of the degeneration of the spectral sequence above, thus the induced filtration $F^\bullet H^{n}_\dR(X/\mathbb C)$ on the de Rham cohomology has the associated graded
$$
\gr_F^\bullet H^{n}_\dR(X/\mathbb C) \simeq \bigoplus_{p+q=n} H^{p,q}(X),  \quad \text{where}\quad H^{p,q}(X):=H^q(X,\Omega^p_X).
$$

The strategy of Deligne-Illusie is to reduce the statement in characteristic $0$ to an analogous question in big enough positive characteristic. Let $k$ be a perfect field of characteristic $p$ and let $Y$ be a smooth scheme over $k$. Then we have:
\begin{thm}[Cartier]\label{Cartier_iso}
Let $Y^{(1)}$ denote the Frobenius twist of $Y$ and let $\phi_Y\colon Y \to Y^{(1)}$ be the relative Frobenius morphism. Then there exists a unique isomorphism of sheaves of $\mathcal O_{Y^{(1)}}$-algebras on $Y^{(1)}_{\Zar}$
$$C^{-1}_Y\colon \bigoplus_i \Omega^i_{Y^{(1)}} \to \bigoplus_i \mathcal H^i(\phi_{Y*}\Omega^\bullet_{Y, \dR}),$$
determined by the property that for any local section $f$ of $\mathcal O_Y$
$$C^{-1}_Y(df) = ``df^p/p" \coloneqq f^{p-1}df.$$
The map $C^{-1}_Y$ is called the \emph{inverse Cartier isomorphism}.
\end{thm}
This way we see that the Postnikov (also called "canonical") filtration on $\phi_{Y*}\Omega_{Y, \dR}^\bullet$ induces another filtration on  $\RG(Y, \Omega_{Y, \dR}^\bullet) \simeq \RG(Y^{(1)}, \phi_{Y*}\Omega_{Y, \dR}^\bullet)$ whose associated graded pieces are $\RG(Y^{(1)}, \Omega_{Y^{(1)}}^p[-p])$. Taking the spectral sequence induced by this filtration we obtain the so-called \emph{conjugate spectral sequence}
$$
E_2^{p,q} = H^p(Y^{(1)}, \Omega^q_{Y^{(1)}}) \Rightarrow H^{p+q}_\dR(Y/k).
$$
Note that for any spectral sequence the $E_\infty$-page is always a subfactor of the $E_r$-page ($r\ge 0$), hence $\dim_k E^{*,*}_\infty \le \dim_k E^{*,*}_r$. If all vector spaces $E_*^{*,*}$ are finite-dimensional, equality holds if and only if all differentials starting  from the $r$-th page vanish. It follows that for $Y$ proper, the conjugate spectral sequence degenerates if and only if $\dim_k H^n_\dR(Y) = \sum_{p+q=n} \dim_k H^{p,q}(Y^{(1)})$. Since $\dim_k H^{p,q}(Y^{(1)}) = \dim_k H^{p,q}(Y)$ this happens if and only if the Hodge-to-de Rham spectral sequence degenerates as well.

The differentials in the conjugate spectral sequence are induced by the connecting homomorphisms for the Postnikov filtration on $\phi_{Y_*} \Omega_{Y, \dR}^\bullet$. In particular, if $\phi_{Y_*}\Omega_{Y, \dR}^\bullet$ is formal (i.e. quasi-isomorphic to the sum of its cohomology), then the conjugate spectral sequence degenerates. While in general this is hard to guarantee, the formality of the truncation $\tau^{\le p-1}\phi_{Y_*} \Omega_{Y, \dR}^\bullet$ turns out to be equivalent to the existence of a lift to the second Witt vectors $W_2(k)$:
\begin{thm}[Deligne-Illusie]\label{Deligne_Illusie_schems}
A smooth scheme $Y$ over $k$ admits a lift to $W_2(k)$ if and only if there exists an equivalence in the derived category of $\mathcal O_{Y^{(1)}}$-modules
$$\bigoplus_{i=0}^{p-1} \Omega_{Y^{(1)}}^i[-i] \areq \tau^{\le p-1}\phi_{Y*}\Omega_{Y, \dR}^\bullet$$
inducing the inverse Cartier isomorphism $C^{-1}_Y$ on $\mathcal H^*$. In particular, if $Y$ admits a lift to $W_2(k)$ and $\dim Y < p$, then the complex $\phi_{Y*}\Omega_{Y, \dR}^\bullet$ is formal, and hence the Hodge-to-de Rham spectral sequence degenerates at the first page.
\end{thm}
The proof of the degeneration in characteristic $0$ is then accomplished by choosing a smooth proper model (the so-called \emph{spreading}) $X_R$ of $X$ over some finitely generated $\mathbb Z$-subalgebra $R$ of $F$. Enlarging $R$ if needed, one can assume that the $R$-modules $H^q(X_R,\Omega^p_{X_R})$ and $H^n_\dR(X_R/R)$ are free of finite rank and that $R$ is smooth over $\mathbb Z$. By smoothness of $R$ any homomorphism from $R$ to a perfect filed of positive characteristic lifts to the second Witt vectors (see \Cref{lemma below}). Picking a perfectization of closed point of $\Spec R$ of residue characteristic $p>\dim X$ one reduces Hodge-to-de Rham degeneration to \Cref{Deligne_Illusie_schems}.

\subsection{Generalization to stacks}
In this work we apply the strategy of Deligne-Illusie in the case of Artin stacks. For a smooth proper Deligne-Mumford stack one can proceed with the original arguments (see e.g. \cite[Corollary 1.7]{Satriano}), but they do not seem to work for a general smooth Artin stack (see \Cref{original_DI_fails}). Instead we use another approach relying on quasi-syntomic descent for the derived de Rham cohomology. 

As in the case of schemes, to establish Hodge-to-de Rham degeneration, we need to impose some properness assumptions. However, the standard notion of a proper stack is too restrictive for our purposes. For example, the quotient stack $[X/G]$ of a proper scheme $X$ by an action of a linear algebraic group $G$ is proper if and only if the stabilizers of all points of $X$ are finite group schemes. On the other hand, as we will see in \Cref{sec:reductive quotients}, the Hodge-to-de Rham spectral sequence for $[X/G]$ with reductive  $G$ always degenerates. 

This suggests that we should look for a more general notion of properness:
\begin{defn}
Let $R$ be a Noetherian ring. A smooth Artin stack $\mstack X$ over $R$ is called \emph{Hodge-proper} if $H^q(\mstack X, \wedge^p \mathbb L_{\mstack X/R})$ is a finitely generated $R$-module for all $p$ and $q$, where $\mathbb L_{\mstack X/R}$ is the cotangent complex of $\mstack X$ over $R$.
\end{defn}
The complex $\RG(\mstack X, \wedge^p \mathbb L_{\mstack X/R})$ is a natural analogue of $\RG(X, \Omega^p_X)$ and, similarly to the scheme case, the de Rham cohomology complex $\RG_{\mr{dR}}(\mstack X/R)$ has a natural (Hodge) filtration whose associated graded pieces are $\RG(\mstack X, \wedge^p \mathbb L_{\mstack X/R}[-p])$; see \Cref{sec:Hodge and de Rham cohomology} for more details. In this way one obtains a spectral sequence
$$
E_1^{p,q} = H^q(\mstack X, \wedge^p \mathbb L_{\mstack X/R}) \ \Rightarrow\ H^{p+q}_\dR(\mstack X/R).
$$    
In the case $R=F$ is a field this spectral sequence degenerates if and only if
\begin{align}\label{stacks_dim_eq_to_prove}
\dim_F H^{n}_\dR(\mstack X/F) =\sum_{p+q=n} \dim_F H^q(\mstack X, \wedge^p \mathbb L_{\mstack X/F}).
\end{align}
\begin{rem}
By smooth descent for the cotangent complex, $\RG(\mstack X, \wedge^p \mathbb L_{\mstack X/F})$ produces the same answer as the definition of the Hodge cohomology via the lisse-\'etale site of $\mstack X$ (see \Cref{Hodge_as_limit}).
\end{rem}
We will now explain the strategy of our proof of the equality \eqref{stacks_dim_eq_to_prove} above. The first step is to extend \Cref{Deligne_Illusie_schems} to the setting of stacks:
\begin{thm*}[\ref{DeligneIllusieQuis}]
Let $\mstack Y$ be a smooth quasi-compact quasi-separated Artin stack over a perfect field $k$ of characteristic $p$ admitting a smooth lift to the ring of the second Witt vectors $W_2(k)$. Then there is a canonical equivalence
$$\RG(\mstack Y, \tau^{\le p-1} \Omega_{\mstack Y,\dR}^\bullet) \simeq \RG\left(\mstack Y^{(1)}, \bigoplus_{i=0}^{p-1} \wedge^i \mathbb L_{\mstack Y^{(1)}/k}[-i]\right).$$
In particular for $n\le p-1$ we have $H^n_\dR(\mstack Y/k) \simeq H^n_\Hdg(\mstack Y^{(1)}/k)$.
\end{thm*}

\begin{rem}
	Note that \Cref{DeligneIllusieQuis} gives only a partial generalization of \Cref{Deligne_Illusie_schems}. Even though the statement indeed follows from the analogous splitting of sheaves (see \Cref{Functorial_DeligneIllusie_for_affines} below) there is no analogue of the "if and only if" statement of the Deligne-Illusie theorem. One reason for this is that the original approach of Deligne-Illusie is poorly suited for general Artin stacks (see \Cref{original_DI_fails}); so instead we use the (slightly enhanced) proof of the splitting due to Fontaine-Messing \cite[Section II]{FontaineMessing_padciPeriods}.
\end{rem}
Since the de Rham cohomology for Artin stacks are defined as the right Kan extension from smooth affine schemes (\Cref{def: de Rham}) one can more or less formally deduce the theorem above from the following very functorial form of Deligne-Illusie splitting for affine schemes:
\begin{thm*}[{\ref{Functorial_DeligneIllusie_for_affines}}]
Let $\Aff^\sm_{/W_2(k)}$ be the category of smooth affine schemes over $W_2(k)$. Then there is a natural $k$-linear equivalence of functors
$$
\bigoplus_{i=0}^{p-1} \Omega^i_{-}\colon B\mapsto \bigoplus_{i=0}^{p-1} \Omega^i_{(B^{(1)}/p)/k}[-i] \ \ \ \text{ and } \ \ \ \tau^{\le p-1} \Omega^\bullet_{-,\dR} \colon B\mapsto \tau^{\le p-1} \Omega^\bullet_{(B/p)/k, \dR}
$$
from $\Aff^{\sm,\op}_{/W_2(k)}$ to the $\infty$-category $\DMod{k}$ which induces the Cartier isomorphism on the level of the individual cohomology functors.
\end{thm*}
The splitting in \Cref{Deligne_Illusie_schems} is already functorial with respect to liftings to $W_2(k)$, but only on the level of the underlying homotopy category and not the $\infty$-category of complexes $\DMod{k}$ itself. To get this higher functoriality we follow \cite[Section II]{FontaineMessing_padciPeriods} using a more convenient language of \cite{BMS2}.

The idea is to extend the de Rham (and the crystalline) cohomology functor to a larger category of quasisyntomic algebras (\Cref{def:qsyn}). This category, endowed with the quasisyntomic topology, has a basis consisting of quasi-regular semiperfectoid $W_n(k)$-algebras (\Cref{def:qrsp}), on which the values of $\RG_{\dR}$ and $\RG_\crys$ (or rather their derived versions $\RG_{\mbb L\dR}$ and $\RG_{\mbb L\crys}$) become ordinary rings. Additionally, the Frobenius morphism, the Hodge filtration and the conjugate filtration can be described explicitly. This way, using quasi-syntomic descent, the question reduces to a certain computation in commutative algebra.

More concretely, for a quasi-regular semiperfect $k$-algebra $S$ one can prove that $\RG_{\mathbb L\crys}(S/W_n(k)) \simeq \mathbb A_{\crys}(S)/p^n$, where $\mathbb A_\crys(S)$ is the divided power envelope of the kernel of the natural surjection $W((S)^\flat) \surj S$ (see \Cref{constr:Acrys}). Under this identification the Hodge filtration on $\RG_{\crys}(-/k) \simeq \RG_{\dR}(-/k)$ corresponds to the filtration by the divided powers of the pd-ideal $I \triangleleft \mathbb A_{\crys}(S)/p$. The conjugate filtration $\Fil_*^\cnj$ admits an explicit description as well (see \Cref{constr:ConjFil}). Given a lifting $\widetilde S$ of $S$ to $W_2(k)$ there is a natural morphism $\theta\colon\mathbb A_{\crys}(S)/p^2\ra \widetilde S$. The image of $K\coloneqq\ker \theta$ under the first divided Frobenius map $\phi_1$ then provides a splitting of $\Fil_1^\cnj$ into $\Fil_0^\cnj\oplus \Fil_1^\cnj/\Fil_0^\cnj\simeq S^\flat/I\oplus I/I^2$ (\Cref{Hodge_splitting_on_QRSPerf}). By multiplicativity this extends to the splitting of $\Fil_{p-1}^\cnj$ whose descent to smooth schemes gives {\Cref{Functorial_DeligneIllusie_for_affines}}.
\begin{rem}\label{original_DI_fails}
The original approach of Deligne-Illusie (at least applied literally) does not seem to work for a general Artin stack; the key result of \cite{DeligneIllusie} is the equivalence of two gerbes on the \'etale site of $Y^{(1)}/k$ for a smooth $k$-scheme $Y$: the one of splittings of $\tau^{\le 1}\phi_{Y*}\Omega_{X, \dR}^\bullet$ in $\QCoh(Y^{(1)})$ and the one of liftings of $Y^{(1)}$ to $W_2(k)$. A general smooth Artin stack $\mstack Y$ can be covered by an affine scheme only \emph{smooth locally}, so one needs to replace the \'etale site of $\mstack Y$ by the smooth one. But both the space of splittings of $\tau^{\le 1}\phi_{Y*}\Omega_{Y, \dR}^\bullet$ and the space of liftings to $W_2(k)$ are not even presheaves there. Nevertheless, it would be still interesting to have an explicit description of the space of liftings to $W_2(k)$ for an arbitrary smooth $n$-Artin stack $\mstack Y$. We do not discuss this question here.
\end{rem}

\paragraph{Spreadings.}
Let now $\mstack X$ be a smooth Hodge-proper stack over a field $F$ of characteristic $0$. If there exists $\mathbb Z$-subalgebra $R \subset F$, which is finitely generated over $\mathbb Z$\footnote{More generally, in \Cref{def: spredable_stacks} we also allow subrings $R\subset F$ that are localizations of finitely generated $\mathbb Z$-algebras under the assumption that the image of $\Spec R$ in $\Spec \mbb Z$ is open, but it is fine to assume that $R$ is finitely generated throughout the introduction.}, and a Hodge-proper stack $\mstack X_R$ over $R$ such that $\mstack X_R \otimes_R F \simeq \mstack X$ (a \emph{Hodge-proper spreading of $\mstack X$}), then one can deduce the equality \eqref{stacks_dim_eq_to_prove} for the $n$-th cohomology from \Cref{DeligneIllusieQuis} by taking a suitable closed point $\Spec k \inj \Spec R$ of characteristic $p > n$ and considering the fiber $\mstack X_{k}$. This way we obtain
\begin{thm*}[\ref{Hodge_degeneration}]
Let $\mstack X$ be a smooth Hodge-properly spreadable Artin stack over a field $F$ of characteristic zero. Then the Hodge-to-de Rham spectral sequence for $\mstack X$ degenerates at the first page. In particular for each $n\ge 0$ there exists a (non canonical) isomorphism
$$H^n_\dR(\mstack X) \simeq \bigoplus_{p+q=n} H^{p,q}(\mstack X).$$
\end{thm*}

We must warn the reader that smooth Hodge-properly spreadable stacks do not enjoy many of the nice properties that smooth proper schemes have, in particular the natural mixed Hodge structure on the singular $n$-th cohomology is not necessarily pure (see \Cref{rem:not pure}). The main motivation for the definition is that it is the most general class of stacks for which the Deligne-Illusie method can be applied. This, however, does not exceed all examples of the Hodge-to-de Rham degeneration (see \Cref{rem:degenerate but not spreadable}).

In order to address the question of the existence of a Hodge-proper spreading we first extend the standard spreading out results for finitely presentable schemes to the case of Artin stacks:
\begin{thm*}[\ref{fp_over_filtered_limit} and \ref{prop:  proper morphisms spread}]
Let $\{S_i\}$ be a filtered diagram of affine schemes with limit $S$. Fix a class of morphism $\mathcal P = $ proper, smooth, flat, surjective, or any other class of morphisms that satisfies the conditions of  \Cref{def:spreadable_class}. For an affine scheme $T$, let $\Stk_{/T}^{n\mdef\Art,\fp, \mathcal P}$ denote the category of finitely presentable $n$-Artin stacks over $T$ and morphisms in $\mathcal P$ between them. Then the natural functor 
$$\indlim[i] \Stk_{/S_i}^{n\mdef\Art,\fp, \mathcal P} \xymatrix{\ar[r] &} \Stk_{/S}^{n\mdef\Art,\fp, \mathcal P}$$
(induced by base-change) is an equivalence.
\end{thm*}
\noindent As a corollary we deduce that any smooth $n$-Artin stack $\mstack X$ over $F$ admits a smooth spreading $\mstack X_R$ over some finitely generated $\mathbb Z$-algebra $R\subset F$ and that any two such spreadings become equivalent after enlarging $R$. Since all smooth proper stacks are Hodge-proper (see \Cref{proper_are_Hodge_proper}), we immediately deduce the Hodge-to-de Rham degeneration in this case. Note that this includes smooth proper Deligne-Mumford stacks as a special case.

However, Hodge-proper spreadings need not to exist in general: one can show that the classifying stack  $BG$ is Hodge-proper for any finite-type group scheme $G$ over $F$ (see \Cref{BG}) but it is not necessarily Hodge-properly spreadable. Indeed, the classifying stack $B\mathbb G_a$ of the additive group has nontrivial Hodge cohomology but is de Rham contractible (i.e. has the de Rham cohomology of a point), so the Hodge-to-de Rham spectral sequence is clearly nondegenerate. By \Cref{Hodge_degeneration} it follows that it is not Hodge-properly spreadable and this forces the Hodge cohomology of $B\mathbb G_{a, \mathbb Z}$ to have infinitely generated $p$-torsion for a dense set of primes $p$, which one can also see from the explicit description (see \Cref{ex: comment on BG_a} or \Cref{prop: description of the cohomology of BG_a}). This illustrates the general phenomenon: the non-degeneracy of the Hodge-to-de Rham spectral sequence in characteristic $0$ is always reflected arithmetically, namely the Hodge cohomology of any spreading \emph{has to be} infinitely generated over the base.

In the main case of our interest, namely the quotient stacks $\mstack X = [X/G]$,
we exhibit some sufficient conditions for Hodge-proper spreadability purely in terms of the geometry of $X,G$ and the action $G \curvearrowright X$. In this case the spreadability is not easy to show, especially if we can't spread $G$ to a linearly reductive group (which is only possible if $G$ is a torus or an extension of a finite group by one). Nevertheless, using certain cohomological finiteness results from \cite{FranjouKallen} we prove
\begin{thm*}[\ref{thm:proper coarse moduli}]
Let $F$ be an algebraically closed field of characteristic $0$. Let $X$ be a smooth scheme and let $Y$ be a finite-type quasi-separated scheme over $F$, both endowed with an action of a reductive group $G$. Assume that
\begin{enumerate}
\item There is a proper $G$-equivariant map $\pi\colon X\ra Y$.
\item The $G$-action on $Y$ is locally linear (Definition \ref{def:locally linear action}).
\item The categorical quotient $Y/\!/G$ is proper.
\end{enumerate}
Then the quotient stack $[X/G]$ is Hodge-properly spreadable\footnote{
In fact we prove a stronger statement, namely that $[X/G]$ is cohomologically properly spreadable, see Definitions \ref{def:cohomologically proper morphism} and \ref{def:coh_prop_spread}.
}.
\end{thm*}
\noindent \Cref{thm:proper coarse moduli} applies to some natural examples of smooth schemes $X$ with a $G$-action, in particular, equivariant "proper-over-affine" varieties (see \Cref{ex:projective-over-affine}) and the GIT quotients, whose coarse moduli space is proper (see \Cref{ex:GIT}). 

We also prove a variant of \Cref{thm:proper coarse moduli} where we drop the reductivity assumption on $G$ but impose an additional Bialynicki-Birula (BB)-completeness assumption on the action when restricted to a subgroup $h\colon \mathbb G_m\ra G$. Moreover, the extra structure given by the map $\pi$ is replaced by the internal condition on the properness of $h(\mathbb G_m)$-fixed points; see \Cref{thm: conical examples} for details. 

Using the results of Halpern-Leistner (specifically, \cite{HL-prep}) on $\Theta$-stratifications we also show that a smooth stack $\mstack X$, which is endowed with a $\Theta$-stratification such that all strata (including the semistable locus) are cohomologically properly spreadable, is also cohomologically properly spreadable (see \Cref{cor:degeneration for Theta-stratified stacks}). This gives rise to new examples of Hodge-properly spreadable stacks where old ones appear as individual $\Theta$-strata. In particular, this way, using \Cref{thm:proper coarse moduli} above, one can show that global quotients of KN-complete varieties are Hodge-properly spreadable; see \Cref{ex:Kempf-Ness}.

\smallskip As an application, for any Hodge-properly spreadable quotient stack $[X/G]$, we deduce an equivariant Hodge-to-de Rham degeneration:
\begin{cor*}[{\ref{cor:Hodge decomposition for singular cohomology}}]
Let $X$ be a smooth scheme over $\mathbb C$ endowed with an action of an algebraic group $G$ such that the quotient stack $[X/G]$ is Hodge-properly spreadable. Then there is a (non-canonical) decomposition
$$H^n_{G(\mathbb C)}(X(\mathbb C) , \mathbb C) \simeq \bigoplus_{p + q = n} H^q([X/G], \wedge^p \mathbb L_{[X/G]/\mathbb C}).$$
\end{cor*}

Finally, it turns out that \Cref{Hodge_degeneration} can be applied even in the case of some non-proper schemes, as we discuss in some detail in \Cref{sec:non-proper schemes}.

\begin{rem}
Even though all stacks in our main applications are classical (i.e. $1$-Artin), the machinery developed in this work to prove Hodge-to-de Rham degeneration applies automatically to higher Artin stacks, so we did not put any artificial restrictions on the level of representablity of stacks considered in the paper. An example of a genuinely higher stack to which our method applies can be found in \Cref{sec: higher examples}.
\end{rem}

\subsection{Relation to previous work and further directions}
Our definition of Hodge-proper stacks is partially motivated by the work \cite{RelaxedProperness} by Halpern-Leistner and Preygel, where several generalized notions of properness for stacks are studied. In Questions 1.3.2 and 1.3.3 of loc.cit. authors ask if any formally proper stack (Definition 1.1.3 of loc.cit.) admits a formally proper spreading and if the Hodge-to-de Rham spectral sequence degenerates for a formally proper stack over a field of characteristic $0$. It follows from our work that the first statement implies the second; however, the method of \Cref{sec:Spreadable stacks} does not help to show the existence of a formally proper spreading. In fact, for the degeneration, only the existence of a Hodge-proper spreading would suffice, but this still seems pretty hard to show (see Question \ref{que: is formally proper spreadable?} in the end of our paper).

The splitting of the $(p-1)$-st truncation of the de Rham complex for a smooth tame $1$-Artin stack over a perfect field $k$ of characteristic $p$ was established (among other things) in \cite{Satriano}. The key observation in \cite{Satriano} is that a smooth tame stack admits a smooth lift together with a lift of Frobenius  \'etale-locally on its coarse moduli space, which enables to follow the original argument of Deligne-Illusie. Our proof is different and works for an arbitrary smooth $n$-Artin stack. 

Even though the main examples of Hodge-spreadable stacks we construct in \Cref{sec:Examples of spreadable Hodge-proper stacks} are classical Artin stacks, \Cref{DeligneIllusieQuis} and \Cref{Hodge_degeneration} work equally well for higher ones. Thus we keep this level of generality throughout the paper. The spreading results of \Cref{sec:Spreadable stacks} in the case of classical Artin stacks are also covered in \cite[Appendix B]{Rydh_NoetherianApprox} and \cite[Chapter 4]{MoretLaumon}. The use of \cite[Section 4]{Pridham} gives a clear way to extend these results to the setting of higher stacks, which we record in \Cref{sec:Spreadable stacks}.

It is worth to mention that there is still no example of a smooth liftable scheme $X$ in characteristic $p$ whose Hodge-to-de Rham spectral sequence does not degenerate (recall that the Deligne-Illusie method gives such a degeneration only for $i+j<p$). Motivated by the recent examples of non-degeneration for the HKR-filtration constructed in \cite{Antieau-Bhatt-Mathew} one could first look for such a counterexample in the world of stacks. The de Rham cohomology of various classifying stacks were considered recently in great detail in \cite{Totaro}; however, in all examples the Hodge-to-de Rham spectral sequence did degenerate.

The equivariant Hodge-to-de Rham degeneration for a reductive group $G$ acting on a scheme $X$, under the Kempf-Ness-completeness assumption was treated (among other things) in \cite{Teleman} by completely different methods. We reprove his result in a (slightly) more general setting (\Cref{ex:Kempf-Ness}) using $\Theta$-stratifications and the associated semiorthogonal decompositions constructed in \cite{HL-prep}. The same strategy applies to any smooth $\Theta$-stratified stack with cohomologically properly spreadable centra of the strata and the semistable locus (\Cref{cor:degeneration for Theta-stratified stacks}). \cite[Section 4]{HL-instability} could provide more examples of stacks that are Hodge-properly spreadable.

Another approach to the equivariant Hodge theory was introduced in \cite{HLP_equiv_noncom}. There the authors deduce (among other things) the noncommutative Hodge-to-de Rham degeneration for the category of perfect complexes $\QCoh([X/G])^\perf$ for a KN-complete $X$ and for some purely non-commutative examples (like the categories of matrix factorizations), by exploiting methods of non-commutative geometry. Note that the result of Kaledin (see \cite{Kaledin} and \cite{Kaledin2}) does not apply in this situation, since the DG-category $\QCoh([X/G])^\perf$ is usually not smooth. It is natural to ask whether the commutative degeneration implies the noncommutative one in this case. This is not immediately clear, since the relation between the Hochschild/periodic cyclic homology of the category of perfect complexes and the Hodge/de Rham cohomology for Artin stacks is more subtle than in the case of schemes. 

\subsection{Plan of the paper}
\Cref{sec:Degeneration of the Hodge-to-de Rham spectral sequence} is devoted to a proof of the degeneration of the Hodge-to-de Rham spectral sequence for Hodge-properly spreadable stacks. In Subsections \ref{sec:Hodge and de Rham cohomology} and \ref{sec:Hodge-proper stacks} we review Hodge and de Rham cohomology of stacks, define Hodge-proper stacks and prove some technical lemmas about them. In \Cref{sec:Proof of the degeneration} we prove (a truncated version of) the Hodge-to-de Rham degeneration in positive characteristic for Hodge-proper stacks that admit a lift to $W_2(k)$. Then, in \Cref{sec:Degeneration in char 0} we prove the Hodge-to-de Rham degeneration in characteristic $0$ for stacks that are Hodge-properly spreadable. As a corollary, in \Cref{sec:Examples of equivariant Hodge degeneration}, in the case of a quotient stack, we also deduce a (non-canonical) Hodge decomposition for the corresponding equivariant singular cohomology.

In \Cref{sec:Spreadings of stacks} we study the spreadability of Hodge proper stacks. In Subsection \ref{sec:Spreadable stacks} we extend the standard spreading out results for finitely presented schemes and their morphisms to the case of Artin stacks (see \Cref{fp_over_filtered_limit}). In \ref{sec:Cohomologically proper stacks} we introduce a more convenient class of cohomologically proper stacks which includes all Hodge-proper ones. In \Cref{sec:Examples of spreadable Hodge-proper stacks} we give first examples of spreadable Hodge-proper stacks: in \Cref{sec:proper stacks} we cover the case of smooth proper stacks, in \Cref{sec:BG} we discuss for which algebraic groups $G$ the classifying stack $BG$ is Hodge-properly spreadable. Then, in \Cref{sec:non-proper schemes} we discuss the case schemes. 

In \Cref{sec:spreadability of quotient stacks} we concentrate on the spreadability of quotient stacks. In \Cref{sec:reductive quotients} we discuss the case of global quotients by reductive groups whose coarse moduli space is proper. In \Cref{sec:BB-complete} we prove Hodge-proper spreadability of $[X/\mathbb G_m]$ under the condition that the associated Bialynicki-Birula stratification is full and $X^{\mbb G_m}$ is proper; then, in \Cref{sec:more quotients} we use this to prove spreadability for a more general class of global quotients, including quotients by some non-reductive groups. Finally, in \Cref{sec:Teleman} we show that finite $\Theta$-stratifications spread out; then using the results of \cite{HL-prep} we show that if all $\Theta$-strata (or rather their centra) together with the semistable locus have cohomologically proper spreadings, then so does the original $\Theta$-stratified smooth stack $\mstack X$. In \Cref{ex:Kempf-Ness} we show how to establish the cohomologically proper spreadability of Kempf-Ness (KN-)complete quotient stacks using this method.

\paragraph{Acknowledgments.}
We are grateful to Bhargav Bhatt for his interest in our work and helpful advice concerning the Hodge-to-de Rham degeneration in characteristic $p$.\footnote{In fact his suggestion that one should be able to prove the Hodge-to-de Rham degeneration for some "cohomologically proper" stacks via the Deligne-Illusie method basically started this project.} We are also grateful to Daniel Halpern-Leistner for sharing the draft of \cite{HL-prep} and answering our questions about \cite{HL-GIT}. We would like to thank Davesh Maulik, Sasha Petrov, Vadim Vologodsky, and Roman Travkin for many helpful discussions. 
We also want to thank Borys Kadets and Chris Brav for carefully reading drafts of this text and providing many useful remarks and suggestions. Finally, we would like to thank anonymous referees for many helpful advises, various comments and questions concerning different sections of the paper.

The second author was partially supported by Laboratory of Mirror Symmetry NRU HSE, RF Government grant, ag. \textnumero 14.641.31.0001 and by the Simons Foundation.

\paragraph{Notations and conventions.}\label{sect:notations}
\begin{enumerate}[leftmargin=0pt,itemindent=*]
\item We will freely use the language of higher categories, modeled e.g. by quasi-categories of \cite{LurHTT}. If not explicitly stated otherwise all categories are assumed to be $(\infty,1)$ and all (co-)limits are homotopy ones. The $(\infty, 1)$-category of Kan complexes will be denoted by $\Type$ and we will call it \emph{the category of spaces}. By $\Lan_i F$ and $\Ran_i F$ we will denote left and right Kan extensions of a functor $F$ along $i$ (see e.g. \cite[Definition 4.3.2.2]{LurHTT} for more details).

\item For a commutative ring $R$ by $\DMod{R}$ we will denote the canonical $(\infty, 1)$-enhancement of the triangulated unbounded derived category of the abelian category of $R$-modules $\UMod{R}$. All tensor product, pullback and pushforward functors are implicitly derived.

\item In this work by Artin stacks we always mean (higher) Artin stacks in the sense of \cite[Section 1.3.3]{TV_HAGII} or \cite[Chapter 2.4]{GaitsRozI}: these are sheaves in \'etale topology admitting a smooth $(n-1)$-representable atlas for some $n\ge 0$ (an inductively defined notion, see \textit{loc.cit.} for more details). We stress that we work with non-derived Artin stacks, i.e. they are defined on the category of \emph{ordinary} commutative rings. When we need to emphasize a precise dependence on $n$ (usually in inductive arguments) we say that $\mstack X$ is an $n$-Artin stack. We denote the $\infty$-category of $n$-Artin stacks over a base scheme $S$ by $\Stk^{n\mdef\Art}_{/S}$. We also freely use the notion of quasi-compact quasi-separated morphism between Artin stacks from \cite[Chapter 2, Section 4.1.9]{GaitsRozI}.

\item For a stack $\mstack X$ we will denote by $\QCoh(\mstack X)$ the category of \emph{quasi-coherent sheaves on $\mstack X$} defined as the limit $\lim_{\Spec A \to \mstack X} \DMod{A}$ over all affine schemes $\Spec A$ mapping to $\mstack X$ (see \cite[Chapter 3.1]{GaitsRozI} for more details). Note that $\QCoh(\mstack X)$ admits a natural $t$-structure such that $\mathcal F \in \QCoh(\mstack X)^{\le 0}$ if and only if $x^*(\mstack F) \in \DMod{A}^{\le 0}$ for any $A$-point $x\in \mstack X(A)$. Moreover, by \cite[Chapter 3, Corollary 1.5.7]{GaitsRozI} if $\mstack X$ is Artin stack, then $\QCoh(\mstack X)$ is left- and right-complete (i.e. Postnikov's and Whitehead's towers converge) and the truncation functors commute with filtered colimits.

\item For an affine group scheme $G$ over a ring $R$, given a representation $M$ (i.e. a comodule over the corresponding Hopf algebra $R[G]$) we denote by $\RG(G,M)\in \DMod{A}$ the \textit{rational cohomology} complex of $G$, namely the derived functor of $G$-invariants $M\mapsto M^G$. By flat descent, for $G$ flat over $R$, the abelian category $\mr{Rep}(G)\coloneqq \Rep_G(\Vect_F)^\heartsuit$ is identified with $\QCoh(BG)^\heartsuit$ and $\RG(G,M)\simeq\RG(BG,M)$.
\end{enumerate}

\section{Degeneration of the Hodge-to-de Rham spectral sequence}\label{sec:Degeneration of the Hodge-to-de Rham spectral sequence}
\subsection{Hodge and de Rham cohomology}\label{sec:Hodge and de Rham cohomology}
In this section we set up the Hodge-to-de Rham spectral sequence for $n$-Artin stacks and prove some technical results needed in subsequent sections of the paper. For the rest of this section fix a base ring $R$. We refer the reader to \cite{TV_HAGII} for an introduction to the theory of Artin stacks and cotangent complexes.
\begin{defn}[Hodge cohomology]
Let $\mstack X$ be an Artin stack over $R$. Define \emph{Hodge cohomology $\RG_\Hdg(\mstack X/R)$ of $\mstack X$} to be
$$\RG_\Hdg(\mstack X/R) \coloneqq \bigoplus_{p \ge 0} \RG\left(\mstack X, \wedge^p \mathbb L_{\mstack X/R}[-p]\right),$$
where $\mathbb L_{\mstack X/R}$ is the cotangent complex of $\mstack X$ over $R$ and $\wedge^p \mathbb L_{\mstack X/R}$ is its $p$-th derived exterior power (see \cite[Chapitre I.4]{Illusie_Cotangent} or \cite[Section 3]{MathewBrantner_LLambda}). For a fixed $n\in \mathbb Z$ we will also denote
$$H^n_\Hdg(\mstack X/R) \coloneqq H^n\RG_\Hdg(\mstack X/R) \simeq \bigoplus_{p+q=n} H^{p,q}(\mstack X/R), \quad\text{where}\quad H^{p,q}(\mstack X/R) \coloneqq H^q\left(\mstack X, \wedge^p \mathbb L_{\mstack X/R}\right).$$
\end{defn}
\begin{notation}
Let $S\coloneqq \Spec A$ be an affine smooth $R$-scheme. The algebraic de Rham complex of $S$ over $R$
$$\xymatrix{A \ar[r]^-d & \Omega^1_{A/R} \ar[r]^-d & \Omega^2_{A/R} \ar[r]^-d & \ldots}$$
will be denoted by $\Omega_{S/R, \dR}^\bullet$. We define $\RG_\dR(S/R) \coloneqq \Omega_{S/R, \dR}^\bullet \in \DMod{R}$.
\end{notation}
\begin{defn}[de Rham cohomology]\label{def: de Rham}
Let $\mstack X$ be a smooth quasi-compact quasi-separated Artin stack over $R$. Define the \emph{(Hodge-completed) de Rham cohomology $\RG_\dR(\mstack X/R)$ of $\mstack X$} to be
$$\RG_\dR(\mstack X/R) \coloneqq \lim_{S\in \Aff^{\sm, \op}_{/\mstack X}} \RG_\dR(S/R),$$
where $\Aff^{\sm}_{/\mstack X}$ is the full subcategory of stacks over $\mstack X$ consisting of affine $R$-schemes that are smooth over $\mstack X$. We will also denote $H^n\RG_\dR(\mstack X/R)$ by $H^n_\dR(\mstack X/R)$.
\end{defn}
In fact the Hodge cohomology complex admits a description similar to our definition of the de Rham cohomology:
\begin{prop}\label{Hodge_as_limit}
For any $p\in \mathbb Z_{\ge 0}$ the natural map
\begin{align}\label{formula:cotangent_descent}
\RG(\mstack X, \wedge^p \mathbb L_{\mstack X/R}) \to \lim_{S\in \Aff^{\sm, \op}_{/\mstack X}} \RG(S, \wedge^p \mathbb L_{S/R})
\end{align}
is an equivalence.

\begin{proof}
By \Cref{flat_descent_for_cotangent_compl} below the left hand side satisfies smooth descent. It follows that both sides of \eqref{formula:cotangent_descent} satisfy smooth descent. Since $n$-Artin stacks are by definition iterated smooth quotients of schemes, by induction on $n$ we reduce to the case when $\mstack X$ is a smooth affine scheme, where the assertion of the proposition is true, since $\Aff^{\sm}_{/\mstack X}$ has a final object given by $\mstack X$.
\end{proof}
\end{prop}
\begin{prop}[Flat descent for the cotangent complex] \label{flat_descent_for_cotangent_compl}
Let $p\colon \mstack U \to \mstack X$ be a surjective quasi-compact quasi-separated flat morphism between Artin stacks and denote by $p_\bullet \colon \mstack U_\bullet \to \mstack X$ the corresponding \v Cech nerve. Then the natural map
$$\wedge^d \mathbb L_{\mstack X/R} \xymatrix{\ar[r] &} \Tot p_{\bullet *}(\wedge^d\mathbb L_{\mstack U_\bullet/R})$$
is an equivalence for each $d\in \mathbb Z_{\ge 0}$.

\begin{proof}
The proof is essentially due to Bhatt (see \cite[Corollary 2.7, Remark 2.8]{Bhatt_derivedDeRham} or \cite[Section 3]{BMS2}). For every $n \in \mathbb Z_{\ge 0}$ we have a co-fiber sequence
$$\xymatrix{p_n^*\mathbb L_{\mstack X/R} \ar[r] & \mathbb L_{\mstack U_n/R}\ar[r] & \mathbb L_{\mstack U_n/\mstack X}}$$
in $\QCoh(\mstack U_n)$. It follows that $\wedge^d\mathbb L_{\mstack U_n/R}$ admits a $d$-step filtration with associated graded pieces $\wedge^ip_n^*\mathbb L_{\mstack X/R}\otimes \wedge^{d-i} \mathbb L_{\mstack U_n/\mstack X}$. Note that by flat descent for $\QCoh$ (see e.g. \cite[Corollary D.6.3.4]{Lur_SAG})
$$\Tot p_{\bullet*} \gr^0 \left(\wedge^d\mathbb L_{\mstack U_\bullet/R}\right) = \Tot p_{\bullet*} p_{\bullet}^* \wedge^d\mathbb L_{\mstack X/R} \simeq \wedge^d \mathbb L_{\mstack X/R}.$$
Hence it is enough to prove that
$$\Tot p_{\bullet*} \gr^i \left(\wedge^d\mathbb L_{\mstack U_\bullet/R}\right) = \Tot p_{\bullet*}( p_{\bullet}^*\wedge^i\mathbb L_{\mstack X/R} \otimes \wedge^{d-i}\mathbb L_{\mstack U_\bullet/\mstack X}) \simeq 0$$
for $i>0$. Moreover, since the map $p$ is faithfully flat, it is enough to show that the pullback
$$p^* \Tot p_{\bullet*} \gr^i \left(\wedge^d\mathbb L_{\mstack U_\bullet/R}\right) \simeq p^*\Tot(\wedge^i \mathbb L_{\mstack X/R} \otimes p_{\bullet*}\wedge^{d-i}\mathbb L_{\mstack U_\bullet/\mstack X})$$
of the totalization above is null-homotopic.

Note that by the qcqs assumption $p^*p_{n*} \simeq q_{n*} \mathrm{pr}_n^*$, where $q_\bullet \colon \mstack U_\bullet\times_{\mstack X} \mstack U \to \mstack U$ is the pullback of the \v Cech nerve $p_\bullet\colon \mstack U_\bullet \to \mstack X$ on $\mstack U$ along $p$ and $\mathrm{pr}_n\colon \mstack U_n \times_{\mstack X} \mstack U \to \mstack U_n$ are natural projections. It follows by base change for the relative cotangent complex that
$$p^* p_{\bullet *}\wedge^{d-i}\mathbb L_{\mstack U_\bullet/\mstack X} \simeq q_{\bullet *} \wedge^{d-i} \mathbb L_{\mstack U_\bullet\times_{\mstack X} \mstack U /\mstack U}.$$
But since $\mstack U_\bullet \times_{\mstack X} \mstack U \to \mstack U$ is a split simplicial object, the same holds for
$$p^*(\wedge^i \mathbb L_{\mstack X/R} \otimes p_{\bullet*}\wedge^{d-i}\mathbb L_{\mstack U_\bullet/\mstack X}) \simeq p^*(\wedge^i \mathbb L_{\mstack X/R}) \otimes q_{\bullet*}\wedge^{d-i} \mathbb L_{\mstack U_\bullet \times_{\mstack X} \mstack U/\mstack U},$$ 
since the class of split simplicial objects is stable under any functor. It follows $\Tot p^* p_{\bullet*} \gr^i \left(\wedge^d\mathbb L_{\mstack U_\bullet/R}\right) \simeq 0$. Finally, since by flat descent $\QCoh(\mstack X)$ is comonadic over $\QCoh(\mstack U)$, the pullback functor $p^*$ preserves totalizations of $p^*$-split cosimplicial objects, hence
$$p^* \Tot p_{\bullet*} \gr^i \left(\wedge^d\mathbb L_{\mstack U_\bullet/R}\right) \simeq \Tot p^* p_{\bullet*} \gr^i \left( \wedge^d\mathbb L_{\mstack U_\bullet/R}\right) \simeq 0.\qedhere$$
\end{proof}
\end{prop}

\begin{cor}
Let $\mstack X$ be a smooth quasi-compact quasi-separated Artin stack over $R$. Then
\begin{enumerate}
\item There exists a complete (decreasing) \emph{Hodge filtration $F^\bullet\RG_\dR(\mstack X/R)$} such that $\gr F^\bullet \RG_\dR(\mstack X/R) \simeq \RG_\Hdg(\mstack X/R)$.

\item There exists a strongly convergent spectral sequence $E_1^{p,q} = H^q(\mstack X, \wedge^p \mathbb L_{\mstack X/R}) \Rightarrow H^{p+q}_\dR(\mstack X/R).$
\end{enumerate}

\begin{proof}
Note that since $\mstack X$ is smooth, all schemes $S\in \Aff^\sm_{/\mstack X}$ are smooth. Since the Hodge filtration on $R\Gamma_\dR(S/R)$ is complete, the same holds for $\RG_\dR(\mstack X/R)$, since complete filtered complexes are closed under limits. Moreover, by construction we have
$$\gr F^\bullet \RG_\dR(\mstack X/R) \simeq \lim_{S\in \Aff^{\sm, \op}_{/\mstack X}} \gr F^\bullet \RG_\dR(S/R) \simeq \lim_{S\in \Aff^{\sm, \op}_{/\mstack X}} \RG_\Hdg(S/R) \simeq \RG_\Hdg(\mstack X/R),$$
where the last equivalence follows from the previous proposition. This filtration induces a spectral sequence with $E_1^{p,q}$ as stated. To prove it is strongly convergent, note that by smoothness of $\mstack X$, for each $n$ the induced filtration on $H^n_\dR(\mstack X/R)$ is finite.
\end{proof}
\end{cor}
The following simple observation will be quite useful in what follows:
\begin{rem}\label{connetedness_of_Hodge_to_deRham}
Let $\mstack X$ be a smooth Artin stack over $R$. Then the cotangent complex $\mathbb L_{\mstack X/R}$ (and its exterior powers) is concentrated in nonnegative cohomological degrees (with respect to the natural $t$-structure on $\QCoh(\mstack X)$). Since the global section functor $\RG$ is left $t$-exact, it follows that the natural map $\RG_\dR(\mstack X/R) \to \RG_\dR(\mstack X/R)/F^p\RG_\dR(\mstack X/R)$ induces an isomorphism on $H^{< p}$.
\end{rem}

Finally, we will need the following
\begin{prop}[Base-change]\label{Hodge_and_deRham_basechange}
Let $\mstack X$ be a smooth quasi-compact quasi-separated Artin stack over $R$ and let $R\to R^\prime$ be a ring homomorphism of finite Tor-amplitude. Then for $\mstack X^\prime \coloneqq \mstack X \otimes_R R^\prime$ the natural map $\RG_\dR(\mstack X / R) \otimes_R R^\prime \to \RG_\dR(\mstack X^\prime / R^\prime)$ is a filtered equivalence. In particular, for each $p\in \mathbb Z_{\ge 0}$ the natural map $\RG(\mstack X, \wedge^p \mathbb L_{\mstack X/R}) \otimes_R R^\prime \to \RG(\mstack X^\prime, \mathbb \wedge^p \mathbb L_{\mstack X^\prime/R^\prime})$ is an equivalence.

\begin{proof}
By the smoothness assumption on $\mstack X$ the fiber product $\mstack X\otimes_R R^\prime$ coincides with the derived fiber product. It follows by \cite[Lemma 1.4.1.16 (2)]{TV_HAGII} that $\mathbb L_{\mstack X/R}\otimes_R R^\prime \simeq \mathbb L_{\mstack X^\prime / R^\prime}$. By the base change for $\QCoh$ (see \cite[Chapter 3., Proposition 2.2.2 (b)]{GaitsRozI}) we deduce that the natural map $\RG_\Hdg(\mstack X / R) \otimes_R R^\prime \to \RG_\Hdg(\mstack X^\prime / R^\prime)$ is an equivalence.\

Next, note that the condition on the morphism $R\to R^\prime$ guarantees that the natural map $\RG_\dR(\mstack X/R)\otimes_R R^\prime \to \lim\limits_{\longleftarrow p} ((\RG_\dR(\mstack X/R)/F^p\RG_\dR(\mstack X/R))\otimes_R R^\prime)$ is an equivalence. Since both sides are complete with respect to the Hodge filtration, and, since by the above the induced map on the associated graded pieces
$$\RG_\Hdg(\mstack X/R) \otimes_R R^\prime \simeq \gr F^\bullet \RG_\dR(\mstack X / R) \otimes_R R^\prime \to \gr F^\bullet  \RG_\dR(\mstack X^\prime / R^\prime) \simeq \RG_\Hdg(\mstack X^\prime / R^\prime)$$
is an equivalence, we deduce that the base-change map for de Rham cohomology is an equivalence as well.
\end{proof}
\end{prop}

\subsection{Hodge-proper stacks}\label{sec:Hodge-proper stacks}
For the rest of this subsection fix a Noetherian base ring $R$. In this section we will introduce a reasonable (at least from the point of view of Hodge-to-de Rham degeneration) generalization of the notion of properness for stacks.
\begin{defn}\label{def:ncoh}
A complex of $R$-modules $M$ is called \emph{bounded below coherent}\footnote{In the previous version of this text we called such complexes \emph{almost coherent}. We decided to change the notation to avoid possible clashes with almost mathematics.} if it is cohomologically bounded below and for any $i\in \mathbb Z$ the cohomology module $H^i(M)$ is finitely generated over $R$. We will denote the full subcategory of $\DMod{R}$ consisting of bounded below coherent $R$-modules by $\Coh^+(R)$.
\end{defn}
\begin{rem}
We use the term \textit{nearly} coherent for objects of $\Coh^+(R)$ to distinguish them from coherent complexes, which in our convention are necessarily bounded (both from above and below).
\end{rem}
We have the following basic properties of $\Coh^+(R)$:
\begin{prop}\label{acoh_basics}
Let $R$ be a Noetherian ring. Then:
\begin{enumerate}
\item The category $\Coh^+(R)$ is closed under finite (co-)limits and retracts. In particular $\Coh^+(R)$ is a stable subcategory of $\DMod{R}$.

\item For each $n\in \mathbb Z$ the category $\Coh^{\ge n}(R) \coloneqq \Coh^+(R) \cap \DMod{R}^{\ge n}$ is closed under totalizations.
\end{enumerate}

\begin{proof}
\begin{enumerate}[wide]
\item This follows from the fact that for a Noetherian $R$ the abelian category of finitely generated $R$-modules is closed under (co)kernels, extensions and direct summands.

\item Let $M^\bullet$ be a co-simplicial object of $\Coh^{\ge n}(R)$. By shifting if necessary, we can assume that $n=0$. Since coconnective modules are closed under limits, $\Tot(M^\bullet) \in \DMod{R}^{\ge 0}$; hence it is enough to prove that $H^i\Tot(M^\bullet)$ is finitely generated $R$-module for all $i\in \mathbb Z_{\ge 0}$. Since all $M^i$ are coconnective, the natural map $\Tot(M^\bullet) \to \Tot^{\le k}(M^\bullet)$ induces an isomorphism on $H^{\le k}$. But since $\Tot^{\le k}$ is a finite limit, each $H^i\Tot^{\le k}(M^\bullet)$ is a finitely generated $R$-module.\qedhere
\end{enumerate}
\end{proof}
\end{prop}

\begin{rem}
Recall that the category of perfect $R$-modules $\DMod{R}^\perf$ is defined as the smallest full subcategory of $\DMod{R}$ containing $R$ and closed under finite (co-)limits and direct summands. Since $R\in \Coh^+(R)$ it follows from \Cref{acoh_basics}, that $\DMod{R}^\perf \subseteq \Coh^+(R)$.
\end{rem}

After this technical digression we are ready to introduce the notion of a Hodge-proper stack:
\begin{defn}[Hodge-proper stacks]
A smooth quasi-compact quasi-separated Artin stack $\mstack X$ over $R$ is called \emph{Hodge-proper} if for every $p\in \mathbb Z_{\ge 0}$ the complex $\RG(\mstack X, \wedge^p \mathbb L_{\mstack X/R})$ is bounded below coherent.
\end{defn}
\noindent For us the most important implication of Hodge-properness is that the de Rham cohomology is bounded below coherent:
\begin{prop}\label{de_rham_finiteness_of_Hodge_proper}
Let $\mstack X$ be a smooth Hodge-proper Artin stack over $R$. Then $\RG_\dR(\mstack X/R)$ is bounded below coherent complex of $R$-modules.

\begin{proof}
By smoothness $\RG_\dR(\mstack X/R)$ is bounded below by $0$, hence it is enough to prove that for each $n \in \mathbb Z_{\ge 0}$ the cohomology module $H^n_\dR(X/R)$ is finitely generated over $R$. By \Cref{connetedness_of_Hodge_to_deRham} the natural map $\RG_\dR(\mstack X/R) \to \RG_\dR(\mstack X/R)/F^{n+1}\RG_\dR(\mstack X/R)$ induces an isomorphism on $H^{\le n}$. We conclude, since $\RG_\dR(\mstack X/R)/F^{n+1}\RG_\dR(\mstack X/R)$, being a finite extension of bounded below coherent complexes $\RG(\mstack X, \wedge^i \mathbb L_{\mstack X/R}[-i])$, $0\le i\le n$, is bounded below coherent.
\end{proof}
\end{prop}


\subsection{Hodge-to-de Rham degeneration in positive characteristic}\label{sec:Proof of the degeneration}
Let $\mstack Y$ be a Hodge-proper Artin stack over a perfect field $k$ of characteristic $p$ admitting a smooth lift to the ring of the second Witt vectors $W_2(k)$. In this section we will prove that the Hodge-to-de Rham spectral sequence $H^j(\mstack Y, \wedge^i \mathbb L_{\mstack Y/k}) \Rightarrow H_\dR^{i+j}(\mstack Y / k)$ degenerates at the first page for $i+j < p$. Our strategy is to interpret both Hodge and de Rham cohomology in terms of crystalline cohomology and then, following Fontaine-Messing \cite{FontaineMessing_padciPeriods} (and Bhatt-Morrow-Scholze \cite{BMS2}), use (quasi-)syntomic descent for the crystalline cohomology to get a very functorial form of the Deligne-Illusie splitting. 

We denote by $\sigma\colon k\xra{x\mapsto x^p} k$ the absolute Frobenius morphism of $k$. We denote by the same letter $\sigma$ the induced automorphisms $W(k)\ra W(k)$ and $W_n(k)\ra W_n(k)$ for any $n\in \mathbb N$. For a $W(k)$-algebra (e.g. a $W_s(k)$-algebra for some $s$) $A$ we denote by $A^{(1)}\coloneqq A\otimes_{W(k),\sigma} W(k)$ its Frobenius twist and by $A^{(-1)}\coloneqq A\otimes_{W(k),\sigma^{-1}} W(k)$ its Frobenius untwist. For each $n\in \mathbb Z$ we have the relative Frobenius map $\phi_A\colon A^{(n)}\ra A^{(n-1)}$.

\begin{defn}[{\cite[Definition 4.10]{BMS2}}]
	\label{def:qsyn}
	A morphism $A\ra B$ of $W_n(k)$-algebras is called \emph{quasisyntomic} if it is flat and $\mathbb L_{B/A}$ has cohomological Tor amplitude $\left[-1,0 \right] $. A morphism $A\ra B$ is \emph{a quasisyntomic cover} if it is quasisyntomic and faithfully flat. We will denote by $\mr{QSyn}_{n}$ the site consisting of quasisyntomic $W_n(k)$-algebras with the topology generated by quasisyntomic covers.
\end{defn}
\begin{rem}
	It probably worth clarifying how our definition of quasisyntomic site compares to the one in \cite[Section 4]{BMS2}. Namely, $\mr{QSyn}_{n}$ is just the \textit{small} quasisyntomic site of $W_n(k)$ in the terminology of \cite{BMS2}. Indeed since all algebras in $\mr{QSyn}_{n}$ are killed by $p^n$, the notions of $p$-complete (faithful) flatness and quasisyntomicity coincide with the classical ones. The rest of the properties can be easily seen to agree as well.
\end{rem}
The notion of a quasisyntomic morphism is a generalization of more classical notion of a \emph{syntomic} morphism: a flat map $A\ra B$ that is locally a complete intersection in a smooth one. Syntomic morphisms include smooth morphisms, and, in the case $A$ is a regular $k$-algebra, the relative Frobenius morphism $\phi\colon A^{(1)}\ra A$. The advantage of quasisyntomic morphisms is that they also include some natural non-finite-type maps, most importantly the direct limit perfection $A\ra A_{\perf}\coloneqq{\indlim}{}_{\phi,n\ge 0} A^{(-n)} $ and its tensor powers $A\ra A_{\perf}\otimes_A\ldots \otimes_A A_{\perf}$ for a smooth $k$-algebra $A$. Using standard properties of the cotangent complex it is not hard to show that quasisyntomic morphisms are stable under composition and pushouts along arbitrary morphisms of algebras (and same for quasisyntomic covers). We refer to Section 4 of \cite{BMS2} for more details.

\smallskip
Recall that an $\mathbb F_p$-algebra $S$ is called semiperfect if $\phi\colon S\ra S$ is surjective.
\begin{defn}\label{def:qrsp}
	A $k$-algebra $S$ is called \emph{quasiregular semiperfect} if $S$ is quasisyntomic and the relative Frobenius homomorphism $\phi\colon S^{(1)}\ra S$ is surjective. We call a $W_n(k)$-algebra $\widetilde S$ \emph{quasiregular semiperfectoid} if it is flat over $W_n(k)$ and $\widetilde S/p$ is quasiregular semiperfect. We will denote by $\mr{QRSPerf}_{n}$ the site consisting of quasiregular semiperfectoid $W_n(k)$-algebras with the topology generated by quasisyntomic covers.
\end{defn}
\begin{rem}
	Note that if $n>1$ our definition of a quasiregular semiperfectoid algebra over $W_n(k)$ does not agree with \cite[Definition 4.10]{BMS2} since $W_n(k)$ itself is not semiperfectoid. Nevertheless, since we assume that all our objects are flat over $W_n(k)$, all the necessary arguments go through essentially without any change by reducing modulo $p$.
\end{rem}
For any $k$-algebra $S$, $H^0(\mathbb L_{S/k})$ is identified with the K\"ahler differentials $\Omega^1_{S/k}$. Since $d(x^p)=0$, we get that $H^0(\mathbb L_{S/k})=0$ for $S$ semiperfect, and that $\mathbb L_{S/k}$ is concentrated in a single cohomological degree $-1$ for $S$ quasiregular semiperfect. The same is true for $\mathbb L_{\widetilde S/W_n(k)}$ for a quasiregular semiperfectoid $W_n(k)$-algebra $\widetilde S$. Moreover, any flat map $\widetilde S_1\ra \widetilde S_2$ between quasiregular semiperfectoids over $W_n(k)$ is quasisyntomic. This gives a map of sites $\mr{QRSPerf}_{n}\ra \mathrm{QSyn}_n$.

\smallskip

In fact quasiregular semiperfectoid algebras form a basis of topology in $\mathrm{QSyn}_n$. This leads to an equivalence between the corresponding categories of sheaves:
\begin{prop}\label{prop:syntomic_vs_perfect}
	The restriction along the natural embedding $u\colon\mr{QRSPerf}_{n}\ra \mr{QSyn}_{n}$ induces an equivalence
	$$
	\Sh(\mr{QSyn}_{n},\fcat C) \xymatrix{\ar[r]_\sim^-{u^{-1}} & }\Sh(\mr{QRSPerf}_{n}, \fcat C)
	$$
	of the categories of sheaves with values in any presentable $\infty$-category $\fcat C$.
	
	\begin{proof}
		Following the proof of \cite[Proposition 4.31]{BMS2} it is enough to show, that first, any quasisyntomic algebra $A$ has a quasisyntomic cover $A\ra S$ by a semiperfectoid, and second, that all terms $S^{\otimes_A i}$ in the corresponding \v Cech object are automatically semiperfectoid. The cover $S$ can be constructed as follows: we take the surjection $W_n(k)[x_a]_{a\in A}\surj A$ from the free polynomial algebra on $A$ and put $S\coloneqq A\otimes_{W_n(k)[x_a]}W_n(k)[x_a^{1/p^\infty}]$. The map $W_n(k)[x_a]\ra W_n(k)[x_a^{1/p^\infty}]$ is quasisyntomic and faithfully flat, thus so is $A\ra S$. Also $W_n(k)[x_a^{1/p^\infty}]\ra S$ is a surjection, $k[x_a^{1/p^\infty}]$ is perfect, thus $S/p$ is semiperfect. We get that $S\in \mr{QRSPerf}_{n}$. The statement about $S^{\otimes_A i}$ then follows from the analogous one modulo $p$ (see e.g. \cite[Lemma 4.30]{BMS2}).
	\end{proof}
\end{prop}
\begin{rem}
	For a sheaf $\mathcal F$ on $\mr{QRSPerf}_{n}$ we will denote its image under the inverse equivalence in \Cref{prop:syntomic_vs_perfect} by $\mathcal F$ as well.
\end{rem}

\begin{ex}\label{ex:smooth k-algebra}
	Let $B$ be a smooth algebra over $W_n(k)$. By smoothness, Zariski-locally on $\Spec B$, there exists an \'etale map $P\ra B$ from the polynomial algebra $P=P_d\coloneqq W_n(k)[x_1,\ldots, x_d]$ for some $d$. Let $P_{\perf}=W_n(k)[x_1^{1/p^\infty},\ldots,x_d^{1/p^\infty}]$ and let $B_{\perf}\coloneqq B \otimes_P P_{\perf}$; it is a quasiregular semiperfectoid $W_n(k)$-algebra\footnote{In fact it is even \emph{quasismooth perfectoid}, since $\mathbb L_{B/W_n(k)} \simeq 0$ and the relative Frobenius for $B_{\mathrm{perf}}/p$ is an isomorphism.} and the natural map $B\ra B_{\perf}$ is a quasisyntomic cover. Moreover all terms ${(B_{\perf}\otimes_B\ldots \otimes_B B_{\perf})}_n$ in the corresponding Cech object are also quasiregular semiperfectoids. Given any sheaf $\mc F$ on $\mr{QRSPerf}_{n}$ its value on $B\in \mr{QSyn}_{n}$ (via \Cref{prop:syntomic_vs_perfect}) can be computed as "the unfolding": $$
	\RG_{\mr{QSyn}_n}(B,\mc F)\xra{\sim}\Tot\left( \xymatrix{\mc F(B_{\perf}) \ar@<-.45ex>[r] \ar@<.45ex>[r]& \mc F(B_{\perf}\otimes_B B_{\perf}) \ar@<+.9ex>[r] \ar[r] \ar@<-.9ex>[r] &  \mc F(B_{\perf}\otimes_B B_{\perf}\otimes_B B_{\perf}) \ar@<+1.35ex>[r] \ar@<+.45ex>[r] \ar@<-.45ex>[r] \ar@<-1.35ex>[r]& \cdots}\right).
	$$
\end{ex}

For a ring $R$ let $\Poly_{R}\subset \CAlg_{R/}$ denote the full subcategory of finitely generated \emph{polynomial $R$-algebras}. Recall that one of the ways to define the cotangent complex $\mathbb L_{A/R}$ for an $R$-algebra $A$ is to consider the left Kan extension of the functor $B\mapsto \Omega^1_{B/R}$ from the category of polynomial $R$-algebras, namely
$$
\mathbb L_{A/R}\simeq \colim_{\Poly_{R}/A}\Omega^1_{B/R}.
$$

One can extend the de Rham and crystalline cohomology functors in a similar way:
\begin{construction}\label{derived_crystalline}
	Let $k$ be a perfect field.
	\begin{itemize}
		\item The \emph{derived de Rham cohomology} functor
		$$\RG_{\mathbb L\dR}(-/W_n(k))\colon \CAlg_{W_n(k)/} \to \DMod{W_n(k)}  $$
		is defined as the left Kan extension of the functor $B\mapsto \Omega^\bullet_{B/W_n(k),\dR}$ on $\Poly_{W_n(k)}$.
		
		\item The \emph{derived crystalline cohomology} functor
		$$\RG_{\mathbb L\crys}(-/W(k))\colon \CAlg_{W_n(k)/} \to \DMod{W(k)}  $$
		is defined as the (derived) $p$-adic completion of the left Kan extension of the functor $B\mapsto \RG_{\crys}((B/p)/ W(k))$ on $\Poly_{W_n(k)}$.
	\end{itemize}
\end{construction}
\begin{rem}
	For a more thorough treatment of the derived de Rham and crystalline cohomology functors we refer the reader to \cite{Illusie_CotangentII} and \cite{Bhatt_derivedDeRham} where these notions were originally considered and applied.
\end{rem}
\begin{rem}\label{rem:de Rham vs crys}
For any $W_n(k)$-algebra $B$ the complexes $\RG_{\mathbb L\crys}(B/W(k))\otimes_{W(k)} W_n(k)$ and $\RG_{\mathbb L\dR}(B/W_n(k))$ are naturally equivalent. Indeed, by construction both functors commute with geometric realizations, hence it is enough to prove the statement for $B$ being a smooth $W_n(k)$-algebra. In this case this is a basic result in the crystalline cohomology theory, see e.g. \cite[Corollary 7.4]{BarthelotOgus}.
\end{rem}

Similarly, we can extend the functor $B\mapsto \tau^{\le m}\Omega^{\bullet}_{B/W_n(k),\dR}$ to get a filtered object $\left(\RG_{\mathbb L\dR}(-/W_n(k)\right),F_{\mr{H}}^*)$ (Hodge filtration). If $n=1$ the functor $B\mapsto \Omega^{\le m}_{B/W_n(k),\dR}$ extends to $(\RG_{\mathbb L\dR}(-/k),\Fil_*^\cnj)$ (conjugate filtration). The conjugate filtration on the derived de Rham cohomology is exhaustive since it is exhaustive on de Rham cohomology of polynomial algebras and since colimits commute. For $\RG_{\mathbb L\dR}(-/k)$ the Cartier isomorphism identifies the corresponding associated graded with $\bigoplus_{r\ge 0}\wedge^r\mathbb L_{B^{(1)}/k}[-r]$. Next lemma shows that the derived de Rham cohomology on the quasi-syntomic site satisfies flat descent:
\begin{lem}
Let $A \to B$ be a faithfully flat homomorphism of $k$-algebras and let $B^\bullet$ be the corresponding \v Cech co-simplicial object. Assume that $\mathbb L_{B/k}$ and $\mathbb L_{B/A}$ have cohomological $\Tor$-amplitude $[-1;0]$. Then the natural map
$$R\Gamma_{\mathbb L\dR}(A/k) \to \Tot R\Gamma_{\mathbb L\dR}(B^\bullet/k)$$
is an equivalence.

\begin{proof}
Note that since $A\to B$ is faithfully flat the base change for cotangent complex and transitivity triangles imply that $\mathbb L_{B^i/k}$ has Tor-amplitude in degrees $[-1;0]$ for all $i\ge 0$, where $B^i$ is the $i$-th term in the corresponding cosimplicial \v Cech object. Consequently, $\wedge^n\mathbb L_{B^{i(1)}/k}[-n]$ is 0-coconnective for any $n$ and $i$. It then follows by flat descent for cotangent complex (see \cite[Theorem 3.1]{BMS2}) that the natural map
$$F^n_{\mathrm{conj}} R\Gamma_{\mathbb L\dR}(A/k) \to \Tot R\Gamma_{\mathbb L\dR}(B^\bullet/k)$$
is $n$-coconnective and hence induces an equivalence after passing to the colimit by $n$ on the left hand side.
\end{proof}
\end{lem}
\begin{cor}
For every $n\ge 1$ the presheaves on $\mr{QSyn}_{n}$
$$A \mapsto R\Gamma_{\mathbb L\dR}(A/W_n(k)) \quad\text{and}\quad A \mapsto R\Gamma_{\mathbb L \crys}((A/p)/W(k))$$
are sheaves.

\begin{proof}
By \Cref{rem:de Rham vs crys} since $W_n(k)$ is of finite $\Tor$-amplitude over $W(k)$ it is enough to prove the assertion for $R\Gamma_{\mathbb L\crys}((-/p)/W(k))$. But since $R\Gamma_{\mathbb L\crys}((-/p)/W(k))$ is derived $p$-complete by construction and since $k$ is a perfect $W(k)$-module (and thus $-\otimes_{W(k)}k$ commutes with limits) it is enough to prove that $R\Gamma_{\mathbb L\crys}((-/p)/W(k))\otimes_{W(k)} k \simeq R\Gamma_{\mathbb L\dR}(-/k)$ is a sheaf. This is a content of the previous lemma.
\end{proof}
\end{cor}

\begin{rem}\label{derived de rham smooth}
	Note that if $R$ were a $\mathbb Q$-algebra, the derived de Rham cohomology would be just equal to $R$ \cite[Remark 2.6]{Bhatt_pAdicDerivedDeRham}. Indeed, by the $\mathbb A^1$-invariance of the de Rham cohomology in characteristic $0$, the de Rham cohomology functor restricted to $\Poly_R$ is constant with value $R$, hence so is its left Kan extension. 
	In particular, the derived de Rham cohomology of a smooth $R$-algebra is usually not equivalent to the classical de Rham cohomology. To improve the situation one usually works with the Hodge-completed version of the derived de Rham cohomology instead. 
	
	However, in positive characteristic the non-completed derived de Rham cohomology is much better behaved. In particular it coincides with the classical de Rham cohomology for smooth $W_n(k)$-algebras. Here, the key observation, which is due to Bhatt (see \cite{Bhatt_pAdicDerivedDeRham}), is to use the conjugate filtration. Namely, to show that the natural morphism $R\Gamma_{\mathbb L\dR}(B) \to \RG_\dR(B)$  is an equivalence for a smooth $W_n(k)$-algebra $B$, it is enough to show this modulo $p$, and then (since both sides are complete), that the induced map on the associated graded of the conjugate filtration is an equivalence. For $B$ smooth, this reduces to natural isomorphisms $\Lambda^i \mathbb L_{B^{(1)}_k/k} \areq \Omega^i_{B^{(1)}_k/k}$.
\end{rem}
\begin{rem}\label{rem:F_p vs k}
	Since the absolute Frobenius $\sigma\colon k\ra k$ is an automorphism, the cotangent complex $\mathbb L_{k/\mathbb F_p}$ (and all its wedge powers) vanishes. It follows that $\RG_{\dR}(k/\mathbb F_p) \simeq k$. Given any $k$-algebra $B$ we have a natural morphism of $E_\infty$-algebras $\RG_{\dR}(k/\mathbb F_p)\ra \RG_{\dR}(B/\mathbb F_p)$. This endows $\RG_{\dR}(B/\mathbb F_p)$ with a natural $k$-linear structure. Similarly, for any $k$-algebra $A$ the complex $\RG_{\mathbb L\crys}(A/\mathbb Z_p)$ has a natural $W(k)$-linear structure. Moreover, the natural morphism 
	\begin{align}\label{eq:crys_over_Zpn_vs_Wn}
	\RG_{\mathbb L\crys}(B/\mathbb Z_p) \to \RG_{\mathbb L\crys}(B/W(k))
	\end{align}
	is $W(k)$-linear. We claim that \eqref{eq:crys_over_Zpn_vs_Wn} is an equivalence. Since both sides are $p$-adically complete it is enough to show that it is an equivalence mod $p$, where we get an analogous map, but for the derived de Rham cohomology of the reduction $B/p$. On the associated graded of the conjugate filtration $\Fil^\cnj_*$ the induced map is an equivalence, since in the transitivity triangle
	$$\mathbb  L_{k/\mathbb F_p}\otimes_{k} B \to \mathbb L_{B/\mathbb F_p} \to \mathbb L_{B/k}$$
	the term $\mathbb  L_{k/\mathbb F_p}$ is equivalent to $0$. Thus \eqref{eq:crys_over_Zpn_vs_Wn} is an equivalence.
\end{rem}

Recall that the cotangent complex $\mathbb L_{\widetilde S/W_n(k)}$ of $\widetilde S\in \mr{QRSPerf}_{n}$ has cohomological $\Tor$-amplitude consentrated in $-1$, thus $\bigoplus_{r}\wedge^r\mathbb L_{\widetilde S/W_n(k)}[-r]$ is supported in cohomological degree 0. The same holds for $\RG_{\mathbb L\dR}(\widetilde S/W_n(k))$; in other words, it is a classical commutative ring. It has a description in terms of one of the Fontaine's period rings $A_{\crys}$:
\begin{construction}\label{constr:Acrys}
	Let $S$ be a quasiregular semiperfect $k$-algebra and let $S^\flat$ be the inverse limit perfection $S^\flat\coloneqq{\prolim}{}_{\phi,n\ge 0}\  S^{(n)}$. We have a natural map $S^\flat \to S$ which is surjective. The ring $\mathbb A_\crys(S)$ is defined as the $p$-adic completion of the divided power envelope of the kernel of the natural composite surjection $\theta_{1,S}\colon W(S^\flat) \surj S^\flat \surj S$ (where the divided power structure agrees with the one on the ideal $(p) \subset W(k)$). Note that $\mathbb A_\crys(S)/p$ is identified with the $PD$-completion $D^{PD}_I(S^\flat)$ along the ideal $I\subset S^\flat$ defined as the kernel of the natural map $S^\flat \surj S$.
	
	Theorem 8.14(3) of \cite{BMS2} (together with \Cref{rem:F_p vs k}) identifies $\RG_{\mathbb L\crys}(S/ W(k))$ with $\mathbb A_{\crys}(S)$. The ring $\mathbb A_\crys(S)$ comes with a natural ring morphism $\phi\colon \mathbb A_\crys(S)^{(1)} \to \mathbb A_\crys(S)$ induced by the relative Frobenius $\phi\colon S^{(1)}\ra S$. It is identified with the natural Frobenius $\phi\colon\RG_{\mathbb L\crys}(S/W(k))^{(1)}\ra \RG_{\mathbb L\crys}(S/W(k))$ on the crystalline cohomology. 
	For each $n$ we define a presheaf of rings $\mathbb A_{\crys}$ on $\mr{QRSPerf}_{n}$ by sending $\widetilde S\in \mr{QRSPerf}_{n}$ to $\mathbb A_{\crys}(\widetilde S/p)$. By the above identification it is in fact a sheaf. Note that by the universal property of the PD-envelope there is a natural map\footnote{Here we endow $(p)\subset \widetilde S$ with the standard PD-structure, given by  $p^{[k]}:=p^k/k!$.} $\theta_{n,\widetilde S}:\mathbb A_{\crys}(\widetilde S/p)\ra \widetilde S$ which factors through $\mathbb A_{\crys}(\widetilde S/p)/p^n$.
\end{construction}

The following two filtrations on $\mathbb A_{\crys}/p$ correspond to the Hodge and the conjugate filtrations:  
\begin{defn}\label{constr:HodgeFil}
	Let $S$ be a quasiregular semiperfect $k$-algebra and let $I$ be the ideal of the natural projection $S^\flat \surj S$. The descending \emph{Hodge filtration} on $\mathbb A_{\crys}(S)/p\simeq D^{PD}_I(S^\flat)$ is defined as the filtration by the divided powers of $I$: $\mathbb A_{\crys}(S)/p\simeq I^0 \supset I^{[1]} = I \supset I^{[2]} \supset I^{[3]}  \supset \cdots$. This filtration is functorial in $S$ and thus defines a filtration by presheaves $\mathbb I^0 \supset \mathbb I^{[1]} = \mathbb I \supset \mathbb I^{[2]} \supset \mathbb I^{[3]}  \supset \cdots$ on the sheaf $\mathbb A_{\crys}/(p)$ on $\mr{QRSPerf}_{n}$ for any $n$. Via Proposition 8.12 of \cite{BMS2} it is identified with the Hodge filtration on $\RG_{\mathbb L\dR}(S/k)\simeq \mathbb A_{\crys}(S)/p$ and thus is in fact a filtration by sheaves.
\end{defn}
\begin{defn}\label{constr:ConjFil}
	The ascending \emph{conjugate filtration} $\mr{Fil}_*^{\cnj}$ on $\mathbb A_{\crys}(S)/p\simeq D^{PD}_{S^\flat}(I)$ is defined by taking $F_r^{\cnj}$ to be the $S^\flat $-submodule generated by the elements of the form $s_1^{[l_1]} s_2^{[l_2]} \ldots s_m^{[l_m]}$ with $s_i\in I$ and $\sum_{i=1}^m l_i < (r+1)p$. This construction is functorial in $S$ and determines an (ascending) filtration  $\mr{Fil}_{*}^\cnj$ on the sheaf $\mathbb A_{\crys}/p$ on $\mr{QRSPerf}_{n}$ for any $n$. By Proposition 8.12 of \cite{BMS2} it is identified with the conjugate filtration on $\RG_{\mathbb L\dR}(S/k)$ and thus is also a filtration by sheaves. Note that both filtrations are multiplicative and the conjugate filtration is exhaustive.
\end{defn}

The following is an analogue of the inverse Cartier isomorphism (see \Cref{Cartier_iso}) between $(\mathbb A_{\crys}/p, \mathbb I^{[*]})$ and $(\mathbb A_{\crys}/p, \Fil_{*}^\cnj)$:
\begin{prop}[\cite{BMS2}, Propositions 8.11 and 8.12] \label{prop:Cartier on A_crys}
	Let $S$ be a semiperfect $k$-algebra.
	There is a well-defined surjective homomorphism of $W(k)$-algebras $\kappa_*\colon\Gamma^*_{S}(I/I^2)^{(1)}\ra \gr_*^\cnj(\mathbb A_{\crys}(S)/p) $\footnote{Where $\Gamma^*$ denotes the free commutative divided power algebra.}. If $S$ is quasiregular, $\kappa_*$ is an isomorphism.
	
	\begin{proof}
		The map is defined as follows: for $s_i\in I$ 
		$$
		\kappa_{k_1+\cdots k_m}:s_1^{[k_1]}\cdots s_m^{[k_m]} \mapsto \prod_{i=1}^m \left(\frac{(pk_i)!}{p^{k_i}k_i!}\right)s_1^{[pk_1]}\cdots s_m^{[pk_m]}. 
		$$
		We have $(s_1s_2)^{[pk]}=p!(s_1^k)^{[p]}s_2^{[pk]}=0$ and $(s_1s_2)^{[l]}\in \Fil_0^\cnj$ for any $l<p$. This shows that for $s\in I^2$, $s^{[l]}\in\Fil_0^\cnj$ for all $l$ and so the map is well-defined.
		Elements $\{s_1^{[pk_1]}\cdots s_m^{[pk_m]}\}_{k_1+\cdots k_m<r+1}$ in fact generate $\Fil_r^\cnj$ over $S^\flat$. Since the integer $\prod_{i=1}^m \left(\frac{(pk_i)!}{p^{k_i}k_i!}\right)$ is a $p$-adic unit the map $\kappa_*$ is surjective. The fact that $\kappa_*$ is an isomorphism for $S$ quasiregular semiperfect is a part of Proposition 8.12 of \cite{BMS2}.
	\end{proof}
\end{prop}
\begin{rem}
	 In particular we get an isomorphism $\kappa_*\colon \Gamma^*_{S}(\mathbb I/\mathbb I^2)^{(1)}\xra{\sim} \gr_*^\cnj (\mathbb A_{\crys}/p) $ of sheaves of algebras on $\mathrm{QRSPerf_n}$.
\end{rem}

Now we descend everything back to the quasisyntomic site $\mr{QSyn}_n$. We record what the sheaves defined above give when computed on a smooth $W_n(k)$-algebra $B$.

\begin{prop}\label{prop:A_crys/p_on_smooth schemes}
	Let $B$ be a smooth $W_n(k)$-algebra considered as an object of $\mr{QSyn}_{n}$.
	Then:
	\begin{enumerate}
		\item For any $0\le s\le n$ there is a natural equivalence of $E_\infty$-algebras $\RG_{\mathrm{QSyn}_n}(B, \mathbb A_{\crys}/p^s)\simeq \Omega_{(B/p^s)/ W_s(k),\dR}^\bullet$.
		
		\smallskip 
		
		\item For any $r\in \mathbb Z_{\ge 0}$ there is a natural equivalence $\RG_{\mathrm{QSyn}_n}(B, \mathbb I^{[r]})\simeq \Omega_{(B/p)/k,\dR}^{\ge r}$,
		where $\Omega_{B/p,\dR/k}^{\ge r}$ is the $r$-th term of the Hodge filtration. 
		\smallskip 
		
		\noindent In particular, $\RG_{\mathrm{QSyn}_n}(B, \mathbb I^{[r]}/\mathbb I^{[r+1]})\simeq \Omega^r_{(B/p)/k}[-r]$.
		
		\item For any $r\in \mathbb Z_{\ge 0}$ there is a natural equivalence $\RG_{\mathrm{QSyn}_n}(B, \Fil_{r}^\cnj) \simeq \tau^{\le r} \Omega^\bullet_{(B/p)/k, \dR}$.
		
		\item The natural map $\Gamma^r_{S}(\mathbb I/\mathbb I^2)\ra \mathbb I^{[r]}/\mathbb I^{[r+1]}$ given by multiplication induces an equivalence
		$$
		\RG_{\mr{QSyn}_n}(B,\Gamma^r_{S}(\mathbb I/\mathbb I^2))\simeq\RG_{\mr{QSyn}_n}(B,\mathbb I^{[r]}/\mathbb I^{[r+1]})
		$$ 
		for any $r\ge 0$.
		
		\item The isomorphism $\kappa_*\colon \Gamma^*_{S}(\mathbb I/\mathbb I^2)^{(1)}\xra{\sim} \gr_*^\cnj (\mathbb A_{\crys}/p) $ from \Cref{prop:Cartier on A_crys} induces the inverse Cartier isomorphism 
		$$
		\bigoplus_{r=0}^\infty \Omega^r_{(B^{(1)}/p)/k} \xymatrix{\ar[r]^-{C^{-1}}_\sim & }   \bigoplus_{r=0}^\infty H^{r}\left(\Omega_{(B/p)/k,\dR}^\bullet\right)
		$$ via the above equivalences.
	\end{enumerate}
	
	\begin{proof}
		
		Parts $1, 2, 3$  follow from Proposition 8.12 and Theorem 8.14(3) of \cite{BMS2} and flat descent for the derived crystalline cohomology (using Remarks \ref{derived de rham smooth}, \ref{rem:de Rham vs crys} and \ref{rem:F_p vs k}).
		
		\smallskip 
		
		4. We use the notations of \Cref{ex:smooth k-algebra}. We have
		$$
		\RG_{\mr{QSyn}_n}(B,\mc F)\xra{\sim}\Tot\left( \xymatrix{\mc F(B_{\perf}) \ar@<-.45ex>[r] \ar@<.45ex>[r]& \mc F(B_{\perf}\otimes_B B_{\perf}) \ar@<+.9ex>[r] \ar[r] \ar@<-.9ex>[r] &  \mc F(B_{\perf}\otimes_B B_{\perf}\otimes_B B_{\perf}) \ar@<+1.35ex>[r] \ar@<+.45ex>[r] \ar@<-.45ex>[r] \ar@<-1.35ex>[r]& \cdots}\right)
		$$
		for any quasisyntomic sheaf $\mc F$. Moreover all terms $(B_{\perf}\otimes_B\ldots \otimes_B B_{\perf})_n$ are in fact regular semiperfect, meaning $I\subset S^\flat$ is generated by a regular sequence. Thus for them $\Gamma^r_{S}(I/I^2)\xra{\sim} I^{[r]}/I^{[r+1]}$ and so $\RG_{\mr{QSyn}_n}(B,\Gamma^r_{S}(\mathbb I/\mathbb I^2))\simeq \RG_{\mr{QSyn}_n}(B,\mathbb I^{[r]}/\mathbb I^{[r+1]})$.
		
		\smallskip 
		
		5. Note that the map depends only on the reduction of $B$ modulo $p$, thus it is enough to consider the case $B\in \mr{QSyn}_1$ is of characteristic $p$. The inverse Cartier isomorphism $C^{-1}$ is uniquely defined by the property that it is multiplicative, $C^{-1}(f)=f^p$ and $C^{-1}(df)=f^{p-1}df$ for any $f\in B$. The map $\kappa_*$ is multiplicative, $\kappa_0$ is by definition given by Frobenius, so it remains to check the third assertion. By functoriality (considering the map $k[x]\xra{x\mapsto f} B$) it is enough to check this in the case $B=k[x]$ and $f=x$. While originally we had the proof using the relation between the Cartier isomorphism and the Bockstein operator, we will present a different proof that was kindly suggested to us by one of the referees.
		
		One can use the explicit formula for the crystalline cohomology via the \v Cech-Alexander complex. Namely, for any smooth $k$-algebra $B$ the homotopy groups of the cosimplicial algebra $C_\crys^\bullet(B)\coloneqq D_{\Ker(B^{\otimes_k^\bullet}\xra{m} B)} B^{\otimes_k^\bullet}$ compute the de Rham cohomology of $B$; here $B^{\otimes_k^\bullet}\xra{m} B$ is the multiplication map and $D_{\Ker(B^{\otimes_k^\bullet}\xra{m} B)} B^{\otimes_k^\bullet}$ is the PD-envelope corresponding to its kernel. 
		By \cite[The proof of Theorem 2.12]{BhattdeJong} the totalization of the bicomplex
		$$
		\xymatrix{B\ar[r]\ar[d]& \Omega^1_B\ar[r]\ar[d]^{d_{\text{\v Cech}}}& \Omega^2_B\ar[r]\ar[d]& \ldots\\
			D(2)\ar[r]^(.39){d_\dR}\ar[d]& \Omega^1_{D(2)} \ar[r]\ar[d]& \Omega^2_{D(2)}\ar[r]\ar[d]& \ldots\\ D(3)\ar[r]\ar[d]& \Omega^1_{D(3)} \ar[r]\ar[d]& \Omega^2_{D(3)}\ar[r]\ar[d]& \ldots\\
			\ldots& \ldots&\ldots&
		}
		$$
		with $D(i)\coloneqq D_{\Ker(B^{\otimes_k^i}\xra{m} B)} B^{\otimes_k^i}$ is quasiisomorphic to both the first row and column (via the embeddings of the latter) this way establishing the comparison quasiisomorphism of $C_\crys^\bullet(B)$ and $\Omega_{B,\dR}^\bullet$. For $B=k[x]$ we have $D(2)=D_{(x_1-x_2)}k[x_1,x_2]$, $d_{\text{\v Cech}}(x^{p-1}dx)=x_1^{p-1}dx_1 -x_2^{p-1}dx_2$ and we leave it to the reader to check that this also equals to $d_{\dR}(a)$ with $a\coloneqq (p-1)!((x_1 -x_2)^{[p]}+\sum_{i=1}^{p-1}(-1)^i x_1^{[p-1-i]}x_2^{[i]})\in D(2)$. Thus under this comparison the class $[x^{p-1}dx]\in H^1_\dR(k[x])$ goes to $[a]$ in $H^1(C_\crys^\bullet(B))$. Note also that by an analogous but simpler computation $[dx]\in H^1_\dR(k[x])$ goes to $[x_1-x_2]\in H^1(C_\crys^\bullet(B))$.

		As we saw in part 4, the cosimplicial algebra $(\mathbb A_{\crys}/p)(B_{\perf}^{\otimes_B \bullet})\simeq D_{\Ker(B_{\perf}^{\otimes_k^\bullet}\xra{} B_\perf^{\otimes_B^\bullet})}B_{\perf}^{\otimes_k^\bullet}$ appearing as the \v Cech object associated to the quasisyntomic cover $B\ra B_\perf$ also computes the de Rham cohomology of $B$. We have a natural map of cosimplicial algebras $C_\crys^\bullet(B)\ra \mathbb (\mathbb A_{\crys}/p)(B_{\perf}^{\otimes_B \bullet})$ induced termwise by $B^{\otimes_k^\bullet}\ra B_{\perf}^{\otimes_k^\bullet}$. The definitions (\ref{constr:HodgeFil} and \ref{constr:ConjFil}) of the conjugate and Hodge filtrations make sense for any PD-envelope and so extend to the cosimplicial algebra $C_\crys^\bullet(B)$ as well. The map $\kappa$ (see the proof of \ref{prop:Cartier on A_crys}) extends naturally as well, the map $C_\crys^\bullet(B)\ra \mathbb (\mathbb A_{\crys}/p)(B_{\perf}^{\otimes_B \bullet})$ preserves the filtrations, commutes with $\kappa$ and in fact is a filtered quasi-isomorphism\footnote{Indeed, both complexes can be considered as \v Cech-Alexander complexes for a slightly unusual "quasisyntomic" version of the char $p$ crystalline site:  namely, we consider triples $(U,T,\delta)$ with $U\ra \Spec B$ a quasisyntomic morphism and $T$ being a char $p$ PD-thickening of $U$. The complexes $C_\crys^\bullet(B)$ and $\mathbb (\mathbb A_{\crys}/p)(B_{\perf}^{\otimes_B \bullet})$ can be interpreted as the \v Cech-Alexander complexes corresponding to two different covers, namely $\Spec B\ra *$ and $\Spec B_\perf\ra *$. The map above is then induced by the map of coverings $\Spec B_\perf\ra \Spec B$ and being a map \v Cech-Alexander complexes (for the structure sheaf) is automatically a quasiisomorphism. From this interpretation it is also clear that it respects the Hodge and conjugate filtrations, as well as the map $\kappa$.}.  Returning to the case $B=k[x]$ we see that, first, $dx\in \Omega^1_{\mathbb A^1_k}\simeq  H^1_{\mr{QSyn}}(B,\mathbb I/\mathbb I^2)$ corresponds to $x_1-x_2\in D(2)\subset D_{\Ker(B_\perf^{\otimes_k^{2}} \ra B_\perf^{\otimes_B^{ 2}})}B_\perf^{\otimes_k^{2}}$ under the comparison, and, second, that the class $[\kappa_1(dx)]$ in $ H^1_{\mr{QSyn}}(B,\gr_1^\cnj)$ is given by the class of the element $(p-1)!(x_1-x_2)^{[p]}\in D(2)\subset D_{\Ker(B_\perf^{\otimes_k^{2}} \ra B_\perf^{\otimes_B^{ 2}})}B_\perf^{\otimes_k^{2}}$ modulo $\Fil_{0}^\cnj$. It remains to note that this element differs from $a$ by $\sum_{i=1}^{p-1}(-1)^i x_1^{[p-1-i]}x_2^{[i]}$ which lies in $\Fil_0^\cnj$. 	\qedhere

	\end{proof}
	
	
	
\end{prop}

\smallskip

Next we prove the following enhancement of the classical Deligne-Illusie splitting:
\begin{thm}\label{Functorial_DeligneIllusie_for_affines}
	Let $\Aff^\sm_{/W_2(k)}$ be the category of smooth affine schemes over $W_2(k)$. Then there is a natural $k$-linear equivalence of functors
	$$
	\bigoplus_{i=0}^{p-1} \Omega^i_{-^{(1)}}\colon B\mapsto \bigoplus_{i=0}^{p-1} \Omega^i_{(B^{(1)}/p)/k}[-i] \ \ \ \text{ and } \ \ \ \tau^{\le p-1} \Omega^\bullet_{-,\dR} \colon B\mapsto \tau^{\le p-1} \Omega^\bullet_{(B/p)/k, \dR}
	$$
	from $\Aff^{\sm,\op}_{/W_2(k)}$ to $\DMod{k}$ which induces the Cartier isomorphism on the level of the individual cohomology functors.
\end{thm}

By \Cref{prop:syntomic_vs_perfect} and \Cref{prop:A_crys/p_on_smooth schemes} to deduce the statement of the theorem it is enough to prove the following: 
\begin{prop}\label{Hodge_splitting_on_QRSPerf}
	There is a natural isomorphism $f\colon \bigoplus_{r=0}^{p-1} \Gamma^r_{S}(\mathbb I/\mathbb I^2) \simeq \Fil_{p-1}^\cnj$ of sheaves of abelian groups on $\mr{QRSPerf}_{2}$ such that it agrees with $\kappa_{\le p-1}\colon \Gamma^{\le p-1}_{S}(\mathbb I/\mathbb I^2)\xra{\sim} \gr_{\le p-1}^\cnj (\mathbb A_{\crys}/p) $ after passing to the associated graded.
	
	\begin{proof}
		Given $\widetilde S\in \mr{QRSPerf}_{2}$ we denote by $S$ the reduction of $\widetilde S$ modulo $p$. As before we denote the kernel of the natural map $S^\flat\ra S$ by $I$. Note that $\Gamma^i_{S}(I/I^2)\simeq \Sym^i_{S}(I/I^2)$ for $i\le p-1$ and so, extending the map by multiplicativity, it is enough to construct a splitting $f\colon S^\flat/I\oplus I/I^2\xra{\sim} \Fil_1^\cnj$. Recall that we have a natural endomorphism $\phi\colon \mathbb A_{\crys}(S)\ra \mathbb A_{\crys}(S)$. We consider the Nygaard filtration (see Definition 8.9 of \cite{BMS2})
		$$
		\mc N^{\ge i} \mathbb A_{\crys}(S)\coloneqq\{x\in \mathbb A_{\crys}(S)\ |\ \phi(x)\in p^i\mathbb A_{\crys}(S)\}.
		$$
		In fact we will be interested only in the first two of its associated graded terms. We will construct $f$ by using the divided Frobenii, defined as follows. By Theorem 8.15(1) of \cite{BMS2} $\mathbb A_{\crys}(S)$ is $p$-torsion free and so for each $i\ge 0$ one has a well defined map
		$$
		\phi_i\coloneqq{\phi}/{p^i}\colon \mc N^{i}\mathbb A_{\crys}(S) \ra \mathbb A_{\crys}(S)/p
		$$
		from the $i$-th graded piece $\mc N^{i}\mathbb A_{\crys}(S)\coloneqq \mc N^{\ge i}\mathbb A_{\crys}(S)/\mc N^{\ge i+1}\mathbb A_{\crys}(S)$ of the Nygaard filtration.
		
		It is clear that $p\cdot \mathbb A_{\crys}(S)\subset\mc N^{\ge 1}\mathbb A_{\crys}(S)$; moreover, by Theorem 8.14(4) of \cite{BMS2}, $\mc N^{\ge 1}\mathbb A_{\crys}(S)\! \mod p\cdot \mathbb A_{\crys}$ is given by $I\subset \mathbb A_{\crys}(S)/p$. Thus $\mc N^0\coloneqq \mc N^{\ge 0}/\mc N^{\ge 1}\simeq S^\flat/I$ and $\phi_0$ induces an isomorphism $S^\flat/I\xra{\sim} \Fil_0^\cnj$ (since $\kappa_0=\phi$, this follows from \Cref{prop:Cartier on A_crys}).  We then also have a map $\phi_1\colon \mc N^1\mathbb A_{\crys}(S) \ra \mathbb A_{\crys}(S)/p$, which, by Theorem 8.14(2) of \cite{BMS2}, is an isomorphism onto $\Fil_1^\cnj$. Multiplication by $p$ induces a natural map $\mc N^0\mathbb A_{\crys}(S)\ra \mc N^1\mathbb A_{\crys}(S)$ which after composing with $\phi_1$ is identified with the embedding $\Fil_0^\cnj\subset \Fil_1^\cnj$. In fact, by flatness of $\mathbb A_{\crys}(S)$, we have $\mc N^{\ge 1}\mathbb A_{\crys}(S)\cap p\cdot \mathbb A_{\crys}(S)\simeq \mc N^{\ge 0}\mathbb A_{\crys}(S)$ and so $\Fil_0^\cnj$ (under the isomorphism given by $\phi_1$) is identified exactly with the subspace of those elements in $\mc N^{1}\mathbb A_{\crys}(S)\simeq \Fil_1^\cnj$ that lift to elements of $\mc N^{\ge 1}\mathbb A_{\crys}(S)$ divisible by $p$.
		
		We now use the lifting $\widetilde{S}$ of $S$ to construct the splitting of $\Fil_1^\cnj$. Recall that we have a map $\theta_{2,\widetilde S}\colon \mathbb A_{\crys}(S)/p^2\surj \widetilde S$ and let $K\coloneqq\ker \theta_{2,\widetilde S}$. Since both $\widetilde S$  and $\mathbb A_{\crys}(S)/p^2$ are flat over $W_2(k)$, we get that $K$ is also flat over $W_2(k)$ and that $K/pK \simeq I$:
		$$
		\xymatrix{0\ar[r]& K \ar[d]\ar[r] &\mathbb A_{\crys}(S)/p^2\ar[d] \ar[r] &\widetilde S\ar[r] \ar[d] &0\\
			0\ar[r]& I \ar[r] &\mathbb A_{\crys}(S)/p \ar[r] &S\ar[r]& 0.
		}
		$$
		The splitting is then given by applying $\phi_1$ to $K$. Namely, since $\phi(I)=0\in \mathbb A_{\crys}(S)/p$ it follows that $\phi(K)\subset p\cdot\mathbb A_{\crys}(S)/p^2$ and $K\subset \mc N^{\ge 1}\mathbb A_{\crys}(S) \mod p^2 \mathbb A_{\crys}(S)$.  The natural projection from $K$ to $\mc N^{1}\mathbb A_{\crys}(S)$ contains $p\cdot K+ \mc N^{\ge 2}\mathbb A_{\crys}(S)\! \mod p^2\mathbb A_{\crys}(S)$ in its kernel. Since $K/pK=I$ and the image of $\mc N^{\ge 2}\mathbb A_{\crys}$ modulo $p$ is given by $I^2$ (e.g. by Theorem 8.14(4) of \cite{BMS2}), we get that $\phi_1$ (applied to $K$) gives a well-defined map $f\colon I/I^2\ra \Fil_1^\cnj$. Moreover $K\cap (p\cdot \mathbb A_{\crys}(S)/p^2)\subset p\cdot K$, since $K$ is flat over $W_2(k)$, and so the image of $f$ does not intersect with $\Fil_0^\cnj$. 
		
		It remains to check that the constructed $f\colon I/I^2\ra \Fil_1^\cnj$ coincides with $\kappa_1$ after the projection to $\Fil_1^\cnj\!/\Fil_0^\cnj$. Given $s\in I$ let $\widetilde s=[s]+p\cdot s'\in K\subset \mathbb A_{\crys}/p^2$ be a lifting of $s$ to an element of $K$. Then 
		$$
		\phi(\widetilde s)=\phi([s])+p\cdot \phi(s')=(p-1)!\cdot p\cdot[s]^{[p]}+p \cdot \phi(s') \ \Rightarrow \ f(s)= (p-1)!\cdot s^{[p]}+ \phi(s').
		$$ 
		By the discussion above (see also Theorem 8.14(2) in \cite{BMS2}) $\phi(s')\in \Fil_0^\cnj$ and $f(s)= (p-1)!\cdot s^{[p]}$ modulo $\Fil_0^\cnj$.
		
		Since the above splitting is clearly functorial in $\widetilde S$ we get the statement of the proposition.
	\end{proof}
\end{prop}

As a corollary we deduce
\begin{thm}\label{DeligneIllusieQuis}
	Let $\mstack Y$ be a smooth Artin stack over a perfect field $k$ of characteristic $p$ admitting a smooth lift to the ring of the second Witt vectors $W_2(k)$. Then there is a canonical equivalence
	$$\RG(\mstack Y, \tau^{\le p-1} \Omega_{\mstack Y,\dR}^\bullet) \simeq \RG\left(\mstack Y^{(1)}, \bigoplus_{i=0}^{p-1} \wedge^i \mathbb L_{\mstack Y^{(1)}/k}[-i]\right).$$
	In particular for $n\le p-1$ we have $H^n_\dR(\mstack Y/k) \simeq H^n_\Hdg(\mstack Y^{(1)}/k)$.
	
	\begin{proof}
		Let $\pi \colon \Stk^{n\mdef\Art, \sm}_{/W_2(k)} \to \Stk^{n\mdef\Art, \sm}_{k}$ be the reduction functor, $\widetilde{\mstack Y} \mapsto \widetilde{\mstack Y}\otimes_{W_2(k)} k$. By \Cref{Functorial_DeligneIllusie_for_affines} it is enough to prove that the natural map (existing by the universal property of the right Kan extensions)
		\begin{align}\label{eq:proof_of_DI_sytcks}
		\RG_{\dR}(-/k) \circ \pi \to \Ran_{i_2}(\RG_\dR(-/k) \circ \pi_{|\Aff_{/W_2(k)}^\sm})
		\end{align}
		(where $i_2$ denotes the inclusion functor $\Aff_{/W_2(k)}^\sm \inj \Stk^{n\mdef\Art, \sm}_{/W_2(k)}$) is an equivalence. Since both sides of \eqref{eq:proof_of_DI_sytcks} satisfy smooth descent, by induction on $n$ we reduce the statement to the case of smooth affine schemes over $W_2(k)$, where \eqref{eq:proof_of_DI_sytcks} is evidently an equivalence.
	\end{proof}
\end{thm}
\begin{cor}
	Let $\mstack Y$ be a smooth Hodge-proper stack over a perfect field $k$ of characteristic $p$ admitting a smooth lift to $W_2(k)$. Then the Hodge-to-de Rham spectral sequence $H^j(\mstack Y, \wedge^i \mathbb L_{\mstack Y/k}) \Rightarrow H_\dR^{i+j}(\mstack Y / k)$ degenerates at the first page for $i+j<p$.

\begin{proof}
	This follows from  \Cref{DeligneIllusieQuis} and the equality of dimensions $\dim_k H^n_{\mr H}(\mstack Y)=\dim_k H^n_{\mr H}(\mstack Y^{(1)})$.
\end{proof}
\end{cor}

\subsection{Degeneration in characteristic zero}\label{sec:Degeneration in char 0}
To reduce the statement in characteristic $0$ to results of the previous section we introduce the following notion:
\begin{defn}\label{def: spredable_stacks}
A smooth Hodge-proper Artin stack $\mstack X$ over a field $F$ of characteristic $0$ is called \emph{Hodge-properly spreadable} if there exists a $\mathbb Z$-subalgebra $R\subset F$ and an Artin stack $\mstack X_R$ over $\Spec R$ such that 
\begin{itemize}
\item $R$ is a localization of a smooth $\mathbb Z$-algebra such that the image of $\Spec R$ in $\Spec \mathbb Z$ is open.

\item $\mstack X_R$ is smooth over $R$ and $\mstack X \otimes_R F \coloneqq \mstack X_R \times_{\Spec R} \Spec F\simeq \mstack X$.

\item $\mstack X_R$ is Hodge-proper over $R$, namely $\RG(\mstack X_R,\wedge^p \mathbb L_{X_R/R})$ is bounded below coherent over $R$ for any $p\ge 0$.
\end{itemize}
\end{defn}
\begin{rem}\label{rem: cofilteredness of all such R in F}
	We note that any field $F$ of characteristic 0 is a union of all such subrings $R\subset F$ (in fact even a union of those that are smooth over $\mbb Z$). As we will see in \Cref{sec:non-proper schemes} allowing some infinite localizations of smooth algebras makes some difference when constructing examples. The condition on openness of the image is added to guarantee that the diagram of all such $R\subset F$ is filtered and that for any such $R$ the image of $\Spec R$ in $\Spec \mbb Z$ is infinite. To see the first point: indeed, having $Q_1=R_1[S_1^{-1}]$, $Q_2=R_2[S_2^{-1}]$, $Q_1,Q_2\subset F$ being localizations of smooth $\mbb Z$-algebras $R_1,R_2$ for some subsets $S_i\subset R_i$ as in \Cref{def: spredable_stacks}, for any finite localization $ (Q_1\cdot Q_2)[1/f]\subset F$ the image of $\Spec (R_1\cdot R_2)[1/f]$ in $\Spec \mbb Z$ is still open. Then, picking $f$ such that $(R_1\cdot R_2)[1/f]$ is again smooth over $\mbb Z$ we get a subring $(Q_1\cdot Q_2)[1/f]\subset F$ that contains both $Q_1, Q_2$ and fits in \Cref{def: spredable_stacks}. 
\end{rem}
We defer a thorough discussion of spreadability of stacks till the next section. We only stress here again, that (unlike in the case of proper schemes) Hodge-proper spreadings do not exist in general (see \Cref{sec:BG}).

Now we will deduce the promised Hodge-to-de Rham degeneration in characteristic $0$:
\begin{thm}\label{Hodge_degeneration}
Let $\mstack X$ be a smooth Hodge-properly spreadable Artin stack over a field $F$ of characteristic zero. Then the Hodge-to-de Rham spectral sequence for $\mstack X$ degenerates at the first page. In particular for each $n\ge 0$ there exists a (non-canonical) isomorphism
$$H^n_\dR(\mstack X/F) \simeq \bigoplus_{p+q=n} H^{p,q}(\mstack X/F).$$

\begin{proof}
For the rest of the proof fix $n\in \mathbb Z_{\ge 0}$. By Hodge-properness of $\mstack X$ it is enough to prove
$$\dim_F H^n_\dR(\mstack X/F) = \dim_F H^n_\Hdg(\mstack X/F).$$

Let $R$ and $\mstack X_R$ be as in \Cref{def: spredable_stacks}. Note that by the assumption on $\mstack X_R$ and \Cref{de_rham_finiteness_of_Hodge_proper} both $H^n_\dR(\mstack X_R/R)$ and $H^n_\Hdg(\mstack X_R/R)$ are finitely generated $R$-modules. Localizing $R$ if necessary, we can assume that $R$ is connected of some Krull dimension $d$, and that the $i$-th cohomology groups $H^i_\dR(\mstack X_R/R)$ and $H^i_\Hdg(\mstack X_R/R)$ are free $R$-modules of finite rank for $i = n, n+1, \ldots, n+d$.\footnote{Note that there does not necessarily exist a localization $R[s^{-1}]$ such that \emph{for all $i$} the $R[s^{-1}]$-modules $H^i_\dR(\mstack X_R/R)[s^{-1}]$ (or $H^i_\Hdg(\mstack X_R/R)[s^{-1}]$) are free, since there are infinitely many of them.} Note that for any point $s\colon \Spec k \to \Spec R$ the map $R \to k$ can be factored as a composition of a flat map $R \to R_{s}$ (where $R_s$ is a local ring of $s$) and a map $R_s \to k$ of finite Tor-amplitude (by regularity assumption $k$ is perfect as an $R_s$-module). Hence by \Cref{Hodge_and_deRham_basechange} we have
$$R\Gamma_\dR(\mstack X_k/k)\simeq \RG_\dR(\mstack X_R/R)\otimes_R k \qquad\text{and}\qquad R\Gamma_\Hdg(\mstack X_k/k)\simeq R\Gamma_\Hdg(\mstack X_R/R)\otimes_R k,$$
where $\mstack X_k \coloneqq \mstack X_R \otimes_R k$. Since the $n, (n+1), \ldots, (n+d)$-th cohomology groups are free as $R$-modules and since the Tor-amplitude of $k$ over $R$ is bounded by $d$, we get $H^n_\dR(\mstack X_k/k)\simeq H^n_\dR(\mstack X_R/R)\otimes_R k$, so
$$
\dim_F H^n_\dR(\mstack X/F) = \rank_R H^n_\dR(\mstack X_R/R) = \dim_k H^n_\dR(\mstack X_k/k)
$$
and analogously for the Hodge cohomology. In particular, to prove that $\dim_F H^n_\dR(\mstack X/F) = \dim_F H^n_\Hdg(\mstack X/F)$ it is enough to show that $\dim_k H^n_\dR(\mstack X_k/k) = \dim_k H^n_\Hdg(\mstack X_k/k)$ for some point $s\colon\Spec k\to \Spec R$.

To do so, note that by the infiniteness of the image of $\Spec R \to \Spec \mathbb Z$ and \Cref{lemma below} below, there exists a closed point $s\colon \Spec k \inj \Spec R$ of characteristic greater than $n$, such that the map $R \to k \to k^\perf$ factors through the ring of the second Witt vectors $W_2(k^\perf)$. Since the base change $\mstack X_{W_2(k^\perf)}\coloneqq \mstack X_R\times_R W_2(k^\perf)$ is smooth and Hodge-proper over $W_2(k^\perf)$, by \Cref{DeligneIllusieQuis} we have
$$\dim_{k^\perf} H^n_\dR(\mstack X_{k^\perf}/k^\perf) = \dim_{k^\perf} H^n_\Hdg(\mstack X_{k^\perf}/k^\perf).$$
Finally, by base change (applied to $k\ra k^{\perf}$) we get 
$$
\dim_{k} H^n_\dR(\mstack X_{k}/k) = \dim_{k} H^n_\Hdg(\mstack X_{k}/k)
$$
as desired.
\end{proof}
\end{thm}
\begin{lem}\label{lemma below}
Let $R$ be a localization of a smooth $\mathbb Z$-algebra. Then for any field $k$ of positive characteristic and a map $R \to k$ the composite map $R \to k\inj k^\perf$ factors through the ring $W_2(k^\perf)$ of the second Witt vectors.

\begin{proof}
By assumption on $R$ the cotangent complex $\mathbb L_{R/\mathbb Z} \simeq \Omega_{R/\mathbb Z}[0]$ is concentrated in degree zero and is a locally free (in particular flat) $R$-module. By the basic deformation theory the obstruction to lift a map $R \to k^\perf$ to $R \to W_2(k^\perf) \to k^\perf$ lies in $\Ext^1_{k^\perf}(\mathbb L_{R/\mathbb Z} \otimes_{R} k^\perf, k^\perf)$. But the latter group vanishes, since by flatness of $\mathbb L_{R/\mathbb Z}$, the restriction $\mathbb L_{R/\mathbb Z} \otimes_{R} k^\perf$ is a complex of $k^\perf$-vector spaces concentrated in degree $0$.
\end{proof} 
\end{lem}

\subsection{Equivariant Hodge decomposition}\label{sec:Examples of equivariant Hodge degeneration}
In this section we apply \Cref{Hodge_degeneration} to obtain a (non-canonical) Hodge decomposition for the equivariant singular cohomology of an algebraic variety $X$ with a $G$ action, under the assumption that the corresponding quotient stack $[X/G]$ is Hodge-properly spreadable.

Let $K$ be a homotopy type with an action of a topological group $H$ (i.e. an $(\infty,1)$-functor $K_\bullet\colon BH \to \Type$, where $\Type$ denotes the $(\infty,1)$-category of spaces, see \Cref{sect:notations}). Recall that the $H$-equivariant cohomology $C^*_H(K, \Lambda)$ of $K$ with coefficients in a ring $\Lambda$ are defined as
$$C^*_H(K, \Lambda) \coloneqq C^*(K_{hH}, \Lambda),$$
where $K_{hH}$ is the homotopy quotient of $K$ by $H$ (i.e. a colimit of the corresponding functor $K_\bullet$, or, more classically, $(K\times EH)/H$). 

If $X$ is a smooth algebraic variety over a field $F\subseteq \mathbb C$ equipped with an action of an algebraic group $G$, then the de Rham cohomology of $[X/G]$ gives a model for the $G(\mathbb C)$-equivariant singular cohomology of $X(\mathbb C)$:
\begin{prop}
Let $X$ and $G$ be as above. Then there is a canonical equivalence
$$C^*_{G(\mathbb C)}(X(\mathbb C), \mathbb C) \simeq \RG_\dR([X/G] / F)\otimes_F \mathbb C.$$
\begin{proof}
By definition we have
\begin{align*}
\left|\xymatrix{\ldots \ar@<+1.35ex>[r] \ar@<+.45ex>[r] \ar@<-.45ex>[r] \ar@<-1.35ex>[r] & G \times G \times X \ar@<+.9ex>[r] \ar[r] \ar@<-.9ex>[r] & G \times X \ar@<-.45ex>[r] \ar@<.45ex>[r] & X}\right| & \simeq [X/G],\\
\left|\xymatrix{\ldots \ar@<+1.35ex>[r] \ar@<+.45ex>[r] \ar@<-.45ex>[r] \ar@<-1.35ex>[r] & G(\mathbb C) \times G(\mathbb C) \times X(\mathbb C) \ar@<+.9ex>[r] \ar[r] \ar@<-.9ex>[r] & G(\mathbb C) \times X(\mathbb C) \ar@<-.45ex>[r] \ar@<.45ex>[r] & X(\mathbb C)}\right| & \simeq X(\mathbb C)_{hG(\mathbb C)}.
\end{align*}
Since the functor of cochains $C^*(-, \mathbb C)$ sends colimits of homotopy types to limits of complexes and by smooth descent for $\RG_\dR(- /F )\otimes_F \mathbb C$, the result follows from the analogous comparison between algebraic de Rham and Betti cohomology for ordinary smooth schemes $X\times G^n$.
\end{proof}
\end{prop}
\begin{cor}[Equivariant Hodge decomposition]\label{cor:Hodge decomposition for singular cohomology}
Let $X$ be a smooth scheme over $\mathbb C$ with an action of an algebraic group $G$. Assume that $[X/G]$ is Hodge-properly spreadable (e.g. $X$ and $G$ satisfy the conditions of Theorem \ref{thm:proper coarse moduli} or \ref{thm: conical examples}). Then for all $n\in \mathbb Z_{\ge 0}$ there exists an isomorphism
$$H^n_{G(\mathbb C)}(X(\mathbb C), \mathbb C) \simeq \bigoplus_{p+q = n} H^q([X/G], \wedge^p \mathbb L_{[X/G]/\mathbb C}).$$
\end{cor}

\begin{ex} Let $X=\Spec \mathbb C$. Then $\wedge^n \mathbb L_{BG} \simeq \Sym^n(\mf g^\vee)[-n]$ where $\mathfrak g$ is the Lie algebra of $G$ endowed with the adjoint action of $G$. This way we get a standard isomorphism 
$$
H^n_{G(\mathbb C)}(\mr{pt},\mathbb C)\simeq \left\{ \begin{array}{ll}
\Sym^k(\mf g^\vee)^G & \text{ if $n=2k$},\\ 
0 & \text{ if $n=2k+1$}.
\end{array}\right.
$$
In particular,
$$
H^\bullet_{G(\mathbb C)}(\mr{pt},\mathbb C)\simeq \Sym(\mf g^\vee)^G,\quad\text{where}\quad \deg(\mathfrak g^\vee) = 2.
$$
\end{ex}

\begin{ex}
As another example one can take a conical resolution $\pi\colon X\ra \Spec A$ (see the second example of \ref{ex:mpre examples}). Following \Cref{ex:mpre examples}, the quotient stack $[X/\mathbb G_m]$ is Hodge-properly spreadable and we get a decomposition for $H^\bullet_{\mathbb C^\times}\!(X(\mathbb C),\mathbb C)$ as in \Cref{cor:Hodge decomposition for singular cohomology}. Note also that in this case $H^1_{\mbb C^\times}(X(\mbb C),\mbb C)\simeq H^1(X(\mbb C),\mbb C)$; indeed one can replace $\mbb C^\times$ with $S^1$ and consider the Serre-Lerray spectral sequence $$
E_2^{p,q}=H^{p}(BS^1, H^q(X(\mbb C),\mbb C)) \Rightarrow H^{p+q}_{S^1}(X(\mbb C),\mbb C).
$$ We have $BS^1\simeq \mbb CP^\infty$, thus $H^{1}(BS^1, H^0(X(\mbb C),\mbb C))=0$ and it's enough to show that $d_2^{0,1}=0$. We leave it as an exercise to the reader to check that this is trues as soon as $X(\mbb C)$ is connected and $X(\mbb C)^{S^1}\neq 0$. 

From all this we get a decomposition
\begin{equation}\label{formula}
H^1(X(\mathbb C),\mathbb C)\simeq H^0([X/\mathbb G_m], \mathbb L_{[X/\mathbb G_m]})\oplus H^1([X/\mathbb G_m], \mc O_{[X/\mathbb G_m]}).
\end{equation}
We have $\mathbb L_{[X/\mathbb G_m]}\simeq \Omega^1_X\xra{a^*} \mc O_X$ as a complex of $\mathbb G_m$-equivariant sheaves on $X$, where $a^*$ is the map dual to the derivative of the action $\mr{Lie}(\mathbb G_m)\otimes_{\mathbb C}\mc O_X\ra \mathbb T_X$ (where $\mathbb T_X$ denotes the tangent bundle). Then $\mathbb H^0(X, \Omega^1_X\xra{a^*} \mc O_X)\simeq \ker\left(H^0(X,\Omega^1_X)\xra{a^*} H^0(X,\mc O_X)\right)$, which is identified with the invariants of the Lie algebra action, which also identifies with the group invariants $H^0(X,\Omega^1_X)^{\mathbb G_m}$. Finally we get
$$
H^0([X/\mathbb G_m], \mathbb L_{[X/\mathbb G_m]})\simeq H^0(X, \Omega^1_X\xra{a^*} \mc O_X)^{\mathbb G_m} \simeq H^0(X,\Omega^1_X)^{\mathbb G_m}
$$
as well. The second summand in (\ref{formula}) is just $H^1(X,\mc O_X)^{\mathbb G_m}$. Thus for any conical resolution we get a formula 
$$
H^1(X(\mathbb C),\mathbb C)\simeq H^0(X,\Omega^1_X)^{\mathbb G_m}\oplus H^1(X,\mc O_X)^{\mathbb G_m}.
$$
This is a partial generalization of results of Section 6 in \cite{KubrakTravkin} to the case when $R^1\pi_*\mc O_X$ is not necessarily $0$.
\end{ex}

\section{Spreadings}\label{sec:Spreadings of stacks}
To apply \Cref{Hodge_degeneration} we need to find a good model of our stack over a finitely generated $\mathbb Z$-algebra, namely a Hodge-proper spreading. However, as we will see, such a spreading does not necessarily exist in general. 

In \Cref{sec:Spreadable stacks} we first prove a general result about the existence of spreadings for some more natural classes of morphisms between Artin stacks (like smooth, flat, etc). Then some examples of Hodge-properly spreadable and nonspreadable stacks are given in \Cref{sec:Examples of spreadable Hodge-proper stacks}.

\subsection{Spreadable classes}\label{sec:Spreadable stacks}
\begin{defn}
Let $\mathcal P$ be a class of morphisms of schemes (e.g. $\mathcal P = $ smooth, flat or proper morphisms) containing all isomorphisms and closed under compositions. For a scheme $S$, define $\Sch_{/S}^{\fp, \mathcal P}$ to be the (non-full) subcategory of schemes over $S$ consisting of finitely-presentable $S$-schemes and morphisms between them that belong to $\mathcal P$.
\end{defn}
\begin{thm}[{\cite[Theorems 8.10.5, 11.2.6]{EGA_IV3} and \cite[Proposition 17.7.8]{EGA_IV4}}] \label{thm:spreading of schemes}
Let $\{S_i\}$ be a filtered diagram of affine schemes with limit $S$ and let $\mathcal P$ be one of the following classes of morphisms: isomorphisms, surjections, closed embeddings, flat, smooth or proper morphisms\footnote{The list is not even nearly complete. See \cite[Appendix C.1]{Poonen} for a much more exhaustive list of classes of morphisms and their properties with precise references.}. Then the natural functor
$$\indlim[i] \Sch^{\fp,\mathcal P}_{/S_i} \to \Sch^{\fp,\mathcal P}_{/S}$$
(induced by the base change $\Sch^{\fp, \mathcal P}_{/S_i} \ni X \mapsto X\times_{S_i} S$) is an equivalence.
\end{thm}
We will say that a scheme $X$ is a \emph{$\mathcal P$-scheme over $S$ } ($\mathcal P$-scheme$/S$) if $X$ is an $S$-scheme and the structure morphism $X\to S$ is in $\mathcal P$. From the theorem above one can formally deduces the following corollary (see \Cref{spreading_with_properties_for_stacks} for a proof in a bit more general stacky setting):
\begin{cor}\label{spreading_with_properties_for_schemes}
Let $\{S_i\}_{i\in I}, S$ and $\mathcal P$ be as above. Then if $X$ is a finitely presentable $\mathcal P$-scheme$/S$, then there exists $i\in I$ and a finitely presentable $\mathcal P$-scheme $X_i$ over $S_i$, such that $X \simeq X_i\times_{S_i} S$.
\end{cor}

Our goal in this section is to extend \Cref{thm:spreading of schemes} to the setting of Artin stacks. First we recall how "finitely presentable" is defined in Artin setting:
\begin{defn}[Finitely presentable Artin stacks]
A $(-1)$-Artin stack $\mstack X$ over a base ring $R$ is called \emph{finitely presentable}, if $\mstack X\simeq \Spec A$ and $A$ is a finitely presentable $R$-algebra. Then, an $n$-Artin stack $\mstack X$ over $R$ is called finitely presentable if there exists a smooth atlas $U \surj \mstack X$ such that $U$ is a finitely presentable affine scheme and $U \times_{\mstack X} U$ is a finitely presentable $(n-1)$-Artin $R$-stack. We will denote the category of finitely presentable $n$-Artin stacks by $\Stk^{n\mdef\Art,\fp}$.
\end{defn}
Our general strategy for proving results about spreadability is to inductively reduce to the case of finitely presentable schemes. For this end it will be technically convenient to use instead of iterative description of Artin stacks a representation as a geometric realization of a coskeletal hypercover by schemes:
\begin{construction}
Let $X_\bullet\colon \Delta^\op \to \fcat C$ be a simplicial object in a category $\fcat C$ admitting finite limits. Define $X(-)\colon \SSet^{\fin, \op} \to \fcat C$ to be the right Kan extension of $X_\bullet$ along the inclusion of $\Delta^\op$ into the opposite  $\SSet^{\fin,\op}$ of the category of finite simplicial sets (meaning simplicial sets with only finitely many non-degenerate simplices). More concretely, for a finite simplicial set $K$
$$X(K) \simeq \lim\limits_{\Delta^n \in \Delta^\op_{/K}} X(\Delta^n).$$
In particular, we denote $M_n(X_\bullet) \coloneqq X(\partial \Delta^n)$ and call it the \emph{$n$-th matching object of $X_\bullet$}.
\end{construction}
\begin{defn}
Let $\mathbb H$ be an $\infty$-topos. An augmented simplicial object $X_\bullet \to X_{-1}$ is called a \emph{hypercover of $X_{-1}$} if for any $n\in \mathbb Z_{\ge 0}$ the natural map $X_n \to M_n(X_\bullet)$ is an effective epimorphism ($M_n$ is computed in the category $\mathbb H_{/X_{-1}}$). A hypercover $X_\bullet$ is called \emph{$n$-coskeletal} if additionally for each $m>n$ the natural map $X_m \to M_m(X_\bullet)$ is an equivalence (equivalently $X_\bullet$ coincides with the right Kan extension of its restriction to $\Delta_{\le n}^\op$). We refer interested reader to \cite[Appendix]{MathewBrantner_LLambda} for a quick recap on hypercovers and to \cite[Section 2]{Pridham} for a discussion of hypergroupoids, which is most relevant for this section.
\end{defn}

With this notation, $n$-Artin stacks can be thought of as some special $(n-1)$-coskeletal hypercovers:
\begin{thm}[{\cite[Proposition 4.1 and Theorem 4.7]{Pridham}}]
Let $\mstack X$ be an $n$-Artin stack over $S$. Then there exists an $(n-1)$-coskeletal hypercover $X_\bullet$ of $\mstack X$ such that all $X_k$ are equivalent to coproducts of affine schemes and for all $m,k,$ with $0\le m \le k$, the maps $X_k \to X(\Lambda^k_m)$ are smooth surjections. Conversely, given $X_\bullet$ as above, its geometric realization $|X_\bullet|$ (in the category of stacks, i.e. sheaves of spaces in \'etale topology) is an $n$-Artin stack.
\end{thm}
\begin{cor}\label{represent_fp}
Let $\mstack X$ be a finitely presented $n$-Artin stack over $S$. Then there exists an $(n-1)$-coskeletal hypercover $X_\bullet$ of $\mstack X$ such that all $X_k$ are finitely presentable affine schemes and for all $m,k,$ with $ \ 0\le m \le k$, the maps $X_k \to X(\Lambda^k_m)$ are smooth surjections. Conversely, given $X_\bullet$ as above, its geometric realization $|X_\bullet|$ (in the category of stacks) is a finitely presentable $n$-Artin stack.

\begin{proof}
Let $\mstack X$ be a finitely presentable $n$-Artin stack. The simplicial scheme $X_\bullet$ from the theorem above is constructed inductively in \cite[Proposition 4.5]{Pridham} using only finite limits and atlases, hence all $X_i$ can be chosen to be finitely presentable.

Conversely, if $X_\bullet$ is a simplicial affine scheme as in the statement of corollary, then by the theorem above $|X_\bullet|$ is an $n$-Artin stack. Moreover, the natural map $X_0 \to |X_\bullet|$ is a smooth finitely presentable atlas. To prove that $X_0 \times_{|X_\bullet|} X_0$ is finitely presented, recall that by \cite[Remark 2.25]{Pridham} there is a natural equivalence
$$X_0 \times_{|X_\bullet|} X_0 \simeq |X_0 \times_{X_\bullet} \mathrm{Dec}_+(X_\bullet)|,$$
where $\mathrm{Dec}_+$ is the d\'ecalage functor, $\mathrm{Dec}_+(X_\bullet)_i \simeq X_{i+1}$. Since finitely presentable affine schemes are closed under fibered products, it follows that $X_0 \times_{X_\bullet} \mathrm{Dec}_+(X_\bullet)$ also satisfies conditions of the corollary. Since moreover, $X_0 \times_{X_\bullet} \mathrm{Dec}_+(X_\bullet)$ is $(n-2)$-coskeletal, its geometric realization $X_0 \times_{|X_\bullet|} X_0$ is finitely presentable by induction.
\end{proof}
\end{cor}

For convenience we introduce the following notation:
\begin{defn}[Spreadable class]\label{def:spreadable_class}
A class of morphism $\mathcal P$ between Artin stacks is called \emph{spreadable} if
\begin{itemize}
\item $\mathcal P$ is closed under arbitrary base changes, compositions and contains all equivalences.

\item (Locality on source and target) Let $f\colon \mstack X \to \mstack Y$ be a morphism of finitely presentable Artin stacks. Then $f$ lies in $\mathcal P$ if and only if there exist smooth finitely presentable affine atlases $U\surj \mstack Y$ and $V\surj U\times_{\mstack Y} \mstack X$ such that the composite map $V \to U \times_{\mstack Y} \mstack X \to U$ is in $\mathcal P$.

\item (Affine spreadability) Let $\{S_i\}$ be a filtered diagram of affine schemes with the limit $S$. Let $f\colon X \to Y$ be a morphism in $\mathcal P$ between affine finitely presentable $S$-schemes. Then for some $i$ there exists a map $f_i\colon X_i\to Y_i$ in $\mathcal P$ of affine finitely presentable $S_i$-schemes, such that $f \simeq f_i \times_{S_i} S$.
\end{itemize}
\end{defn}
\begin{ex}
If $\mathcal P$ and $\mathcal Q$ is a pair of spreadable classes, then $\mathcal P\cap \mathcal Q$ and $\mathcal P \cup \mathcal Q$ are also spreadable. There exists the smallest spreadable class (consisting only of equivalences) and the largest one (consisting of all finitely presentable morphisms).
\end{ex}
\begin{ex}
Since surjective, smooth and flat morphisms of Artin stacks are by definition local on the source and the target for the flat topology, by \Cref{thm:spreading of schemes} we get that these classes are spreadable.
\end{ex}
\begin{defn}
Let $\mathcal P$ be a spreadable class and let $S$ be a scheme. Let us denote by $\Stk_{/S}^{n\mdef\Art,\fp, \mathcal P}$ the subcategory of the category of finitely presentable $n$-Artin stacks over $S$ and morphisms from $\mathcal P$ between them.
\end{defn}

We are now ready to prove the main technical result of this section (see \cite{Rydh_NoetherianApprox}, \cite[Chapter 4]{MoretLaumon} for similar results in the context of $1$-Artin stacks and \cite[Theorem 4.4.2.2]{Lur_SAG} for the spectral version):
\begin{thm}\label{fp_over_filtered_limit}
Let $\{S_i\}$ be a filtered diagram of affine schemes with limit $S$. Let $\mathcal P$ be a spreadable class. Then the natural functor
$$\indlim[i] \Stk_{/S_i}^{n\mdef\Art,\fp, \mathcal P} \xymatrix{\ar[r] &} \Stk_{/S}^{n\mdef\Art,\fp, \mathcal P}$$
(induced by the base-change $\Stk_{/S_i}^{n\mdef\Art,\fp, \mathcal P} \ni \mstack X_i \mapsto \mstack X_i \times_{S_i} S$) is an equivalence.

\begin{proof}
We will prove the statement by induction on $n$. The base of the induction $n=-1$, i.e. the case of affine schemes, holds by the definition of a spreadable class. To make the induction step, we first prove the statement for $\mathcal P = $ "all (finitely presented) morphisms" (using the induction assumption for smooth surjective morphisms) and then deduce the statement for a general spreadable class $\mathcal P$.

\smallskip\clause{Essential surjectivity for $\mathcal P =\text{"all"}$}. Since all $n$-Artin stacks are $(n+1)$-truncated, the Yoneda embedding $\Stk^{n\mdef\Art} \inj \Fun(\CAlg, \Type)$ factors through a full subcategory $\Fun(\CAlg, \Type_{\le n+1})=: \PStk_{\le n+1}$. Let now $\mstack X$ be a finitely presented $n$-Artin $S$-stack and let $X_\bullet$ be a simplicial diagram of finitely presented affine $S$-schemes, so that $|X_\bullet| \simeq \mstack X$ (as in \Cref{represent_fp}). Since for any simplicial diagram $A_\bullet$ in any $(n+1, 1)$-category the natural map $|A_\bullet|_{\le n+2} \to |A_\bullet|$ is an equivalence, we see that $\mstack X \simeq |X_\bullet|_{\le n+2}$ in $\PStk_{\le n+1}$. But $X_{\bullet|\Delta_{\le n+2}^\op}$ is a finite diagram of finitely presented affine schemes, hence there exists $S_i$ and a diagram $X_{\bullet|\Delta_{\le n+2}^\op, S_i}$ such that $X_{\bullet|\Delta_{\le n+2}^\op} \simeq X_{\bullet|\Delta_{\le n+2}^\op, S_i} \times_{S_i} S$. We set $\mstack X_{S_i} \coloneqq |X_{\le n+2, S_i}|$. By applying the inductive assumption with $\mathcal P = $ "smooth surjective", we can assume that all maps $X_{k, S_j} \to X_{\bullet, S_j}(\Lambda^k_m)$ are smooth and surjective for some $S_j$; hence by \Cref{represent_fp} $\mstack X_{S_j}$ is a finitely presented $n$-Artin spreading of $\mstack X$.

\smallskip\clause{Fully-faithfulness for $\mathcal P =\text{"all"}$}. Let $\mstack X_i, \mstack Y_i$ be a pair of $n$-Artin stacks of finite presentation over $S_i$. We then have
\begin{align}\label{eq:fp_fully_faithful}
\indlim[j] \Hom_{\PStk_{/S_j}}(\mstack X_i \times_{S_i} S_j, \mstack Y_i \times_{S_i} S_j) \simeq \indlim[j] \Hom_{\PStk_{/S_i}}(\mstack X_i \times_{S_i} S_j, \mstack Y_i) \simeq \indlim[j] \Hom_{\PStk_{\le n+1 /S_i}}(\mstack X_i \times_{S_i} S_j, \mstack Y_i),
\end{align}
where the second equivalence follows from the fact that filtered co-limits commute with $\pi_*$, hence preserve $(n+1)$-truncated spaces. Let now $X_\bullet \to \mstack X$ be as in \Cref{represent_fp}. Then
\begin{align*}
\eqref{eq:fp_fully_faithful} & \ldots \simeq \indlim[j] \Hom_{\PStk_{\le n+1 /S_i}}(|X_\bullet \times_{S_i} S_j |_{\le n+2}, \mstack Y_i) \simeq \Tot_{\le n+2} \indlim[j] \Hom_{\PStk_{\le n+1 /S_i}}(X_\bullet \times_{S_i} S_j, \mstack Y_i) \simeq \\ & \simeq \Tot_{\le n+2} \indlim[j] \Hom_{\Stk_{/S_i}}(X_\bullet \times_{S_i} S_j, \mstack Y_i),
\end{align*}
where the second equivalence follows from the fact that, since $\Delta_{\le n + 2}$ is a finite diagram, limits along $\Delta_{\le n+2}$ commute with filtered co-limits. Similarly, one shows that
$$\Hom_{\Stk_{/S}}(\mstack X_i\times_{S_i} S, \mstack Y_i\times_{S_i} S) \simeq \Tot_{\le n+2} \Hom_{\Stk_{/S_i}}(X_\bullet \times_{S_i} S, \mstack Y_i).$$
Finally, since $\mathcal Y_i$ is finitely presentable, by \cite[Chapter 2, Proposition 4.5.2]{GaitsRozI}
$$\indlim[j] \Hom_{\Stk_{/S_i}}(X_\bullet \times_{S_i} S_j, \mstack Y_i) \simeq \Hom_{\Stk_{/S_i}}(X_\bullet \times_{S_i} S, \mstack Y_i).$$

\smallskip\clause{General $\mathcal P$}. Let $f\colon \mstack X \to \mstack Y$ be a morphism in a spreadable class $\mathcal P$ over $S$. It is enough to prove that there exists $i$ and a map between finitely presentable $n$-Artin $S_i$-stacks $f_i\colon \mstack X_i \to \mstack Y_i$ such that $f_i\times_{S_i} S\simeq f$ and $f_i \in \mathcal P$. Choose affine finitely presentable atlases $U\surj \mstack Y$ and $V \surj U \times_{\mstack Y} \mstack X$. The induced map $g\colon V\to U$ belongs to $\mathcal P$, so by the previous part and definition of spredable classes, the diagram
$$\xymatrix{
V \ar@{->>}[r] \ar[d]^g & \mstack X \ar[d]^f \\
U \ar@{->>}[r] & \mstack Y
}$$
can be spread out to some $S_i$, such that $g_{S_i}$ belongs to $\mathcal P$. It follows by the definition of spreadable class, that $f_{S_i}$ is also in $\mathcal P$.
\end{proof}
\end{thm}
A stack $\mstack X$ is called an \emph{$n$-Artin $\mc P$-stack over $S$} if the structure morphism $\pi\colon \mstack X \to S$ exhibits $\mstack X$ as an $n$-Artin stack and $\pi$ is in $\mathcal P$.
\begin{cor}[Existence of spreading in a predefined class]\label{spreading_with_properties_for_stacks}
Let $\{S_i\}_{i\in I}$ be a filtered diagram of affine schemes, $S \coloneqq \prolim S_i$ and $\mathcal P$ be a spreadable class. Then if $\mstack X$ is a finitely presentable  $n$-Artin $\mathcal P$-stack over $S$, then there exists $i\in I$ and a finitely presentable $n$-Artin  $\mathcal P$-stack $\mstack X_i$ over $S_i$, such that $\mstack X \simeq \mstack X_i\times_{S_i} S$.

\begin{proof}
Let $\pi\colon \mstack X \to S$ be the structure morphism. By the previous theorem and the description of objects in filtered colimits of categories (see e.g. \cite{Nick_FilteredCats}) there exists a finitely presented stack $\pi_j \colon \mstack X_j \to S_j$ such that $\pi_j\times_{S_j} S = \pi$. A morphism in a filtered colimit of categories is a filtered co-limit of morphisms, hence
$$\Hom_{\Stk_{/S}^{n\mdef\Art, \fp, \mathcal P}}(\mstack X, S) \simeq \indlim[k]\Hom_{\Stk_{/S_i}^{n\mdef\Art, \fp, \mathcal P}}(\mstack X_i\times_{S_j} S_i, S_i).$$
Since the left hand side is non-empty by assumption, the right hand side also must be nonempty for some $i$, i.e. there exists $i\in I$ such that $\pi_i\colon \mstack X_i \to S_i$ is in $\mathcal P$.
\end{proof}
\end{cor}

\subsection{Cohomologically proper stacks}
\label{sec:Cohomologically proper stacks}
In most examples for which we are able to construct a Hodge-proper spreading, the spreading in fact satisfies a stronger property, namely it is \textit{cohomologically proper}. This property enjoys many natural properties that Hodge-properness does not: e.g. it translates along proper maps and a cohomologically proper scheme is necessarily proper. To introduce it we first need to extend \Cref{def:ncoh} to all locally Noetherian Artin stacks:
\begin{defn}\label{def:Noetherian_Artin}
An Artin stack is called \emph{locally Noetherian} if it admits an atlas $\coprod_i U_i$, where all $U_i$ are Noetherian affine schemes. An Artin stack is called \emph{Noetherian} if it is locally Noetherian and quasi-compact quasi-separated.

For a locally Noetherian Artin stack $\mstack X$ we will denote by $\Coh(\mstack X)$ (resp. $\! \Coh^+(X)$) the full subcategory of $\QCoh(\mstack X)$ consisting of sheaves $\mathcal F$ such that the restriction of $\mathcal F$ to some (equivalently to any) locally Noetherian atlas has bounded (resp. bounded below) coherent cohomology sheaves.
\end{defn}
\begin{defn}\label{def:cohomologically proper morphism}
A quasi-compact quasi-separated morphism $f\colon \mstack X \to \mstack Y$ of locally Noetherian Artin stacks is called \emph{cohomologically proper} if the induced functor $f_*\colon \QCoh(\mstack X) \to \QCoh(\mstack Y)$ preserves the full subcategory of bounded below coherent sheaves. A locally Noetherian Artin stack $\mstack X$ over a Noetherian ring $R$ is called \emph{cohomologically proper} if the structure morphism $\mstack X \to \Spec R$ is cohomologically proper.
\end{defn}
\begin{rem}\label{coh_proper_trivial_rem}
By the left exactness of $f_*$ it is enough to prove that $f_*(\Coh(\mstack X)^\heartsuit) \subset \Coh^+(\mstack Y)$.
\end{rem}
We have the following basic properties of cohomologically proper morphisms:
\begin{prop} \label{properties_of_coh_prop_morps}
In the notations above we have:
\begin{enumerate}
\item The class of cohomologically proper morphism is closed under compositions.

\item Let $f\colon \mstack X \to \mstack Y$ be a cohomologically proper morphism and assume that $\mstack X$ is Noetherian. Then for any open quasi-compact embedding $\mstack U \inj \mstack Y$ the pullback $\mstack U \times_{\mstack Y} \mstack X$ is cohomologically proper over $\mstack U$.

\item Let $f\colon \mstack X \to \mstack Y$ be a quasi-compact quasi-separated morphism such that for some smooth cover $\pi\colon \mstack U\ra \mstack Y$ the pull-back $f_{\mstack U}\colon \mstack X\times_{\mstack Y} \mstack U \to \mstack U$ is cohomologically proper. Then $f$ is cohomologically proper.
\end{enumerate}

\begin{proof}
The first point is obvious. To prove the second one note that by base change it is enough to show that any coherent sheaf on $\mstack U\times_{\mstack Y} \mstack X$ is a retract of a restriction of a coherent sheaf on $\mstack X$. This is proved in \Cref{cor:coherent_exted_from_open} below. The third point follows by base change as well since it is enough to check that a sheaf belongs to $\Coh^+$ on a smooth cover.
\end{proof}
\end{prop}

\begin{prop}\label{compact_generation_of_bounded_qcoh}
Let $\mstack X$ be a Noetherian Artin stack. Then for all $n\in \mathbb Z$ the category $\QCoh(\mstack X)^{\ge n}$ is compactly generated by $\Coh(\mstack X)^{\ge n}$.

\begin{proof}
The shift functor $\mathcal F \mapsto \mathcal F[n]$ induces an equivalence $\QCoh(\mstack X)^{\ge n} \simeq \QCoh(\mstack X)^{\ge 0}$, hence without loss of generality we can assume that $n=0$. During the proof we will freely use the fact that the truncation functors for the natural $t$-structure on $\QCoh(\mstack X)$ preserve filtered colimits (see e.g. \cite[Chapter 3.3,
Corollary 1.5.7]{GaitsRozI}).

We first prove that $\Coh(\mstack X)^{[0;m]}$ is compact in $\QCoh(\mstack X)^{[0;m]}$ for all $m\ge 0$. Let $U_\bullet$ be an affine Noetherian smooth hypercover of $\mstack X$. Since $\QCoh(\mstack X)^{[0; m]}$ is an $(m+1)$-category we then have
$$\QCoh(\mstack X)^{[0; m]} \simeq \Tot^{\le m+2} \QCoh(U_\bullet)^{[0; m]}.$$
Since $\Delta_{\le m+2}$ is a finite diagram it follows that a sheaf in $\Coh(\mstack X)^{[0;m]}$ is compact in $\QCoh(\mstack X)^{[0;m]}$, since all of its images are compact in $\QCoh(U_i)^{[0;m]}$. Note also that since for any $\mathcal F \in \QCoh(\mstack X)^{\le m}$ and $\mathcal G\in \QCoh(\mstack X)$ we have
$$\Hom_{\QCoh(\mstack X)^{\ge 0}}(\mathcal F, \mathcal G) \simeq \Hom_{\QCoh(\mstack X)}(\mathcal F, \tau^{\le m}\mathcal G)$$
and since truncation functor $\tau^{\le m}$ preserves filtered colimits, it follows that $\mathcal F \in \Coh(\mstack X)^{[0; m]}$ is compact in $\QCoh(\mstack X)^{\ge 0}$ as well.

Next we show that $\QCoh(\mstack X)^\heartsuit \simeq \Ind(\Coh(\mstack X)^\heartsuit)$. The argument is a slight variation of \cite[Tag 07TU]{StacksProj}. By assumption on $\mstack X$ there exists an affine Noetherian atlas $p\colon U\surj \mstack X$. Let $\mathcal F \in \QCoh(\mstack X)^\heartsuit$ and write $p^*\mathcal F \simeq \indlim \mathcal G_\alpha$, where the diagram on the right runs over all finitely generated submodules of $p^*\mathcal F$. For each $\alpha$ define $\mathcal F_\alpha \in \QCoh(\mstack X)$ as a pullback
$$\xymatrix{
\mathcal F_\alpha \ar@{^(->}[r]\ar[d] & \mathcal F \ar[d] \\
\mathcal H^0 p_*\mathcal G_\alpha \ar@{^(->}[r] & \mathcal H^0 p_*p^*\mathcal F.
}$$
Using triangular identities one easily checks that the inclusion $p^*\mathcal F_\alpha \inj p^*\mathcal F$ factors through an inclusion $\mathcal G_\alpha \inj p^*\mathcal F$. In particular, $p^*\mathcal F_\alpha$, being a submodule of a finitely generated module $\mathcal G_\alpha$ over Noetherian ring $\Gamma(U, \mathcal O_U)$, is finitely generated itself. By definition it means that $\mathcal F_\alpha$ is coherent. Finally, since $p$ is quasi-compact quasi-separated, the pushforward functor $\mathcal H^0p_*$ preserves filtered colimits, hence the natural map $\indlim \mathcal F_\alpha \to \mathcal F$ is an isomorphism.

Let now $i\colon \Ind(\Coh(\mstack X)^{\ge 0}) \to \QCoh(\mstack X)^{\ge 0}$ be a natural functor. Note that since by the previous $\Coh(\mstack X)^{\ge 0}$ is compact in $\QCoh(\mstack X)^{\ge 0}$ this functor is fully faithful. Moreover, since $i$ preserves colimits it admits a right adjoint $R$. Then to prove that $i$ is essentially surjective it is enough to show that the fiber $\mathcal G$ of the co-unit $iR \mathcal F \to \mathcal F$ vanishes. But $R$ being right adjoint preserves fibered products, hence $R \mathcal G \simeq \fib(RiR\mathcal F \to R\mathcal F) \simeq 0$. By Yoneda's lemma and adjunction $i \dashv R$ we conclude that $\Hom_{\QCoh(\mstack X)^{\ge 0}}(\mathcal H, \mathcal G) \simeq *$ for all $\mathcal H \in \Coh(\mstack X)^{\ge 0}$. We claim that $\mathcal G \simeq 0$. To see this assume that $\mathcal G \not\simeq 0$ and let $i$ be the smallest integer such that $\mathcal H^i(\mathcal G) \not \simeq 0$. By the previous part there exists a coherent subsheaf $\mathcal H \subseteq \mathcal H^i(\mathcal G)$. It follows the composition $\mathcal H[-i] \to \mathcal H^i(\mathcal G)[-i] \to \mathcal G$ is non-zero, a contradiction.
\end{proof}
\end{prop}
\begin{cor}\label{cor:coherent_exted_from_open}
Let $\mstack X$ be a Noetherian Artin stack and let $j\colon \mstack U \to \mstack X$ be an open embedding. Then every coherent sheaf on $\mstack U$ is a retract of a restriction of a coherent sheaf from $\mstack X$.

\begin{proof}
Note that since $\mstack X$ is Noetherian, the stack $\mstack U$ is also Noetherian. In particular the embedding $j$ is quasi-compact and quasi-separated. Next, by pulling back to an atlas and using base change (which holds by qcqs assertion about $j$), one finds that the co-unit of adjunction $j^*j_*\mathcal F \to \mathcal F$ is an equivalence for any quasi-coherent sheaf on $\mstack U$. Let now $\mathcal F \in \Coh(\mstack U)$. By the previous proposition $j_*\mathcal F \simeq \indlim \mathcal G_\alpha$ for some filtered diagram of coherent sheaves $\mathcal G_\alpha$. It follows that $\mathcal F \simeq \indlim j^*\mathcal G_\alpha$. By compactness of $\mathcal F$ we conclude that it is a retract of some $j^*\mathcal G_\alpha$.
\end{proof}
\end{cor}

If $R$ is regular, the cohomological properness is stronger than the Hodge-properness:
\begin{prop}\label{coh_proper_stronger}
Let $\mstack X$ be a smooth cohomologically proper Artin stack over a regular Noetherian ring $R$. Then $\mstack X$ is Hodge-proper over $R$.

\begin{proof}
Since $\mstack X$ is smooth over a regular Noetherian ring $R$, the category of coherent sheaves on $\mstack X$ coincides with the category of perfect complexes. So by assumption, it is enough to prove that $\wedge^i \mathbb L_{\mstack X/R}$ is perfect for all $i\ge 0$. By smoothness, the cotangent complex $\mathbb L_{\mstack X/R}$ is perfect and concentrated in non-negative cohomological degrees. It follows that $\mathbb L_{\mstack X/R}$ admits a finite filtration with the associated graded pieces being negative shifts of vector bundles. Hence by induction it is enough to prove that if $E$ is a quasi-coherent sheaf on $\mstack X$ such that $\wedge^{j} E$ is perfect for $j \le i$, then $\wedge^i(E[-1])$ is also perfect. But by construction (see \cite[Theorem 3.35]{MathewBrantner_LLambda}) the functor $\wedge^i$ is $i$-excisive, so $\wedge^i(E[-1])$ is a finite limit of sheaves of the form $\wedge^i(E^{\oplus n}), n\le i$, hence is perfect.
\end{proof}
\end{prop}

Moreover, all proper morphisms are cohomologically proper. To show this, let's first recall the notion of a proper morphisms between higher stacks (following \cite[Section 4]{PortaYu_StacksGAGA}):
\begin{defn}\label{defn:proper morphism}
A $0$-representable morphism $\mstack X \to \mstack Y$ is called \emph{proper} if for any affine scheme $S$ mapping to $\mstack Y$, the pullback $\mstack X\times_{\mstack Y} S$ is a proper $S$-scheme. Next, assuming that the notion of a proper $(n-1)$-representable morphism is already defined, an $n$-representable morphism $f\colon \mstack X \to \mstack Y$ is called \emph{proper} if
\begin{itemize}
\item $f$ is \emph{separated}, i.e. the diagonal map $\mstack X \to \mstack X \times_{\mstack Y} \mstack X$ (which is $(n-1)$-representable) is proper.

\item For any affine scheme $S$ mapping to $\mstack Y$ the pullback $\mstack X_S := \mstack X \times_{\mstack Y} S$ admits a surjective $S$-morphism $P \surj \mstack X_S$ such that $P$ is a proper $S$-scheme.
\end{itemize}

\end{defn}
\begin{rem}
Since the property of a morphism of schemes to be proper is flat local on the target, it is enough in the definition above to check the second condition only for some atlas of $\mstack Y$.
\end{rem}
\begin{rem}
A potentially more familiar definition of a (classical) proper algebraic stack $p:\mstack X\ra S$ is that $p$ should be separated, finite type and universally closed. We note that such stacks over $S$ are proper 1-Artin stacks in the definition above. Indeed, by \cite[Theorem 1.1]{Olsson} in this case there exists a proper surjective map $U\ra \mstack X$ from a proper scheme $U$.
\end{rem}
From the standard results about proper morphisms of schemes and representable morphisms of stacks one formally deduces:
\begin{prop}
With the notations above:
\begin{enumerate}

\item Proper morphism are closed under base change.

\item The property of being a proper morphism is flat local on the target.

\item Proper morphisms are also closed under compositions.
\end{enumerate}
\end{prop}

The fact that proper morphisms are cohomologically proper was proved in \cite[Theorem 5.13]{PortaYu_StacksGAGA}, but in a slightly different context. Their proof essentially follows the argument of \cite[Theorem 15.6]{MoretLaumon} in the case of classical proper stacks. For the reader's convenience we sketch the argument here:
\begin{prop}\label{proper_are_Hodge_proper}
Let $f\colon \mstack X \to \mstack Y$ be a proper morphism between locally Noetherian Artin stacks. Then $f$ is cohomologically proper.

\begin{proof}[Sketch of the proof]
The question is local on the target, hence we can assume that $\mstack Y = Y$ is an affine Noetherian scheme. Moreover, by localizing further if necessary, we can assume that there exists a surjective map $\pi\colon P\surj \mstack X$ such that $P$ is a proper scheme over $Y$. Let us also assume that $\mstack X$ is $n$-Artin for some $n\ge 0$ and let us prove the statement by induction on $n$. The statement for the $n=0$ is the fundamental result about the direct image of a coherent sheaf under a proper morphism of schemes \cite[Chapter III, Theorem 3.2.1]{EGA_III1}.

By \Cref{coh_proper_trivial_rem} it is enough to prove that $f_*(\Coh^\heartsuit(\mstack X)) \subset \Coh^+(Y)$. Let $\mathcal F \in \Coh^\heartsuit(\mstack X)$. Since $\mstack X$ is proper over a Noetherian base, it is Noetherian itself. It follows that there exists a finite filtration (by power of nil-radical of $\mathcal O_{\mstack X}$) of $\mathcal F$ with  the associated graded pieces coming from $\mstack X^{\red}$. Since $\Coh^+(Y)$ is closed under finite extensions, it follows that we can assume that both $\mstack X$ and $Y$ are reduced.

Let us denote $\Tot \pi_{\bullet, *}(\mathcal H^0(\pi_\bullet^*\mathcal F))$ by $\mathcal F^\prime$, where $\pi_\bullet \colon P_{\bullet} \to \mstack X$ is the \v Cech nerve of the map $P \to \mstack X$. By the higher "generic flatness" \cite[Theorem 8.3]{PortaYu_StacksGAGA} there exists an open dense substack $\mstack U$ of $\mstack X$ such that the induced map $P_{\mstack U} := P\times_{\mstack X} \mstack U \to \mstack U$ is flat. In particular, $\pi_{\mathcal U, n}^*(\mathcal F) \simeq \mathcal H^0(\pi_{\mathcal U, n}^*(\mathcal F))$ for all $n\in \mathbb Z_{\ge 0}$. It follows by flat descent that the natural map $\mathcal F \to \mathcal F^\prime$ becomes an equivalence when restricted to $\mstack U$. By Noetherian induction we can assume that $f_*(\fib(\mathcal F \to \mathcal F^\prime))$ lies in $\Coh^+(Y)$. So to prove that $\pi_*(\mathcal F)$ lies in $\Coh^+(Y)$ it is enough to show that $\pi_*(\mathcal F^\prime) \in \Coh^+(Y)$. On the other hand, all elements $P_n$ of the \v Cech nerve are $(n-1)$-Artin proper stacks over $Y$ and all sheaves $\mathcal H^0(p_n^*(\mathcal F))$ are coherent. Since the global section functors $f_{n,*}\colon \QCoh(P_n) \to \QCoh(Y)$ are right $t$-exact, it follows by induction and \Cref{acoh_basics} that the totalization
$$f_*(\mathcal F^\prime) \simeq \Tot f_{\bullet,*}(\mathcal H^0(p_n^*\mathcal F))$$
lies in $\Coh^+(Y)$.
\end{proof}
\end{prop}
\begin{cor}
Let $\mstack X$ be a smooth proper Artin stack over a regular Noetherian ring $R$. Then $\mstack X$ is Hodge-proper.

\begin{proof}
Follows immediately from the previous proposition and \Cref{coh_proper_stronger}.
\end{proof}
\end{cor}

Finally, we record the following observation, which allows to construct new examples of cohomologically proper stacks in inductive way.
\begin{prop}\label{coh_prop_hypercover}
Let $\pi_\bullet\colon \mstack U_\bullet \to \mstack X$ be a flat hypercover such that all $\mstack U_n$ are cohomologically proper over a Noetherian base ring $R$. Then $\mstack X$ is cohomologically proper over $R$.

\begin{proof}
Let $\mathcal F$ be a coherent sheaf on $\mstack X$. By shifting if necessary we can assume that $\mathcal H^{<0}(\mathcal F) \simeq 0$. By the flat descent
$$R\Gamma(\mstack X, \mathcal F) \simeq \Tot R\Gamma(\mstack U_\bullet, \pi^*_\bullet \mathcal F).$$
Since the global section functors $R\Gamma(\mstack U_n, -)$ are right $t$-exact and by assumptions on $\mstack U_n$ the diagram $R\Gamma(\mstack U_\bullet, \pi^*_\bullet \mathcal F)$ consists of coconective bounded below coherent complexes. By \Cref{acoh_basics} the complex $R\Gamma(\mstack X, \mathcal F)$ is also bounded below coherent.
\end{proof}
\end{prop}

\subsection{Examples of Hodge-properly spreadable stacks}\label{sec:Examples of spreadable Hodge-proper stacks}
In this subsection we begin to study which Hodge-proper stacks in characteristic 0 admit a Hodge-proper spreading over some finitely generated $\mathbb Z$-algebra. We will make extensive use of \Cref{fp_over_filtered_limit} in the following situation: let $F$ be an algebraically closed field of characteristic $0$, then $\Spec F\simeq \prolim R$ where $R\subset F$ runs through subrings of $F$ that are smooth over $\mathbb Z$. This diagram is filtered since for any two such subalgebras $R_1,R_2\subset F$ some finite localization $(R_1\cdot R_2)[1/f]$ of their composite in $F$ is again smooth over $\mbb Z$. In particular we have an equivalence 

$$\indlim[R\subset F] \Stk_{/R}^{n\mdef\Art,\fp, \mathcal P} \xymatrix{\ar[r]^\sim &} \Stk_{/F}^{n\mdef\Art,\fp, \mathcal P}$$ 
for any spreadable class $\mc P$. Another option is to also allow those localizations of smooth $\mathbb Z$-algebras for which the image of $\Spec R$ in $\Spec \mathbb Z$ is open (as in \Cref{def: spredable_stacks}). By \Cref{rem: cofilteredness of all such R in F} the diagram of all such $R$ is again filtered and so \Cref{fp_over_filtered_limit} can be applied. In \Cref{sec:non-proper schemes} we will see that allowing these localizations actually makes a difference.
In what follows $F$ will always denote an algebraically closed field of characteristic 0 and we will pick $R\subset F$ to be a smooth $\mathbb Z$-subalgebra (except \Cref{sec:non-proper schemes}, where it will be an infinite localization of one) of $F$. We also freely use the standard spreading out results for schemes (\Cref{thm:spreading of schemes}) and their easy consequences (like spreading out group schemes, group actions, group homomorphisms, closed subgroups, etc.) without any additional reference.

We start with the Hodge-proper spreadability for proper Artin stacks, which is  deduced from the spreadability of proper morphisms (\Cref{fp_over_filtered_limit}). This is done in \Cref{sec:proper stacks}. Then we discuss in great detail the question of Hodge-proper spreadability of $BG$ in \Cref{sec:BG}; the case of more general quotient stacks is postponed till \Cref{sec:spreadability of quotient stacks}. Finally, in \Cref{sec:non-proper schemes} we try to grasp the scope of potential applications of \Cref{Hodge_degeneration} concentrating on the case of schemes: in fact a particular set of examples given by semiabelian surfaces. 

As was mentioned, often, along with Hodge-proper spreadability, we are able to prove a somewhat stronger statement, saying that the stacks we consider admit a cohomologically proper spreading. For this it is  convenient to introduce the following variant of \Cref{def:cohomologically proper morphism}:
\begin{defn}\label{def:coh_prop_spread}
A morphism  $f\colon \mstack X\ra \mstack Y$ of Artin stacks over a field $F$ of characteristic $0$ is called \emph{cohomologically properly spreadable} if there exists a finitely generated $\mathbb Z$-algebra $R\subset F$\footnote{Or a suitable localization of one, as in \Cref{def: spredable_stacks}.} and a morphism $f_R\colon \mstack X_R\ra \mstack Y_R$ over $\Spec R$, such that 
\begin{itemize}
\item $f_R\otimes_R F:=f_R\times_R F\simeq f$.

\item $f_R\colon \mstack X_R\ra Y_R$ is cohomologically proper (see \Cref{def:cohomologically proper morphism}).
\end{itemize}
\end{defn}
In the case $\mstack Y=\Spec F$ we will call $\mstack X$ \textit{cohomologically properly spreadable}. By \Cref{coh_proper_stronger} any such $\mstack X$ is also Hodge-properly spreadable.

\subsubsection{Proper stacks}\label{sec:proper stacks}
In this subsection we show that all proper stacks are cohomologically (and in particular Hodge-)properly spreadable. By \Cref{proper_are_Hodge_proper} it is just enough to show that proper morphisms spread out. 

Following the convention of \Cref{sec:Spreadable stacks}, for an affine scheme $S$ we denote by $\Stk_{/S}^{n\mdef\Art, \fp, \pr}\subset \Stk_{/S}^{n\mdef\Art, \fp}$ the subcategory consisting of finitely presented $n$-Artin $S$-stacks and with morphisms given by proper maps (see \Cref{defn:proper morphism}).

The results of \Cref{sec:Spreadable stacks} allow to deduce the spreadability of proper morphisms from the analogous statement for classical schemes:
\begin{prop}\label{prop: proper morphisms spread}
Let $\{S_i\}$ be a filtered diagram of affine schemes with a limit $S$. Then the natural functor
$$\indlim[i] \Stk_{/S_i}^{n\mdef\Art, \fp, \pr} \xymatrix{\ar[r] &} \Stk_{/S}^{n\mdef\Art, \fp, \pr}$$
is an equivalence.

\begin{proof}
By \Cref{fp_over_filtered_limit} it is enough to prove that for a proper morphism $f\colon \mstack X \to \mstack Y$ there exists a proper morphism $f_i \colon \mstack X_i \to \mstack Y_i$ such that $f_i\times_{S_i} S \simeq f$. Assume that $f$ is $n$-representable. We will prove the statement by induction on $n$.

For $n=0$ let $U\surj \mstack Y$ be an affine finitely presentable atlas. Then by assumption $\mstack X_U := \mstack X \times_{\mstack Y} U$ is a scheme proper over $U$. By \Cref{fp_over_filtered_limit} we can spread the commutative square
$$\xymatrix{
\mstack X_U \ar[r] \ar[d] & \mstack X \ar[d]^f \\
U \ar@{->>}[r] & \mstack Y
}$$
to some $S_i$. By spreadability of equivalences, smooth surjective morphisms of stacks and proper morphisms of schemes we can assume, taking base change to some $S_j$ if necessary, that the natural map $\mstack X_{U, j} \to U_j \times_{\mstack Y_j} \mstack X_j$ is an equivalence. We can also assume that $U_j \to \mstack Y_j$ is a smooth atlas and that $\mstack X_{U,j}$ is proper scheme over $U_j$. Since the property of a map of schemes to be proper is flat local on target, it follows that for any $T$ mapping to $\mstack Y_j$ the pullback $T\times_{\mstack Y_j} \mstack X_j$ is a proper $T$-scheme. So the map $f_j\colon \mstack X_j \to \mstack Y_j$ is proper.

Finally, assume that the statement for $(n-1)$-representable morphism is already proved. Let $U\surj \mstack Y$ be a smooth finitely presentable atlas and let $P \surj \mstack X_U$ be a surjection from a proper $U$-scheme $P$. Then by the induction assumption we can find a spreading $f_i\colon \mstack X_i \to \mstack Y_i$ such that $f_i$ is separated. By taking base change to some $S_j$ we can assume that $P_j$ is proper over $U_j$ and that the map $P_j \to \mstack X_{U_j}$ is surjective.
\end{proof}
\end{prop}
\begin{cor}\label{cor:proper stacks spread out}
Let $\mstack X$ be a smooth proper stack over a field $F$ of characteristic $0$. Then $\mstack X$ is Hodge-proper and Hodge-properly spreadable.
\end{cor}
\begin{proof} By \Cref{prop: proper morphisms spread} (and \Cref{spreading_with_properties_for_stacks}) applied to $\mstack X \ra \Spec F$ we get a smooth proper spreading $\mstack X_R\ra \Spec R$. Then $\mstack X_R$ is Hodge-proper by Propositions \ref{coh_proper_stronger} and \ref{proper_are_Hodge_proper}.
\end{proof}

We will also use the following corollary:
\begin{cor}\label{chow_lem_for_spreadings}
Let $f\colon \mstack X \to \mstack Y$ be a proper map of finitely presentable Artin stacks over a field $F$ of characteristic $0$ such that $\mstack Y$ is cohomologically properly spreadable (see \Cref{def:coh_prop_spread}). Then $\mstack X$ is also cohomologically properly spreadable.

\begin{proof}
Let $f_R\colon \mstack X_R \to \mstack Y_R$ be some proper spreading of $f$. Since any two spreadings become equivalent after some finite localization of $R$, we can assume that $\mstack Y_R$ is cohomologically proper over $R$. Then we conclude by \Cref{proper_are_Hodge_proper} and the first part of \Cref{properties_of_coh_prop_morps}. 
\end{proof}
\end{cor}

\begin{rem}
More generally, a composition of two cohomologically properly spreadable morphisms is again cohomologically properly spreadable.
\end{rem}

\subsubsection{Classifying stacks} \label{sec:BG}
Let $F$ be an algebraically closed field of characteristic zero. We start our investigation of Hodge-proper spredability by first understanding for which algebraic groups $G$ over $F$ the classifying stack $BG$ is Hodge-proper. The answer turns out to be easy: for all finite type $G$. In fact $BG$ is even cohomologically proper:
\begin{prop}\label{BG}
Let $G$ be a finite type group scheme over $F$. Then $BG$ is cohomologically (and, in particular, Hodge-)proper.

\begin{proof}
In fact we will show a more precise statement, namely that for any coherent sheaf $\mc F$ on $BG$, $\RG(BG,\mc F)$ lies in $\Coh(F)$ if $G$ is linear and in $\Coh^+(F)$ if $G$ is general.

Note that, since we are in characteristic $0$, $G$ is smooth and thus so is the natural map $\Spec F\ra BG$\footnote{Note that since the property of a morphism to be smooth can be checked flat locally on the source the structure map $BG \to \Spec F$ is always smooth even when $G$ is not. We refer interested reader to \cite{Toen_AutoFlat} or \cite[Tag 0DLS]{StacksProj} for more details.}. In particular, the natural $t$-structure on $\Coh(BG)$ coincides with the usual $t$-structure on $\Coh(F)$ after taking pullback $\Coh(BG) \ra \Coh (F)$ (aka the forgetful functor in terms of representations). It is enough to show the statement for $\mc F\in \Coh(BG)^\heartsuit$ (\Cref{coh_proper_trivial_rem}). Note that such $\mc F$ is the same thing as a finite-dimensional algebraic representation of $G$ over $F$. 

By Chevalley's structure theorem there is an exact sequence $1\ra L\ra G \ra{} A\ra 1$ where $L$ is a linear algebraic group and $A$ is proper. Then for $L$ we have another short exact sequence 
$$
1\ra U\ra L \ra H \ra 1,
$$
where $U$ is the unipotent radical of $L$ and $H\simeq L/U$ is reductive.

Let $j\colon BU\ra BL$, $f\colon BL\ra BH$, $i\colon BL\ra BG$ and $p\colon BG\ra BA$ be the corresponding maps between classifying stacks. We will prove the statement step by step, starting from the unipotent case.

\clause{Case 1. $G=U$ is unipotent.} We assume $\mc F\in \Coh(BU)^{\heartsuit}$. Since the characteristic of $F$ is 0 and $U$ is unipotent, $\RG(BU,\mc F)$ can be computed as the cohomology of the Lie algebra $\mf u$. Explicitly, this is given by the Chevalley complex:
$$
0\ra \mc F\ra \mc F\otimes \mf u^* \ra \mc F\otimes \wedge^2\mf u^*\ra \ldots \mc F\otimes \wedge^{\dim U}\!  \mf u^*\ra 0.
$$
Since $\mc F$ is finite dimensional this complex is clearly perfect.

\clause{Case 2. $G=H$ is reductive.} This follows from the fact that the abelian category $\Rep(H)$ is semi-simple (since $\mathrm{char}(F) = 0$). Namely for $\mc F\in \Coh(BH)^{\heartsuit}$, the complex $\RG(BH,\mc F)$ is equal to the $H$-invariants $\mc F^H$ (in cohomological degree 0). Since $\mc F$ is finite-dimensional we get $\RG(BH,\mc F)\in \Coh(F)$.

\clause{Case 3. $G=A$ is proper.} Let $\mc F\in \Coh(BA)$. We can compute $\RG(BA,\mc F)$ using the smooth $q\colon \Spec F \ra BA$. Let $p_n\colon A^n\ra BA$ be the map from the $n$-th term of the associated {\v C}ech simplicial object. We get a cosimplicial object 
$$
[n]\mapsto \RG(A^n,p_n^*\mc F),
$$
in $\Mod_F$, and
$$
\RG(BA,\mc F)\simeq \Tot \RG(A^\bullet,p_\bullet^*\mc F).
$$
However, each term $\RG(A^n,p_n^*\mc F)$ lies in $\Coh(F)$ (since $A^n$ is proper) and has cohomology only in non-negative degrees. By \Cref{acoh_basics} it follows that 
$\RG(BA, \mathcal F)$ lies in $\Coh^+(F)$.

\clause{Case 4. $G=L$ is linear.} We assume $\mc F\in \Coh(BU)^{\heartsuit}$ and consider $f_*\mc F\in \QCoh(BH)$ (for $f\colon BL\ra BH$). We claim that $f_*\mc F\in \Coh(BH)$. It is enough to check that after taking pull-back to the smooth cover $q\colon \Spec F\ra BH$. We have a fibered square
$$
\xymatrix{BU \ar[r]^j\ar[d] & BL\ar[d]^f\\
	\Spec F \ar[r]^q & BH
}
$$
and by base change we have $q^*f_*\mc F\simeq \RG(BU,j^*\mc F)$. The map $j$ is flat, so $j^*\mc F$ is coherent and thus $\RG(BU,j^*\mc F)\in \Coh(F)$ by Case~$1$. It follows that $f_*\mc F\in \Coh(BH)$. But then $\RG(BL,\mc F)\simeq \RG(BH,f_*\mc F)$ and we are done by Case 2. At this point we have the statement for $G$ linear.

\clause{Case 5. $G$ is general.} The argument in Case 4 works here as well, replacing $U$ with $L$ and $H$ with $A$. Namely $p_*\mc F\in \Coh(BA)$ and then by Case $3$
$$
\RG(BL,\mc F)\simeq \RG(BH,f_*\mc F).\qedhere
$$
\end{proof}
\end{prop}

Even though $BG$ is Hodge-proper practically for any $G$, there are definitely some algebraic groups $G$ for which $BG$ is not Hodge-properly spreadable. Indeed, consider $G=\mathbb G_a$. If $B\mathbb G_a$ were Hodge-properly spreadable, then by \Cref{cor:Hodge decomposition for singular cohomology} we would get a decomposition 
$$
H^n_\dR(B\mathbb G_a/F)\simeq \bigoplus_{p+q=n} H^q(B\mathbb G_a, \wedge^p \mathbb L_{B\mathbb G_a/F}).
$$
However, this is impossible. Indeed, the left hand side vanishes for $n>0$ by the $\mathbb A^1$-homotopy invariance of the de Rham cohomology in characteristic $0$. On the other hand $\wedge^p \mathbb L_{B\mathbb G_a/F} \simeq \mc O_{B\mathbb G_a}[-p]$ and $H^i(B\mathbb G_a,\mc O_{B\mathbb G_a})$ is non-zero for $i=0,1$. In particular, the right hand side is non-zero for all $n$, contradiction.

Note that by \Cref{Hodge_degeneration} it follows that the Hodge cohomology of any spreading of $B\mathbb G_a$ \emph{has to be} infinitely generated. This is also confirmed by an explicit computation of the cohomology of $\mc O_{B\mbb G_a}$ over $\mbb Z$ which the reader can find in \Cref{appendix:BG_a}. We only slightly comment on this here:
\begin{ex}\label{ex: comment on BG_a}
	Let $\mbb G_a$ be the additive group considered as an algebraic group scheme over $\mbb Z$. By the computation in \Cref{appendix:BG_a} one has an embedding 
	$$
	\xymatrix{\left(\mathbb Z[v_1] \otimes_{\mathbb Z}\Sym_{\mathbb Z}^*\left(\bigoplus_{p}\mathbb F_p v_p\oplus \mathbb F_pv_{p^2}\oplus \ldots\right)\right) \bigg/ v_1^2=v_2 \  \ar@{^{(}->}[r]& \  H^*(B\mathbb G_a,\mc O_{B\mathbb G_a})},
	$$
	where $v_1$ has cohomological degree 1 and all other $v_{p^i}$ are of degree 2. In particular $H^2(B\mbb G_a, \mc O_{B\mathbb G_a})$ has infinitely generated elementary $p$-torsion for any prime $p$. Given any spreading $\mstack X$ of ${(B\mbb G_a)}_F$ over some $R\subset F$, by \Cref{fp_over_filtered_limit} it becomes isomorphic to ${(B\mbb G_a)}_R$ for some larger $R$. Choosing prime $p$ in the image of $
\Spec R$ in $\Spec \mbb Z$, by flat base change we get that $H^2(\mstack X, \mc O_{\mstack X})$ contains an infinite sum $(R/p)^{\oplus \mbb N}$ and thus $\mstack X$ is not Hodge-proper over $R$.
\end{ex}	
\noindent Given the complexity of $B\mathbb G_a$ from cohomological point of view, it is natural to ask for which algebraic groups $G$, the classifying stack $BG$ is Hodge-properly  spreadable. We provide a list of examples:

\begin{ex} \label{ex: examples of Hodge-spreadable stacks} $BG$ is Hodge-properly (and in fact also cohomologically properly) spreadable if
\begin{itemize} 
	\item $G$ is proper (=an extension of a finite group by an abelian variety). Then $BG$ is a proper stack and this is covered by \Cref{cor:proper stacks spread out};
	\item $G$ is reductive. This follows from \Cref{prop: Y/G --> Y//G is cohomologically spreadable} if we take $Y=\Spec F$;
	\item $G\!=\! P\subset H$ is some parabolic subgroup of some reductive group $H$. This is a particular case of \Cref{thm: conical examples}. Alternatively, it follows from the previous point and \Cref{chow_lem_for_spreadings} (using that $BP\ra BH$ is proper).
\end{itemize}

\end{ex}
\begin{rem}\label{rem: extension of abelian by reductive}
	By an argument similar to \Cref{BG} it is also possible to show the spreadability of $BG$ for an extension of an abelian variety by a parabolic subgroup of some reductive group.
\end{rem}

The fact that $BP$ is cohomologically properly spreadable can look a little surprising and we would like to illustrate what happens by the simplest non-trivial example, a Borel subgroup $B\subset \SL_2$:
\begin{ex}\label{ex:BB}
	Let $G=B\subset \SL_2$ be the standard Borel subgroup of $\SL_2$ over $\mathbb Z$, namely
	$$
	B = \left\lbrace
	\begin{pmatrix}t & s \\ 0 & t^{-1} \end{pmatrix}
	\right\rbrace \subset \SL_2.
	$$ 
	In this example we will show that $BB$ is a cohomologically proper spreading of $BB_F$.
	
	Note that $B\simeq \mathbb G_a \rtimes \mathbb G_m$ with $\mathbb G_m=\Spec \mathbb Z[t,t^{-1}]$ acting on $\mathbb G_a= \Spec \mathbb Z[x]$ by multiplication of $x$ by $t^2$. 
	Consider the natural map $p\colon BB\ra B\mathbb G_m$ and take $p_*(\mc O_{BB})$. We have a fiber square
	$$
	\xymatrix{B\mathbb G_a \ar[r]^j \ar[d] & BB \ar[d]\\
		\mr{pt}\ar[r]^q & B\mathbb G_m.
	}
	$$
	We have $j^*\mc O_{BB}\simeq \mc O_{B\mathbb G_a}$ and by base change the underlying complex $q^*p_*\mc O_{BB}$ is equal to $\RG(B\mathbb G_a,\mc O_{B\mathbb G_a})$. It follows that 
	$$
	\RG(BB,\mc O_{BB})\simeq \RG(B\mathbb G_m,p_*\mc O_{BB})\simeq 
	\RG(B\mathbb G_a,\mc O_{B\mathbb G_a})^{\mathbb G_m},$$ since $\mathbb G_m$-invariants is an exact functor. In the terms of the computation in \Cref{appendix:BG_a} this corresponds to the 0-th graded component $\RG(B\mathbb G_a,\mc O_{B\mathbb G_a})_0\subset \RG(B\mathbb G_a,\mc O_{B\mathbb G_a})$, which, as we figure there, is just given by $\mathbb Z$.
	 Consequently $\RG(BB,\mc O_{BB})=\mathbb Z$. 
	
	Summarizing, we can see that even though the cohomology of $\mc O_{B\mathbb G_a}$ is enormous, the $\mathbb G_m$-action contracts it, ultimately making the cohomology of $\mc O_{BB}$ finitely generated. 
	
	In fact more is true: namely for any Borel subgroup $B$ of a split reductive group $G$ over $\mbb Z$ the stack $BB$ is cohomologically proper. Indeed, one first shows that $BG$ is cohomologically proper over $\mbb Z$ by applying \Cref{Franjou} in the case $A=R=\mbb Z$ and, since the morphism $BB\ra BG$ is proper, the rest follows from \Cref{proper_are_Hodge_proper}. 
\end{ex}

\subsubsection{Non-proper schemes}\label{sec:non-proper schemes}
Considering schemes, it is natural to ask whether a Hodge-properly spreadable scheme is necessarily proper. On the other extreme, one can ask whether any schematic example of the Hodge-to-de Rham degeneration is Hodge-properly spreadable. Below we consider an example of a semiabelian surface $X$ given by an extension of an elliptic curve by $\mathbb G_m$; as we will see, appropriate choices of extensions give counterexamples to both statements above.
\begin{ex}\label{ex:key}
Let $E$ be an elliptic curve over a field $k$. Let $K/k$ be a (not necessarily algebraic) field extension and let $\mc L \in \Pic^0(E)(K)\simeq E(K)$ be a degree $0$ line bundle on $E_{K}$. Let $X$ be the total space of the associated $\mathbb G_m$-torsor. The $K$-scheme $X$ is clearly smooth and non-proper; moreover, by \cite[VII.3.16]{Serre_GACC}, $X$ is in fact an algebraic group, more concretely a semiabelian surface. 

\begin{lem}\label{Hodge-proper when non-torsion}
$X$ is Hodge-proper over $K$ if and only if $\mc L^{\otimes n}\neq \mc O_X$ for all $n>0$. If $\mr{char} \ \! K=0$ this is also equivalent to the degeneration of the Hodge-to-de Rham spectral sequence for $X
$.

\begin{proof}
Note that $\Omega_{X}^1 \simeq \mathcal O_{X}^{\oplus 2}$, since $X$ is a group. Denote the natural projection $X \to E_K$ by $\pi$. We find
$$R\pi_*\mathcal O_{X} \simeq \pi_* \mathcal O_{X} \simeq \bigoplus_{n\in \mathbb Z} \mathcal L^n.$$
Next, since the degree of $\mc L$ is zero and $\mathcal L\neq \mc O_{X}$, we have $R\Gamma(E_K, \mathcal L) \simeq 0$. If $\mathcal L$ is non-torsion, the same holds for $\mathcal L^n$, for all $n \ne 0$. So
$$R\Gamma(X, \Omega^2_{X}) \simeq R\Gamma(X, \mathcal O_{X}) \simeq R\Gamma(E_K, \mathcal O_{E_K}) \simeq K\oplus K[-1], \quad R\Gamma(X, \Omega^1_{X}) \simeq R\Gamma(X, \mathcal O_{X}^{\oplus 2}) \simeq K^{\oplus 2}\oplus K[-1]^{\oplus 2}.$$
If, on the other hand, $\mathcal L$ is torsion, $\pi_*\mathcal O_{X}$ has infinitely many copies of $\mc O_{E_K}$ as direct summands and so $H^0(X,\mc O_X)$ is infinite-dimensional.

For the second assertion it is enough to consider the case $K=\mathbb C$, where we can compare the de Rham cohomology with the singular one. Since the degree of $\mc L$ is 0, $\mathcal L$ is topologically trivial, and $X$ is homotopy equivalent to $(\mathbb S^1)^{\times 3}$; comparing the dimensions we see that the Hodge-to-de Rham spectral sequence degenerates at the first page.
\end{proof}
\end{lem}

\begin{rem}\label{Hodge-proper over finite field}
Note that if $K$ is a subfield of $\overline{\mathbb F}_p$, then $\Pic^0(E)(K)$ is a torsion abelian group. Thus, in this case $X$ is never Hodge-proper.
\end{rem}

Now let $E$ be an elliptic curve over $\mathbb Q$. Let's consider $\mc L_{\mathbb C}\in \Pic^0(E)(\mathbb C)$; the corresponding semiabelian variety $X_{\mathbb C}$ is a variety over $\mathbb C$. 

\begin{prop}\label{Hodge-spreadability vs rationality}
$X_{\mathbb C}$ is Hodge-properly spreadable if and only if  $\mc L_{\mathbb C}\in \Pic^0(E)({\mathbb C}) \setminus \Pic^0(E)(\ol{\mathbb Q})$.

\begin{proof}
First, let $\mc L_{\mathbb C} \in E(\ol{\mathbb Q})$. Let $K\subset \ol{\mathbb Q}$ be the field of definition of $\mc L_{\mathbb C}$; then $\mc L_{\mathbb C}$ and $E_{\mathbb C}$ are defined over $K$ and we will denote the corresponding line bundle and elliptic curve over $K$ by $\mc L_K$ and $E_K$. We also denote by $X_K$ the total space of $\mc L_K$. Let $\mc O_K\subset K$ be the ring of integers. Consider the filtered system $\{\mc O_K[1/n]\}_{n\in \mathbb N}$ of subrings of $K$; we have $K= \colim_n \mc O_K[1/n]$. By \Cref{spreading_with_properties_for_schemes} $X_K$ has a smooth spreading $X_A$ over $A=\mc O_K[1/n]$ for $n$ big enough and, taking a larger $n$, we can assume that $X_A$ is the total space of a line bundle $\mc L_A$ over an elliptic curve $E_A$ (with $E_A$ and $\mc L_A$ being spreadings of $E$ and $\mc L$). Note that all closed points of $\Spec A$ are of positive characteristic and have finite residue fields. Localizing $A$ further we can assume that base change for Hodge and de Rham cohomology holds with respect to all closed points of $A$. Let $x \colon \Spec \mathbb F_q \inj \Spec A$ be some closed point. By \Cref{Hodge-proper over finite field} the reduction $\mc L_{\mathbb F_q}$ is torsion and thus $X_{\mathbb F_q}$ is not Hodge-proper. It follows that $X_A$ is not Hodge-proper and thus, since any subring $R\subset K$ satisfying the conditions in \Cref{def: spredable_stacks} is contained in $\mc O_K[1/n]$ for some $n$, $X$ is not Hodge-properly spreadable. We claim that neither is $X_{\mathbb C}$; indeed, let $X_R'$ be a Hodge-proper spreading over some localization $R\subset \mathbb C$ of some finitely generated $\mathbb Z$-algebra as in \Cref{def: spredable_stacks}. Since $\mathbb C$ can be represented as a colimit of flat finitely generated $R$-algebras (and those fit in the framework of \Cref{def: spredable_stacks}), by the "spreading out" for schemes we can assume that ${\{(X_R')\}}_{R\cdot K}\simeq \{(X_K)\}_{R\cdot K}$. Considering two systems: ${\{{(X_A)}_{R\cdot A}\}}_{A}$ and ${\{{(X_R')}_{R\cdot A}\}}_{A}$ with $X_A$ as above and $A$ running over $\mc O_{K}[\frac{1}{n}]$ for various integers $n$ we get that ${(X_A)}_{R\cdot A}\simeq {(X_R')}_{R\cdot A}$ some $A$. Note that since the image of $\Spec R$ in $\Spec \mathbb Z$ is open and $R\cdot A$ is torsion free (over $\mathbb Z$, and thus also over $A$), $R\cdot A\subset \mathbb C$ becomes faithfully flat over $A=\mc O_K[1/n]$ if we take $n$ big enough. Since $X_A$ is not Hodge-proper, neither is ${(X_A)}_{R\cdot A}$. Replacing $A$ we can also assume $R\cdot A$ is flat over $R$; indeed by "spreading out" of flat morphisms of schemes it is enough to show that $R\cdot K$ is flat over $R_{\mathbb Q}$. But $R\cdot K$ is a direct summand of $R_{\mathbb Q}\otimes_{\mathbb Q} K$, so this is clear. Thus by flat base change (applied to $R\cdot A$ which is now flat over $R$) we get that $X_R'$ can't be Hodge-proper, which is a contradiction. Thus $X_{\mathbb C}$ is not Hodge-properly spreadable.  

It remains to deal with the transcendent case $\mc L_{\mathbb C}\notin \Pic^0(E)(\ol{\mathbb Q})$. Let's consider the universal line bundle $\mc P$ on $E\times\Pic^0(E)$.  Since $\mc L$ is not a $\ol{\mathbb Q}$-point, the corresponding map $\Spec \mathbb C \ra \Pic^0(E)$ factors through the generic point $\Spec \mathbb Q(E)\subset E\simeq \Pic^0(E)$ and thus both $\mc L_{\mathbb C}$ and $X_{\mathbb C}$ are defined over $\mathbb Q(E)$. We denote the corresponding bundle and $\mathbb Q(E)$-scheme by $\mc L$ and $X$. Let $y^2=x^3 +ax+b$ be an equation of (the affine part of) $E$. Let $B=\mathbb Z[1/n][x,y]/(y^2-x^3-ax-b)\subset \mathbb Q(E)$ where $n$ is big enough to be divisible by the denominators of both $a$ and $b$, and so that $B$ is smooth over $\Spec \mathbb Z$. Note that $E$ has a smooth proper model $E_{\mathbb Z[1/n]}$  over $\mathbb Z[1/n]$ given by the projective closure of $\Spec B$. Now let $R=B[S^{-1}]\subset \mathbb Q(E)$ be the localization of $B$ with respect to the set $S$ of elements $s\in B$ that are non-zero modulo all primes $p\in \mathbb Z$ provided $(p,n)=1$.\footnote{More explicitly, one can see that it is enough to invert functions $y^n-1$ for $n\ge 1$.} The reduction $R/(p)$ is equal to the field of fractions of $B/(p)$ which is nothing but $\mathbb F_p(E_{\mathbb F_p})$. Identifying $\Pic^0(E)$ with $E$, we obtain spreadings $\mc L_R$ (over $E_R$) and $X_R$ of $\mc L$ and $X$ (considered as a line bundle on $E_{\mathbb Q(E)}$ and a scheme over $\mathbb Q(E)$ correspondingly). Localizing $R$ we can assume that $X_R$ is a group scheme and thus it is enough to show that the cohomology of $\RG(X_R,\mc O_{X_R})$ is finitely generated over $R$. We have 
$$R\pi_{*}\mathcal O_{X_R} \simeq \pi_{*} \mathcal O_{X_R} \simeq \bigoplus_{n\in \mathbb Z} \mathcal L_R^n.$$
Each $\mathcal L_R^n$ is a coherent sheaf on $E_R$ and $\RG(E_R,\mathcal L_R^n)\in \Mod_{R}^{\mr{coh}}$. We claim that it is zero if $n\neq 0$; note that $R$ is regular, thus $\RG(E_R,\mathcal L_R^n)$ is perfect and so it is enough to check this modulo all primes $p\in \mathbb Z$. But by the construction the reduction $\mathcal L_{R/p}$ is the restriction of the universal line bundle on $E_{\mathbb F_p}\times_{\mathbb F_p} \Pic^0(E_{\mathbb F_p})$ to $E_{\mathbb F_p}\times_{\mathbb F_p} \mathbb \Spec \mathbb F_p(E_{\mathbb F_p})$; in particular it is non-torsion. Thus $\RG(E_{R},\mathcal L_{R}^n)\simeq 0$, and so 
$$\RG(X_R,\mc O_{X_R})\simeq\RG(E_{R},\pi_{*} \mathcal O_{X_R})\simeq \RG(E_{R},\mc O_{E_R})$$ is coherent.
\end{proof}
\end{prop}
\end{ex}

\begin{rem}\label{rem:degenerate but not spreadable}
Due to \Cref{Hodge-spreadability vs rationality} and \Cref{Hodge-proper when non-torsion}, a semiabelian surface $X_{\mathbb C}$, for the choice of a non-torsion  line bundle $\mc L_{\mathbb C}\in \Pic^0(E)(\ol{\mathbb Q})$, gives an example of a scheme which is not Hodge-properly spreadable, but the Hodge-to-de Rham spectral sequence degenerates. On the other hand, taking $\mc L_{\mathbb C}\notin \Pic^0(E)(\ol{\mathbb Q})$ gives an example of a Hodge-properly spreadable scheme which is not proper.
\end{rem}

\begin{rem}
Note that even in the case $\mc L_{\mathbb C}\notin \Pic^0(E)(\ol{\mathbb Q})$, the scheme $X_{\mathbb C}$ does not have a Hodge-proper spreading over a finitely generated $\mathbb Z$-algebra $R$ (because then $\Spec R$ necessarily has points with finite residue fields over which $\mc L$ becomes torsion). Thus it really makes a difference to allow arbitrary localizations of the latter in \Cref{def: spredable_stacks}.
\end{rem}

\begin{rem}
The ring $R$ used in the proof of \Cref{Hodge-spreadability vs rationality} is a slight generalization of a ring, that could be called "quantum integers/rationals". Recall that the "quantum integer $n$" polynomial $[n]_q\in \mathbb Z[q]$ is defined as $[n]_q\coloneqq 1+q+\ldots +q^{n-1}$; we then can consider $\mathbb Q_{\mathbf{q}}\coloneqq\mathbb Z[q][[n]_q^{-1}]_{n\in \mathbb N}$. The ring $\mathbb Q_{\mathbf{q}}$ is a principal ideal domain of Krull dimension $1$ whose reduction modulo a prime $p$ is given by 
$$
\mathbb Q_{\mathbf{q}}\otimes_{\mathbb Z}\mathbb F_p\simeq \mathbb F_p[q][[n]_q^{-1}]_{n\in \mathbb N}\simeq \mathbb F_p(q),
$$
the field of rational functions over $\mathbb F_p$.
\end{rem}

\begin{rem}\label{rem:not pure}
Topologically, $X_{\mathbb C}(\mathbb C)\simeq  \mathbb C^\times\times (\mathbb S^1)^2$, since $\mc L_{\mathbb C}$ has degree 0 and is topologically trivial. The space $H^1_{\mr{sing}}(X_{\mathbb C}(\mathbb C),\mathbb C)\simeq \mathbb C^3$ is odd dimensional and thus the corresponding mixed Hodge structure can't be pure of weight 1. In particular this shows that the mixed Hodge structure\footnote{Appropriately defined, say via Section 8.3 of \cite{HodgeIII}.} on the $n$-th singular cohomology of a Hodge-properly spreadable stack is not necessarily pure of weight $n$.
\end{rem}


\subsubsection{Higher examples}\label{sec: higher examples}
Here we also record some examples of Hodge-properly spreadable stacks that are not classical.
	\begin{ex}
		Let $G$ be a classical abelian algebraic group over $F$ such that $BG$ is cohomologically properly spreadable (e.g. by \Cref{rem: extension of abelian by reductive} and the discussion above it, one can take $G$ to be an algebraic torus, an abelian or even a semiabelian variety). Then the higher stack $K(G,n)\coloneqq B^nG$ (given by the sheafification of $S \mapsto K(G(S),n)$) is also cohomologically properly spreadable. Indeed, after enlarging $R$ any cohomologically proper spreading $K(G,1)_R$ becomes equal to $K(G_R,1)$ for some spreading $G_R$ of $G$ which has a group structure. Taking the $K(G_R,n)$ for such $G_R$ gives a spreading of $K(G,n)$ which, we claim, is cohomologically proper. We will show this by induction. Consider the \v Cech simplicial object $\mstack U_{\bullet}$ corresponding to the cover $\Spec R\ra K(G_R,n)$. All of its terms $\mstack U_k$ are given by products $K(G_R,n-1)^k$, the projection morphism $K(G_R,n-1)^k\ra K(G_R,n-1)^{k-1}$ is cohomologically proper\footnote{Indeed, its pull-back to $\Spec R\ra K(G_R,n-1)^{k-1}$ is $K(G_R,n-1)\ra \Spec R$ which is cohomologically proper by the induction assumption, thus by \Cref{properties_of_coh_prop_morps} (3) so is the original map.} and thus, projecting $k$ times, we see that all $\mstack U_k$ are cohomologically proper. It then follows from \Cref{coh_prop_hypercover} that  $K(G_R,n)$ is cohomologically proper as well.
	\end{ex}

\section{Hodge-proper spreadability of quotient stacks}
\label{sec:spreadability of quotient stacks}
In this section we study in more detail the case of quotient stacks, providing several families of non-trivial examples of cohomologically proper spreadable stacks. Here the proofs of spreadability are much more involved; the following two important representation-theoretic results will be used: 
\begin{thm}[Theorem 3 and Proposition 57 of \cite{FranjouKallen}\footnote{Proposition 57 in \textit{loc.cit.} is stated for $R=\mbb Z$. However we can consider $A$ as a $\mbb Z$-algebra with an action of $\mbb G_{\mbb Z}$; indeed, since the action on $R$ is trivial and $R[G_R]\simeq \mbb Z[\mbb G_{\mbb Z}]\otimes_{\mbb Z}R$ we have $A\otimes_R R[G_R]\simeq A\otimes_{\mbb Z} \mbb Z[G_{\mbb Z}]$ and thus get the comultiplication $A\ra A\otimes_R R[G_R]\simeq A\otimes_{\mbb Z} \mbb Z[G_{\mbb Z}]$. Same thing works for any $G_R$-representation and, moreover, $\RG(G_R,M)\simeq \RG(G_{\mbb Z},M)$ (e.g. because they are computed by the same standard complex). Thus Proposition 57 in \textit{loc.cit.} also applies to any $R$ that is finitely generated over $\mbb Z$.}]\label{Franjou} 
	Let $G_{\mathbb Z}$ be a split reductive group over $\mathbb Z$ and $R$ be a finitely generated algebra over $\mathbb Z$. Let $A$ be a finitely generated $R$-algebra endowed with a (rational) action of $G_R$ and let $M$ be a finitely generated $G_R$-equivariant $A$-module. 
	Then the algebra $A^{G_R}$ of $G_R$-invariants is  finitely generated over $R$ and  $H^n(G_R, M)$ is a finitely generated $A^{G_R}$-module for any $n\ge 0$.
\end{thm}

\begin{thm}[{Kempf's theorem, see e.g. \cite[Proposition II.4.5]{Jantzen}}]\label{Kempf}
Let $G_{\mathbb Z}$ be a split connected reductive group over $\mathbb Z$ and let $B_{\mathbb Z}\subseteq G_{\mathbb Z}$ be a Borel subgroup. Let ${(G/B)}_{\mathbb Z}$ be the corresponding flag variety. Then $\RG({(G/B)}_{\mathbb Z},\mc O_{{(G/B)}_{\mathbb Z}})\simeq \mathbb Z$.
\end{thm}
As we will see, these two theorems, together with the semiorthogonal decompositions of derived categories constructed by Halpern-Leistner (\cite{HL-GIT}, \cite{HL-prep}) allow to prove in a lot of cases that the stack is cohomologically properly spreadable. We stick to the notations of \Cref{sec:Examples of spreadable Hodge-proper stacks}, in particular $F$ will denote an algebraically closed field of characteristic $0$ and we will choose the base $R$ of the spreading to be a finitely generated $\mathbb Z$-subalgebra of $F$.

Here is a plan of the section. In \Cref{sec:reductive quotients} we show that a quotient by reductive group is cohomologically properly spreadable (\Cref{thm:proper coarse moduli}), provided its coarse moduli space is proper and the action is locally linear. The key result is \Cref{prop: Y/G --> Y//G is cohomologically spreadable}, where we use \Cref{Franjou} to show that under the latter assumption the natural morphism $q\colon [Y/G]\ra Y/\!/G$ is cohomologically properly spreadable. Particular examples then include a proper-over-affine scheme with an action of a reductive group (\Cref{ex:projective-over-affine}) and quotients coming from GIT (\Cref{ex:GIT}). In \Cref{sec:more quotients} we look at some other set of examples (\Cref{thm: conical examples}), given by quotients that come from BB-complete $\mathbb G_m$-actions (\Cref{def:BB-complete}). \Cref{thm: conical examples} also allows some quotients by non-reductive groups, and \Cref{Kempf} is used to pass from the quotient by a Borel subgroup to the quotient by the whole group (\Cref{lem:use Kempf}). In \Cref{sec:Teleman} we also give a recipe of reestablishing the degeneration results of \cite{Teleman} in the KN-complete case using \Cref{Hodge_degeneration}; here, however, we need to assume some results which are going to appear in the upcoming work of Halpern-Leistner \cite{HL-prep}. 
\subsection{Global quotients by reductive groups} \label{sec:reductive quotients}
In \cite{Teleman} Teleman proved the Hodge-to-de Rham degeneration for the quotient of a smooth scheme $X$ by a Kempf-Ness (KN) complete action of a reductive group. In this section we establish spreadability for certain class of global quotients, which includes the semistable (single KN-stratum) case $X=X^{\mr{ss}}(\mc L)$ (\ref{ex:GIT}) and another standard KN-complete example given by equivariant "projective-over-affine" variety (\ref{ex:projective-over-affine}). Moreover, the "projectivity" condition  in the latter is replaced by the "proper" one almost for free.

Let $Y$ be a quasi-separated finite type scheme over $F$ and let $G$ be a reductive group acting on $Y$. 
\begin{defn}\label{def:locally linear action}
The action of $G$ on  $Y$ is called \textit{locally linear} if there exists a $G$-invariant affine cover of $Y$.
\end{defn}

In this case there exists the categorical quotient $Y/\!/ G$; in other words, the quotient stack $[Y/G]$ has a coarse moduli $q\colon [Y/G]\ra Y/\!/G$ and $Y/\!/ G$ is representable by a scheme. More explicitly, if $\{U_i\}_{i\in I}$, with $U_i\coloneqq\Spec A_i$, is the $G$-invariant affine cover of $X$, the categorical quotient $X/\!/ G$ is glued out of $\{U_i/\!/G\}_{i\in I}$, with $U_i/\!/G\coloneqq \Spec A_i^G$, with the natural gluing maps induced by the gluing maps for $\{U_i\}_{i\in I}$. Note that if $G$ is a torus and $Y$ is normal, the action is automatically locally linear by the result of Sumihiro (\cite[Corollary 2]{Sumihiro}).

\begin{prop}\label{prop: Y/G --> Y//G is cohomologically spreadable}
Let $Y$ be a quasi-separated finite type scheme over $F$ with a locally linear action of a reductive group $G$. Then the natural morphism $q\colon [Y/G]\ra Y/\!/G$ is cohomologically properly spreadable.

\begin{proof}
The group $G$ is split and has a Chevalley model $G_{\mathbb Z}$ over $\mathbb Z$; this defines a split reductive spreading of $G$ over any $R\subset F$, namely just put $G_R\coloneqq G_{\mathbb Z}\otimes_{\mathbb Z} R$. We can also spread $Y$ to a quasi-separated scheme $Y_R$ over some $R$ and assume that $Y_R$ has a $G_R$-action. It is enough to show that after a suitable enlargement $R$ the quotient stack $[Y_R/G_R]$ becomes cohomologically proper. Note that by picking a $G$-invariant affine cover $\{U_i\}_{i\in I}$ as above and a spreading $A_{i,R}$ of each $A_i$, localizing $R$, we can assume that the affine schemes $U_{i,R}\coloneqq \Spec A_{i,R}$ have a $G_R$-action and give a $G_R$-equivariant affine cover of $Y_R$.

Let $Y_R/\!/G_R$ be the categorical quotient, namely the scheme obtained by gluing the spectra $\Spec A_{i,R}^{G_R}$ of the invariants in the same way that was described above. Note that by \Cref{Franjou} the scheme $Y_R/\!/G_R$ is of finite type and so it is a valid spreading of $Y/\!/G$ (since $A_{i,R}^{G_R}\otimes_R F\simeq A_i$ for all $i\in I$). We have a natural map $q_R\colon [Y_R/G_R] \ra Y_R/\!/ G_R$ given locally (on $Y_R/\!/ G_R$) by $q_{i,R}\colon [\Spec A_{i,R}/G_R]\ra \Spec A_{i,R}^{G_R}$. The map $q_R$ is a spreading of $q$, thus it is enough to show that $q_R$ is cohomologically proper. This can be checked locally, so it is enough to show that for any finitely generated $R$-algebra $A$ the map $q_R\colon [\Spec A/G_R]\ra \Spec A^{G_R}$ is cohomologically proper. 

Let $\mc F\in \Coh([\Spec A/G_R])^\heartsuit$ and let $M$ be the corresponding $G_R$-equivariant finitely generated $A$-module. Since $\Spec A$ is affine, the module $q_{R*}\mc F$ has a simple description: it is just given by $\RG(G_R,M)$ considered as a complex of modules over $A^{G_R}$. The complex $\RG(G_R,M)$ lies in $\Mod_{A^{G_R}}^{\ge 0}$ and its cohomology are finitely generated by \Cref{Franjou}. Thus $\RG(G_R,M)\in \Coh^+(A^{G_R})$ and $q_R$ is cohomologically proper.
\end{proof}
\end{prop}

\begin{rem}
In the case of $BG$ some stronger results in a similar direction were established in \cite[Proposition 4.3.4]{HLP_equiv_noncom}. Namely under some mild restrictions on the reductive group scheme $\mcr G$ and the base scheme $S$ the structure morphism $B\mcr G\ra S$ is formally proper (in the sense of \cite{HLP_equiv_noncom}). 
\end{rem}

Note that we did not assume that $Y$ was smooth. This was on purpose: the actual cohomologically properly spreadable examples are given by the following theorem:
\begin{thm}\label{thm:proper coarse moduli}
Let $X$ be a smooth scheme and $Y$ be a finite-type scheme over $F$, both endowed with an action of a reductive group $G$. Assume that
\begin{enumerate}
\item There is a proper $G$-equivariant map $\pi\colon X\ra Y$.

\item The $G$-action on $Y$ is locally linear.

\item The categorical quotient $Y/\!/G$ is proper.
\end{enumerate}
Then the quotient stack $[X/G]$ is cohomologically properly spreadable.

\end{thm}
\begin{proof}
Consider the map $q\colon [Y/G]\ra Y/\!/G$; by \Cref{prop: Y/G --> Y//G is cohomologically spreadable} it is cohomologically properly spreadable. The map $\pi\colon [X/G]\ra [Y/G]$ is proper, thus $[X/G]$ is cohomologically properly spreadable by \Cref{chow_lem_for_spreadings}.
 \end{proof}

\begin{rem}
More generally one can replace $[X/G]$ by any $\mstack X$ with a cohomologically properly spreadable morphism $\pi\colon \mstack X \ra [Y/G]$. 
\end{rem}

We discuss some applications of \Cref{thm:proper coarse moduli}:

\begin{ex}\label{ex:projective-over-affine}
Let $X$ be a smooth proper-over-affine scheme $X$ with $\dim H^0(X,\mc O_X)^G<\infty$. By definition, this means that there  is a proper $G$-equivariant map $\pi\colon X\ra \Spec A$. By replacing $\Spec A$ with the image of $\pi$ we can assume that $\pi$ is surjective. Then $A^G$ embeds in $H^0(X,\mc O_X)^G$ and thus is finite-dimensional; equivalently, $(\Spec A)/\!/G$ is finite, and in particular proper. Applying \ref{thm:proper coarse moduli} to $Y=\Spec A$ we get that $[X/G]$ is cohomologically properly spreadable. Also note that we were able to relax the "projective" assumption on the map $\pi$ to the "proper" one. 
\end{ex}

\begin{ex}\label{ex:GIT}
Let $X$ have an ample $G$-equivariant line bundle $\mc L$, and let's assume that $X=X^{\mr{ss}}\coloneqq X^{\mr{ss}}(\mc L)$. Basically by definition, the action on $X^{\mr{ss}}(\mc L)$ is locally linear (see e.g. the proof of \cite[Theorem 1.10]{MumfordFogartyKirvan}). Assume further that $\dim_F H^0(X, \mc O_X)^G<\infty$. Then the scheme
$$X/\!/G\simeq \mr{Proj}\Big(\bigoplus_{n\ge 0}H^0(X,\mc L^{\otimes n})^G\Big)$$
is projective over $\Spec H^0(X, \mc O_X)^G$ and hence also projective over $F$. Thus $[X/G]$ is cohomologically properly spreadable by \Cref{thm:proper coarse moduli} applied to $Y=X$.
\end{ex}

\subsection{Global quotients coming from BB-complete $\mathbb G_m$-actions}
\label{sec:more quotients}
In this section we prove another result (\Cref{thm: conical examples}) about the cohomologically proper spreadability of quotient stacks, which also allows quotients by groups that are not necessarily reductive. Another benefit of \Cref{thm: conical examples} (compared, say, to \Cref{thm:proper coarse moduli}) is that the condition on $X$ (and the $G$-action) is internal: no additional structure, such as a map to another scheme $Y$, is involved.

\subsubsection{Varieties with a $\mathbb G_m$-action and Bialynicki-Birula stratification}\label{sec:G_m-actions}
Let $X$ be a smooth scheme over an algebraically closed field $F$ of characteristic 0 with an action $a\colon\mathbb G_m \curvearrowright X$. By \cite{Sumihiro} such an action is always locally linear:  $X$ has a $\mathbb G_m$-invariant affine cover $\{U_i=\Spec A_i\}_{i\in I}$. A $\mathbb G_m$-action on a given affine scheme $\Spec A$ induces a $\mathbb Z$-grading $A^*$ on $A$; we denote by $I^+\subset A^*$ the ideal generated by $A^{>0}$ and by $I^{\pm}\subset A^*$ the ideal generated by $A^{>0}$ and $A^{<0}$. 

Here are some examples:

\begin{ex}

\begin{itemize}
\item 
$\mr{pt}\coloneqq\Spec F$ with the trivial $\mathbb G_m$-action; here $A=A^0\simeq F$ and $I^+=I^\pm=0$.
\item $\mathbb A^1\coloneqq\Spec F[x]$ endowed with the standard action by dilation, $s\mapsto ts$ for $s\in \mathbb A^1$; in this case $\deg x= -1$, so $I^+=0$ and $I^{\pm}=(x)\subset F[x]$.
\end{itemize}
There are natural $\mathbb G_m$-equivariant maps $\xymatrix{\mr{pt}\ar@/^0.4pc/[r]^{i_0}&\mathbb A^1 \ar@/^0.4pc/[l]^{p}}$ given by the projection and the embedding of $0\in \mathbb A^1(F)$. We also have a (non-equivariant) map $i_1\colon\mr{pt}\ra \mathbb A^1$ given by the embedding of $1\in \mathbb A^1(F)$. 
 
\end{ex} 

To a smooth scheme $X$ endowed with a $\mathbb G_m$-action one can associate the following objects:
\begin{itemize}[wide]
\item  \textit{The fixed points} $X^0\coloneqq\mr{Maps}(\mr{pt},X)^{\mathbb G_m}$; its functor of points is given by $X^0(S)=\mr{Maps}(S,X)^{\mathbb G_m}$, meaning the $\mathbb G_m$-equivariant maps from $S$ to $X$, where the action on $S$ is trivial. There is a natural map $\iota\colon X^0\ra X$ sending a map $f\in X^0$ to its evaluation $f(\mr{pt})$. The map $\iota$ is a closed embedding (\cite[Proposition 1.2.2]{Drinfeld}); the affine cover $\{U_i\}_{i\in I}$ defines an affine cover $\{U_i^0\}_{i\in I}$ of $X^0$ with $U_i^0\coloneqq \Spec (A_i/I_i^\pm)$ (glued along $U_{ij}^0$). There is a natural $\mathbb G_m$-action on $X^0$, which is trivial. 
\item \textit{The attractor} $X^+\coloneqq \Map(\mathbb A^1,X)^{\mathbb G_m}$; here the functor of points is given by $X^+(S)=\mr{Maps}(S\times \mathbb A^1,X)^{\mathbb G_m}$, where the $\mathbb G_m$-action on $S\times \mathbb A^1$ is diagonal. By \cite[Corollary 1.4.3]{Drinfeld}, this functor is indeed represented by a scheme. There are two natural $\mathbb G_m$-actions on $\Map(\mathbb A^1,X)$, one coming from the $\mathbb G_m$-action on $\mathbb A^1$, the other coming from the action on  $X$; their restrictions to $\Map(\mathbb A^1,X)^{\mathbb G_m}$ coincide and thus define the same $\mathbb G_m$-action on $X^+$.

There are natural $\mathbb G_m$-equivariant maps $\xymatrix{X^0\ar@/^0.4pc/[r]^\sigma&X^+ \ar@/^0.4pc/[l]^{\pi}\ar[r]^j& X}$ induced by $i_0$, $i_1$ and $p$. Namely,  
\begin{itemize}
\item $\sigma(f)\in X^+$ is given by pre-composing $f\colon\mr{pt}\ra X$ with $p\colon\mathbb A^1\ra \mr{pt}$; the map $\sigma\colon X^0\ra X^+$ is a closed embedding.
\item $\pi(f)\in X^0$ is given by the evaluation of $f\colon\mathbb A^1\ra X$ at $0\in \mathbb A^1$; by \cite{Bial-Bir}\footnote{More precisely, by \cite[Proposition 1.4.10]{Drinfeld} $X^+$ is smooth, and (say, by \cite[Section 1.4.3]{Drinfeld}) the action of $\mbb G_m$ at any point $x\in X^0\hookrightarrow X^+$ is definite (in the terminology of \cite[Section 1]{Bial-Bir}). Then one can apply \cite[Theorem 2.5]{Bial-Bir} to $X^+$.} $\pi\colon X^+\ra X^0$ is a Zariski locally trivial fibration with the fiber given by an affine space.
\item $j(f)\in X\simeq \Map(\mr{pt},X)$ is given by the evaluation of $f\colon\mathbb A^1\ra X$ at $1\in \mathbb A^1$; the map $j\colon X^+\ra X$ is a locally closed embedding when restricted to each component of $X^+$.
\end{itemize}

Similarly to $X^0$, the affine cover $\{U_i\}_{i\in I}$ defines an affine cover $\{U_i^+\}_{i\in I}$ of $X^+$ with $U_i^+\coloneqq\Spec (A_i/I^+_i)$. 
\end{itemize}

\noindent Also, by \cite[Proposition 1.4.10]{Drinfeld} both $X^0$ and $X^+$ are smooth.

\begin{rem}\label{rem:maps in terms of a covering}
	The construction of the maps $\sigma,\pi,j$ is local and is described as follows in terms of a $\mbb G_m$-invariant cover $\{U_i\}_{i\in I}$: $\sigma_i\colon U_i^0 \hookrightarrow U_i^+$ is induced by the projection $A_i/I^+_i \ra A_i/I^\pm_i$; $A_i/I^\pm_i$ is identified with the 0-th graded component of the negatively graded algebra $A_i/I^+_i$ and this way the contraction $\pi_i\colon U_i^+ \ra U_i^0$ corresponds to the embedding $A_i/I^\pm_i\simeq (A_i/I^+_i)^0\ra A_i/I^+_i$. Finally $j_i\colon U_i^+ \ra U_i$ is given by the projection $A_i\ra A_i/I_i^+$. The maps for $X,X^+$ and $X^0$ are  obtained by gluing along the analogous maps for $U_{i,j}\coloneqq U_i\cap U_j$.
\end{rem}

Let $\pi_0(X^0)=\pi_0(X^+)$ be the set of connected components of $X^0$ (equivalently, $X^+$). For a given $c\in \pi_0(X^0)$ we denote by $Z_c\subset X^0$ the corresponding connected component of $X^0$, and by $S_c\coloneqq \pi^{-1}(S_c)\subset X^+$ the corresponding connected component of $X^+$; we call $\{S_c\}_{c\in \pi_0(X^0)}$ \textit{the Bialynicki-Birula (BB) strata}. 

\begin{defn}\label{def:BB-complete}
The $\mathbb G_m$-action $a\colon \mathbb G_m \curvearrowright X$ is called \textit{BB-complete} if the map $j\colon X^+\ra X$ is a surjection on the underlying topological spaces $|j|\colon|X^+|\twoheadrightarrow |X|$.
\end{defn}

\noindent  Equivalently, $j\colon X^+\ra X$ gives a \textit{full} stratification of $X$ with individual strata being locally closed. In this case, the limit of $a(t)\circ x$ for $t\ra 0$ exists for any point $x\in X$; in particular, the Bialynicki-Birula stratification $\{S_c\}\in \pi_0(X^0)$ gives a full stratification of $X$. We will fix some ordering on $\pi_0(X^0)$ with the only condition that $c'>c$ if $\dim S_{c'} < \dim S_{c}$; in this case a stratum $S_c$ is necessarily closed in $X_{\le c}:= X \setminus \cup_{c'>c} S_{c'}$.

\begin{rem}\label{rem: enough to check on F-points}
Since both $X^+$ and $X$ are of finite type and $F$ is algebraically closed, $|j|\colon |X^+|\ra |X|$ is surjective if and only if the corresponding map $j(F)\colon X^+(F)\ra X(F)$ between the $F$-points is.
\end{rem}


\subsubsection{BB-complete quotients by $\mathbb G_m$}
\label{sec:BB-complete}
\underline{\textit{Spreading BB-stratification.}}  Let $X$ be a smooth scheme over $F$ with a $\mathbb G_m$-action and let $\{U_i\}_{i\in I}$, $U_i\simeq \Spec A_i$ be a $\mathbb G_m$-invariant affine cover as above. The smooth schemes $X^+$ and $X^0$ are glued out of $\{U_i^+\}$ and $\{U_i^0\}$ (along $U_{ij}^+$ and $U_{ij}^0$) correspondingly.

By \Cref{spreading_with_properties_for_schemes} we can spread $X$ to a smooth $R$-scheme $X_R$ endowed with an action of $\mathbb G_{m,R}\coloneqq\Spec R[t,t^{-1}]$. We can also spread the cover $\{U_i\}_{i\in I}$ to a $\mathbb G_m$-invariant affine cover $\{U_{i,R}\}_{i\in I}$, $U_{i,R}\simeq \Spec A_{i,R}$ over some regular finitely generated $\mathbb Z$-algebra $R\subset F$. Each algebra $A_{i,R}$ is $\mathbb Z$-graded and we can consider closed subschemes $U_{i,R}^+\coloneqq \Spec (A_{i,R}/I_{i,R}^+)$ and $U_{i,R}^0\coloneqq \Spec (A_{i,R}/I_{i,R}^\pm)$, as well as schemes $X_R^+$ and $X_R^0$, obtained as their gluings along $U_{ij,R}^+$ and $U_{ij,R}^0$ (defined analogously). We have  $(X_R^+)\times_R F\simeq X^+$ and $(X_R^0)\times_R F\simeq X^0$.

Recall \Cref{rem:maps in terms of a covering} to see how the maps $\sigma, \pi, j$ between $X, X^+, X^0$ are defined in terms of the covering $\{U_i\}_{i\in I}$. Performing the same construction with $U_{i,R}$, $U_{i,R}^+$ and $U_{i,R}^0$ we obtain maps $\xymatrix{X^0_R\ar@/^0.4pc/[r]^{\sigma_R}&X_R^+ \ar@/^0.4pc/[l]^{\pi_R}\ar[r]^{j_R}& X_R}$ which spread $\sigma, \pi, j$. After enlarging $R$ further, by \Cref{spreading_with_properties_for_schemes} we can assume that $\sigma_R$ (resp. $j_R$) is a closed (resp. locally closed when restricted to each connected component) embedding. Picking an open cover $\{V_i\}_{i\in I}$ on which the affine fibration $\pi\colon X^+\ra X^0 $ is trivial, $\pi^{-1}(V_i)\simeq V_i\times \mathbb A^d$, we can spread it out; enlarging $R$ we can assume that $\{V_{i,R}\}_{i\in I}$ cover $X^0_R$, $\pi_R^{-1}(V_{i,R})$ cover $X_R^+$ and $\pi_R^{-1}(V_{i,R})\simeq V_{i,R}\times_R \mathbb A^d_R$. Furthermore, we can assume that $X^0_R$, and consequently $X_R^+$, are smooth over $R$.

After enlarging $R$ further, the set $\pi_0(X_R^0)$ can be identified with $\pi_0(X^0)$; the connected component $Z_{c,R}$ for $c\in \pi_0(X^0_R)\simeq \pi_0(X^0)$ then is a spreading of $Z_c\subset X^0$. Similarly, $S_{c,R}\coloneqq\pi_R^{-1}(Z_{c,R})$ is a connected component of $X_R^+$ and a spreading of $S_c$.  

If the $\mathbb G_m$-action on $X$ is BB-complete, we can assume the same about the $\mathbb G_{m,R}$-action on $X_R$; indeed, enlarging $R$ we can assume that the map $X_R^+\ra X_R$ is a surjection. In particular, we can assume that the spreading $\{S_{c,R}\}$ of the stratification $\{S_c\}$ of $X$ gives a full stratification of $X_R$ by locally closed smooth subschemes.

\smallskip

\noindent \underline{\textit{Semiorthogonal decomposition of $\Coh(\mstack X_R)$.}}
We now discuss certain semiorthogonal decompositions of the category $\Coh(\mstack X_R)$ given by the theory of so-called magic windows developed by Halpern-Leistner in \cite{HL-GIT}. Let $X$ be a smooth scheme over $F$ with a $\mathbb G_{m}$-action and let $X_R$ be a spreading of $X$ as constructed above. Let $S_R:=S_{c,R}\subset X_R$ be a closed stratum among $\{S_{c',R}\}$. Let $Z_R:=Z_{c,R}$ for the same $c$; the scheme $Z_R$ is called the \emph{centrum} of $S_R$. The subschemes $S_R$ and $Z_R$ enjoy the following nice properties:

\begin{prop}[]\label{KN-stratum-defn}
 
\begin{itemize}
\item Both $S_R$ and $Z_R$ are smooth over $R$;
\item There is a $\mathbb G_m$-equivariant map $\pi_R\colon S_{R}\simeq S_{R}^+ \ra S_{R}^0\simeq Z_{R}$, which is a Zariski locally trivial fibration with an affine space $\mathbb A^d_R$ (for some $d$) as a fiber;
\item The conormal bundle $N^\vee_{S_R} X_R$ has strictly positive weights when restricted to the fixed locus $Z_R\subset S_R$. 
\end{itemize}
\end{prop}
\begin{proof}
Since $Z_R$ and $S_R$ are connected components of $X_R^0$ and $X_R^+$ correspondingly, it is enough to show that the above properties hold for $X_R^0$ and $X_R^+$ in place of $Z_R$ and $S_R$. The first two properties are included in our construction of the spreading. The last property can be checked locally, thus we can assume $X_R\simeq \Spec A_R$, $S_R\simeq \Spec A_R/I_R^+$ and $Z_R\simeq \Spec A_R/I_R^\pm$. The normal bundle $N^\vee_{X_R^+} X_R$ is then given by $I_R^+/(I_R^+)^2$ as a $A_R/I_R^+$-module and its restriction to $X_R^0$ is given by $I_R^+/(I_R^+\cdot I_R^\pm)$, considered as a $A_R/I_R^\pm$-module). Since by definition $I_R^+$ is generated by elements of strictly positive weight, the weights of $I_R^+/(I_R^+\cdot I_R^\pm)$ are also strictly positive.
\end{proof}

\begin{rem}
In other words, using the terminology of \cite{HL-GIT}, $S_R\subset X_R$ is a smooth KN-stratum satisfying properties $(A)$ and $(L+)$ (in the case $G=\mathbb G_{m,R}$). 
\end{rem}
Let $\mstack X_R \coloneqq [X_R/\mathbb G_{m,R}]$, $\mstack S_R\coloneqq [S_R/\mathbb G_{m,R}]$ and $\mstack Z_R\coloneqq [Z_R/\mathbb G_{m,R}]$.
Also let $U_R\coloneqq X_R\setminus S_R$ be the complement and let $\mstack U_R \coloneqq [U_R/\mathbb G_{m,R}]$. Let $\Coh(\mstack X_R)$ be the (bounded derived) category of coherent sheaves on $\mstack X_R$  and let $\Coh_{\mstack S_R}(\mstack X_R)\subset \Coh(\mstack X_R)$ be the full subcategory of sheaves whose pull-back to the flat cover $X_R\twoheadrightarrow \mstack X_R$ is set-theoretically supported on $S_R$. The action of $\mathbb G_{m,R}$ on $Z_R$ is trivial and so $\mstack Z_R\simeq Z_R\times_R B\mathbb G_{m,R}$. Thus the heart $\Coh(\mstack Z_R)^\heartsuit$ of the category of coherent sheaves on $\mstack Z_R$ is identified with the category of $\mathbb Z$-graded objects in $\Coh(Z_R)^\heartsuit$. For a given $w$ we denote by $\Coh(\mstack Z_R)_{<w}\subset \Coh(\mstack Z_R)$ the full subcategory spanned by objects $\mathcal F\in \Coh(\mstack Z_R)$ whose cohomology sheaves $\mc H^i(\mathcal F)\in \Coh(\mstack Z_R)^\heartsuit$ have grading less than $w$. Similarly, by $\Coh(\mstack Z_R)_{\ge w}\subset \Coh(\mstack Z_R)$ we denote the full subcategory spanned by objects  whose cohomology sheaves have grading greater or equal than $w$.

Let $i_R\colon \mstack U_R \ra \mstack X_R$ and $j_R\colon \mstack S_R\ra \mstack X_R$ be the natural embeddings. Then, given $w\in \mathbb Z$, by \cite[Theorem 2.10]{HL-GIT}, we have a semiorthogonal decomposition 
\begin{equation}\label{decomposition}
\Coh(\mstack X_R)=\langle \ \! \Coh_{\mstack S_R}(\mstack X_R)_{< w}, \mr{G}_w , \Coh_{\mstack S_R}(\mstack X_R)_{\ge w} \ \!\rangle,
\end{equation}
where 
\begin{itemize}
\item $\Coh_{\mstack S_R}(\mstack X_R)_{\ge w}\coloneqq\{\mathcal F\in \Coh_{\mstack S_R}(\mstack X_R)\ | \ \sigma_R^* \mathcal F\in \Coh(\mstack Z_R)_{\ge w} \} \subset \Coh_{\mstack S_R}(\mstack X_R)$,
\item $\Coh_{\mstack S_R}(\mstack X_R)_{< w}\coloneqq\{\mathcal F\in \Coh_{\mstack S_R}(\mstack X_R)\ | \ \sigma_R^* j_R^! \mathcal F\in \Coh(\mstack Z_R)_{<w} \} \subset \Coh_{\mstack S_R}(\mstack X_R)$ 
\end{itemize} 
\noindent and $\mr{G}_w \subset \Coh(\mstack X_R)$ is a certain (full) subcategory such that the functor $i_R^*:\Coh(\mstack X_R)\ra \Coh(\mstack U_R)$ restricted to $\mr G_w$ 
$$
{i_R^*|}_{\mr G_w}\colon \mr G_w \xymatrix{\ar[r]^-\sim &} \Coh(\mstack U_R)
$$
is an equivalence. 

\begin{rem}
	Note that the \cite{HL-GIT} assumes the base to be a field of characteristic 0. Even though in the case of $G=\mathbb G_m$ this assumption is not really necessary, the semiorthogonal decomposition above is also covered by the proof of \Cref{prop:Theta-stratified stacks are cohomologically proper}. Indeed it is enough to show that $\mstack S_R$ gives a $\Theta$-stratum in $\mstack X_R$. This follows from the intrinsic description of $\Theta$-strata (see \cite[Proposition 1.4.1]{HL-prep} or also see the proof of \Cref{lem: Theta-strata spread out}): the third point of \Cref{KN-stratum-defn} exactly says that $\mbb L_{\mstack S_R/\mstack X_R}\in \QCoh(\mstack S_R)^{\ge 1}$ (meaning the term of the corresponding baric decomposition).
\end{rem}	

\begin{prop}\label{prop:KN-stratum induction step}
$\mstack X_R$ is cohomologically proper if and only if both $\mstack U_R$ and $\mstack S_R$ are.
\end{prop}
\begin{proof} 
Fix some $w\in \mathbb Z$.

"$\Rightarrow$" $\mstack S_R$ is cohomologically proper, since  for $\mathcal F\in \Coh(\mstack S_R)$ we have $j_{R*}\mathcal F \in\Coh(\mstack X_R)$ and $\RG(\mstack S_R, F)\simeq \RG(\mstack X_R, j_{R*}\mathcal F)$. Since $\mstack X_R$ is smooth, we have $\Coh(\mstack X_R)\simeq \mr{Perf}(\mstack X_R)$ and so any object of $\Coh(\mstack X_R)$ is dualizable; in particular, for any $V_1,V_2\in \Coh(\mstack X_R)$ we have 
$$
\Hom_{\Coh(\mstack X_R)}(V_1,V_2)\simeq \RG(\mstack X_R, \HHom(V_1,V_2)),
$$
where $\HHom(V_1, V_2) \in \Coh(\mstack X_R)$. Thus $\Hom_{\Coh(\mstack X_R)}(V_1,V_2)\in \Coh^+(R)$, or, in other words, $\Coh{(\mstack X_R)}$
is a \textit{nearly proper} $R$-linear stable $\infty$-category. Now, taking $E\in \Coh(\mstack U_R)$, we get 
$$\RG(\mstack U_R, E)\simeq \Hom_{\Coh(\mstack U_R)}(\mc O_{\mstack U_R}, E)\simeq \Hom_{\Coh(\mstack X_R)}(({i_R^*|}_{\mr G_w})^{-1}\mc O_{\mstack U_R}, ({i_R^*|}_{\mr G_w})^{-1} E)\in \Coh^+(R),$$
via the equivalence $({i_R^*|}_{\mr G_w})^{-1}\colon \Coh(\mstack U_R) \xra{\sim} \mr G_w\subset \Coh(\mstack X_R)$.

"$\Leftarrow$" It is enough to show that the category $\Coh(\mstack X_R)$ is nearly proper. Let's first show that the subcategory $\Coh_{\mstack S_R}(\mstack X_R)\subset \Coh(\mstack X_R)$ is nearly proper. Every object of $\Coh_{\mstack S_R}(\mstack X_R)$ has a finite filtration with graded pieces of the form $j_{R*} \mathcal F$, where $\mathcal F\in \Coh(\mstack S_R)$. Thus it is enough to show that $\mr{Hom}(j_{R*} \mathcal F_1,j_{R*} \mathcal F_2)\in \Coh^+(R)$ for any $\mathcal F_1, \mathcal F_2\in \Coh(\mstack S_R)$.  Since $\mstack S_R$ and $\mstack X_R$ are smooth we have 
$$
\Hom_{\Coh(\mstack S_R)}(j_{R*} \mathcal F_1, j_{R*} \mathcal F_2)\simeq \RG(\mstack S_R, \HHom(j^*_{R}j_{R*}\mathcal F_1, \mathcal F_2))
$$
with $\HHom(j^*_{R}j_{R*}\mathcal F_1, \mathcal F_2)\in \Coh(\mstack S_R)$; then $\RG(\mstack S_R, \HHom(j^*_{R}j_{R*}\mathcal F_1, \mathcal F_2))\in \Coh^+(R)$ since $\mstack S_R$ is cohomologically proper. 

More generally, given $\mathcal F\in \Coh (\mstack S_R)$ and $V\in \Coh(\mstack X_R)$ we have
\begin{align*}
 &\Hom_{\Coh (\mstack X_R)}(V,j_{R*}\mathcal F)\simeq \Hom_{\Coh (\mstack S_R)}(j_{R}^*V, \mathcal F)\in \Coh^+(R),\\
 &\Hom_{\Coh (\mstack X_R)}(j_{R*}\mathcal F, V)\simeq \Hom_{\Coh (\mstack S_R)}(\mathcal F, j_{R}^!V)\in \Coh^+(R),
\end{align*}
where $j_{R}^!V \in \Coh(\mstack S_R)$, since both $\mstack X_R$ and $\mstack S_R$ are smooth. It follows that $\Hom_{\Coh (\mstack X_R)}(V,E)$ and $\Hom_{\Coh (\mstack X_R)}(E,V)$ both lie in $\Coh^+(R)$ if $E\in \Coh_{\mstack S_R}(\mstack X_R)$ and $V$ is any coherent sheaf on $\mstack X_R$. 
 
Now, due to the semiorthogonal decomposition in \eqref{decomposition} any $V\in \Coh(\mstack X_R)$ has a finite (in fact three-step) filtration  with the associated graded pieces lying either in $ \Coh_{\mstack S_R}(\mstack X_R)$ or  $\mr G_w$. By the above discussion, hom-complex $\Hom_{\Coh(\mstack X_R)}(E, -)$ for any object $E$ in $\Coh_{\mstack S_R}(\mstack X_R)$ is always bounded below coherent. The category $\mr G_w \simeq \Coh(\mstack U_R)$ is nearly proper since $\mstack U_R$ is cohomologically proper. Taking such filtrations for a given pair $\mathcal F_1, \mathcal F_2 \in \Coh(\mstack X_R)$ and using the exactness of $\Hom$ in each variable we get that $\Hom_{\Coh(\mstack X_R)}(\mathcal F_1, \mathcal F_2)\in \Coh^+(R)$.
\end{proof} 

The $\mathbb G_{m,R}$-action on $X_{R}^0$ is trivial; thus $ \mstack X_R^0$ is isomorphic to $X_R^0\times_R B\mathbb G_{m,R}$. Let $p\colon\mstack X_R^0 \ra X_R^0$ be the projection. 

\begin{lem}\label{lem: from + to 0}
The morphism $p\circ\pi_R\colon \mstack X_R^+ \ra X_R^0$ is cohomologically proper.
\end{lem}
\begin{proof}
The statement is local on $X_R^0$. Let $\{U^0_{i,R}\}_{i\in I}$, $U^0_{i,R}=\Spec A^0_{i,R}$ be an affine cover of $X_R^0$ such that the affine bundle given by $\pi_R\colon X_R^+\ra X_R^0$ is trivial. Let $U^+_{i,R}\coloneqq \pi_R^{-1}(U^0_{i,R})$. It is enough to show that the morphism $\mstack U^+_{i,R}\coloneqq [U^+_{i,R}/\mathbb G_{m,R}]\ra U^0_{i,R}$ induced by the composition of $\pi_R$ and $p$ is cohomologically proper. We have $U^+_{i,R}\simeq U^0_{i,R}\times_R \mathbb A^d_R$ where $\mathbb G_{m,R}$ acts with negative weights on $\mathbb A^d_R$. Let $A^+_{i,R}$ be the ring of functions on $U^+_{i,R}$; it is naturally $\mathbb Z$-graded and $A^+_{i,R}\simeq A^0_{i,R}[x_1,\ldots,x_d]$ where $x_i$'s can be chosen to be homogeneous of strictly negative degree. 

The functor $(p\circ \pi_R)_*$ is $t$-exact, since $\pi_R$ is affine and $\mathbb G_{m,R}$ is linearly reductive. We have an equivalence between the abelian category $\Coh(\mstack U^+_{i,R})^\heartsuit$ and the category of finitely generated graded $A^+_{i,R}$-modules. Via this equivalence, $(p\circ \pi_R)_*$ sends a graded module $M^*$ to the $A^0_{i,R}$-module $M^0$. Since the degrees of $x_i$'s are strictly negative, it is straightforward to see that if a $A^0_{i,R}[x_1,\ldots,x_d]$-module $M$ is finitely generated, the $A^0_{i,R}$-module $M^0$ is finitely generated as well, and thus corresponds to a coherent sheaf on $U_{i,R}^0$. 
\end{proof}

\begin{prop}\label{prop: X_R vs X_R^+ vs X_R^0}
Let $X$ be smooth scheme over $F$ and with a BB-complete $\mathbb G_m$-action. Let $X_R$ be a spreading as above. Then the following are equivalent: 
\begin{enumerate}
\item $\mstack X_R$  is cohomologically proper. 
\item $\mstack X_R^+$ is cohomologically proper.
\item $X_R^0$ is proper.
\end{enumerate}

\begin{proof}
$1\Leftrightarrow 2$. Note that $\mstack X_R^+\simeq \cup_{c\in \pi_0(X^0)} \mstack S_{c,R}$. Let's fix an ordering on $\pi_0(X^0)$ such that each $S_{c,R}$ is closed in $X_{\le c,R}:=X_R\setminus \cup_{c'>c} S_{c',R}$. Since $\{S_{c,R}\}$ form a full finite stratification of $X_R$, applying \Cref{prop:KN-stratum induction step} we are done by induction on $c$.

$2\Rightarrow 3$.  The morphism $\sigma_R\colon \mstack X_R^0 \ra \mstack X_R^+$ is a closed embedding (in particular proper) and thus is cohomologically proper. It follows that $\mstack X_R^0$ is cohomologically proper. On the other hand $\mstack X_R^0\simeq X_R^0\times_R B\mathbb G_{m,R}$ and so $X_R^0$ is cohomologically proper as well. Indeed, $p_*\mc O_{\mstack X_R^0}\simeq \mc O_{X_R^0}$ and given $\mathcal F\in \Coh(X_R^0)$ we have $\RG(\mstack X_R^0, p^*\mathcal F)\simeq \RG(X_R^0, \mathcal F)$ by the projection formula; since $p^*\mathcal F\in \Coh(\mstack X_R^0)$ we get $\RG(X_R^0, \mathcal F)\in \Coh^+(R)$. Being a cohomologically proper $R$-scheme, $X_R^0$ is forced to be proper (Corollary 3.8 of \cite{GAGA}).  

$2\Leftarrow 3$: The structure morphism $\mstack X_R^+\ra \Spec R$ factors as the composition of $\mstack X_R^+\xra{p\circ \pi_R} X_R^0$ and $X_R^0\ra \Spec R$. The first map is cohomologically proper by \Cref{lem: from + to 0}, the second --- since $X_R^0$ is proper. 
\end{proof}
\end{prop}

\begin{cor}\label{cor:  X_R is cohomologically proper}
Let $X$ be a smooth scheme over $F$ with an action of $\mathbb G_m$. Assume that the action is BB-complete and that the scheme of $\mathbb G_m$-fixed points $X^0$ is proper. Then $\mstack X$ is cohomologically properly spreadable.
\end{cor}
\begin{proof} Let $X_R$ be a spreading as above, then $X_R^0$ is a spreading of $X^0$ and thus, after enlarging $R$, we can assume that $X_R^0$ is proper. Then $\mstack X_R$ is cohomologically proper by 
\Cref{prop: X_R vs X_R^+ vs X_R^0}. 
\end{proof}

\subsubsection{Quotients by $G$ that are BB-complete with respect to a subgroup} 
In \Cref{cor:  X_R is cohomologically proper} we gave some sufficient conditions for the quotient stack $[X/\mathbb G_m]$ to be cohomologically spreadable. As we will see soon, Kempf's theorem allows to generalize this to a quotient by an arbitrary linear group $G$; however, a certain extra weight-positivity assumption with respect to a 1-parameter subgroup $h\colon\mathbb G_m\ra G$ is necessary.

Let $G$ be a linear algebraic group and let $B\subset G$ be a Borel subgroup\footnote{Recall that a subgroup $B\subset G$ is called Borel if it is a maximal Zariski-closed solvable subgroup of $G$.}. Let $U\subset B$ be the unipotent radical of $B$ and let $T\subset B$ be a maximal torus. We have a short exact sequence $1\ra U\ra B\ra T\ra 1$. Let $X^*(T)\coloneqq \Hom(T,\mathbb G_m)$ and $X_*(T)\coloneqq\Hom(\mathbb G_m,T)$ be the character and cocharacter lattices of $T$. One has $X_*(T)\simeq X^*(T)^\vee$. Given a $T$-representation $V$ and a character $\lambda\in X^*(T)$ we denote by $V_\lambda \subset V$ the subspace of $V$ of weight $\lambda$. 
The adjoint action of $T$ on $U$ induces an action on the Lie algebra $\mf u$ of $U$ and we denote by $\Phi^+\subset X^*(T)$ the set of weights of $\mf u$ with respect to this action. 

\begin{thm}\label{thm: conical examples}
Let $X$ be a smooth scheme over $F$ endowed with an action of a linear algebraic group $G$. Let $B\subset G$ be a Borel subgroup and let $T\subset B$ be a maximal torus. Let $\Phi^+\subset X^*(T)$ be the set of $T$-weights of the Lie algebra $\mf u$ of the unipotent radical $U\subset B$ with respect to the adjoint action of $T$ on $U$. 

Assume that there is a subgroup $h\colon \mathbb G_m \ra T$, $h\in X_*(T)$, such that 
\begin{enumerate}
\item $h(\Phi^+)>0$.
\item The $\mathbb G_m$-action induced by $h$ is BB-complete.
\item The $h(\mathbb G_m)$-fixed points $X^0$ are proper.
\end{enumerate}

\noindent Then the quotient stack $[X/G]$ is cohomologically properly spreadable.

\begin{proof}

Let's first assume that $G$ is connected. Note that $B$ is isomorphic to a semidirect product $T\ltimes U$ and can be spread out to a semidirect product $T_R\ltimes U_R$ of a split torus $T_R$ and a unipotent group $U_R$ over a finitely generated $\mathbb Z$-algebra $R\subset F$. Since $T_{R}$ is split $X^*(T_{R})\simeq X^*(T)$. In particular we have a cocharacter $h\colon\mathbb G_{m,R}\ra T_R$. The subgroup $B\subset G$ can be spread out to a closed subgroup $B_R \subset G_R$ and we can assume that $G_R$ is split. Let $U_G$ be the unipotent radical of $G$. Then $G/U_G$ is reductive and can be spread out to a split reductive group $(G/U_G)_R$. We then also have a spreading $p_R\colon G_R \ra (G/U_G)_R$ of the projection $p\colon G\ra G/U_G$ and the kernel $(U_{G})_R\coloneqq\Ker(p_G)$ is a spreading of $U_G$ and thus can be assumed to be unipotent. Since $U_G$ is a closed subgroup of $B$, we can assume that $(U_G)_R$ is a closed subgroup of $B_R$. The image of $B_R$ under $p_R$ is a spreading of $B/U_G\subseteq G/U_G$ and thus can be assumed to be a Borel subgroup of the split reductive group $(G/U)_R$. Note that with all these assumptions $G_R/B_R\simeq (G/U_G)_R/p_R(B_R)$.

We can also spread $X$ with the action $a\colon G\curvearrowright X$ to a smooth scheme $X_R$ over $R$ with an action $a_R\colon G_R\curvearrowright X_R$. Note that by \cite{Sumihiro} the restriction of the action of $G$ on $X$ to $\mathbb G_m$ (via $h$) is locally linear; consider a $\mathbb G_m$-invariant open cover $\{U_i\}_{i\in I}$ of $X$, $U_i=\Spec A_i$. We have spreadings $A_{i,R}$ of $A_i$ with an action of $\mathbb G_{m,R}$; localizing $R$ if necessary, we can assume that $U_{i,R}:=\Spec A_{i,R}$ cover $X_R$ and that the restriction of $a_R$ to $\mathbb G_{m,R}$ via $h$ is locally linear. Localizing $R$, we can assume that the $\mathbb G_{m,R}$-fixed points  $X_R^0$ are proper, the action is BB-complete and $X_R$ is such that the conditions of \ref{prop: X_R vs X_R^+ vs X_R^0} are satisfied. Thus we have that $[X_R/h(\mathbb G_{m,R})]$ is cohomologically proper. It is enough to show that $[X_R/G_R]$ is cohomologically proper. 

We split the argument into a sequence of lemmas:
\begin{lem}\label{from /G_m to /T_R} Let $X_R$ be as above. Then 
$$
\text{$[X_R/h(\mathbb G_{m,R}))]$ is cohomologically proper over $R$} \ \  \Rightarrow \ \ \text{$[X_R/T_R]$ is cohomologically proper over $R$}.
$$
\end{lem}

\begin{proof}
Let $p\colon X_R\ra [X_R/T_R]$ and $q\colon [X/h(\mathbb G_{m,R})]\ra [X_R/T_R]$ be the natural smooth covers. Then, given $\mc F\in \Coh([X_R/T_R])$, we have $\RG([X_R/T_R], \mc F)\simeq \RG(X_R,p^*\mc F)^{T_R}$. But 
$$
\RG(X_R,p^*\mc F)^{T_R}\simeq \left(\RG(X_R,p^*\mc F)^{h(\mathbb G_{m,R})}\right)^{T/h(\mathbb G_{m,R})}\simeq \RG([X_R/{h(\mathbb G_{m,R})}],s^*\mc F)^{T/h(\mathbb G_{m,R})};
$$
since $s^*\mc F\in \Coh([X_R/{h(\mathbb G_{m,R})}])$ we have $\RG([X_R/{h(\mathbb G_{m,R})}],s^*\mc F)\in \Coh^+(R)$. Recall that the coherence of a complex of $R$-modules is equivalent to being $t$-bounded and having finitely generated cohomology. The group scheme $T/h(\mathbb G_{m,R})$ is a torus and the functor of $T_R/h(\mathbb G_{m,R})$-invariants is $t$-exact. Thus $\RG([X_R/{h(\mathbb G_{m,R})}],s^*\mc F)^{T/h(\mathbb G_{m,R})}$ is also bounded and has finitely generated cohomology, so is coherent.
\end{proof}

\begin{lem} \label{from /G_m to /B_R} Let $X_R$ be as above. Then 
$$
\text{$[X_R/h(\mathbb G_{m,R}))]$ is cohomologically proper over $R$} \ \  \Rightarrow \ \ \text{$[X_R/B_R]$ is cohomologically proper over $R$}.
$$

\end{lem}

\begin{proof}
Consider the natural smooth cover $q\colon  [X_R/T_R]\ra [X_R/B_R]$ induced by the embedding $T_R\ra B_R$. Since $B_R\simeq  T_R\ltimes U_R $, the $n$-th term of the {\v C}ech complex associated to $q$ is given by
$$
[X_R\overset{T_R}{\times}\underbrace{B_R\overset{T_R}{\times} B_R\overset{T_R}{\times} \cdots \overset{T_R}{\times} B_R}_n/T_R]\simeq [(X_R\times \underbrace{U_R\times U_R\times \cdots \times U_R}_n)/T_R]\footnote{The isomorphism is given by the formula $[(x, b_1,\ldots, b_n)]\mapsto [(x,u_1, \ldots, u_n)]$, where $b_i=t_i\cdot u_i\in T_R \ltimes U_R$. },
$$
where the action of $T_R$ on $X_R\times U_R\times U_R\times \cdots \times U_R$ is given by the product of the action on $X_R$ and the adjoint action on each copy of $U_R$. 

Note that the underlying scheme of $U_R$ can be $T_R$-equivariantly identified with its Lie algebra $\mf u_R$ (see II.1.7 in \cite{Jantzen}); this way functions on $U_R$ (resp $U_R^n$) are identified with $\Sym_{R}(\mf u_R^*)$ (resp. $\Sym_{R}(\mf u_R^*)^{\otimes n}$). Since $h(\Phi^+)>0$ we get that the $\mathbb G_{m,R}$-weights on non-constant homogeneous functions on $U_R^n$ are strictly negative. It follows that $U_R^n\simeq (U_R^n)^+$ and  $(U_R^n)^0\simeq \Spec R$. We have $(X_R\times U_R^n)^+\simeq X_R^+\times U_R^n$, so the map $(X_R\times U_R^n)^+\ra X_R\times U_R^n$ is surjective and  thus the $\mathbb G_{m,R}$-action on $X_R\times U_R^n$ is BB-complete for every $n$. Also, $(X_R\times U_R^n)^0\simeq X_R^0$ is proper. Finally, since $U_R^n$ is isomorphic to the affine space, the Bialynicki-Birula strata still satisfy the conditions in \ref{KN-stratum-defn}. Consequently, \Cref{prop: X_R vs X_R^+ vs X_R^0} applies to $X_R\times U_R^n$ for all $n$; we get that $[(X_R\times U_R^n)/h(\mathbb G_{m,R})]$ is cohomologically proper. By \Cref{from /G_m to /T_R} it follows that $[(X_R\times U_R^n)/T_R]$ is cohomologically proper for all $n$. By \Cref{coh_prop_hypercover} we get that $[X_R/B_R]$ is cohomologically proper.
\end{proof}

To pass from $B_R$ to $G_R$ we use the Kempf's theorem (\ref{Kempf}):

\begin{lem}\label{lem:use Kempf}
 Let $X_R$ be as above. Then 
$$
\text{$[X_R/B_R]$ is cohomologically proper over $R$} \ \  \Rightarrow \ \ \text{$[X_R/G_R]$ is cohomologically proper over $R$}.
$$

\end{lem}
\begin{proof}
Let $j\colon BB_R\ra BG_R$ be the natural morphism. Then by the projection formula
$$
\RG(BB_R,j^*\mc F)\simeq \RG(BG_R,j_*j^*\mc F)\simeq \RG(BG_R, \mc F\otimes j_*\mathcal O_{BB_R}).
$$
By base change, the underlying complex of $R$-modules of $j_*\mc O_{BB_R}$ is equivalent to $\RG(G_R/B_R, \mathcal O_{G_R/B_R})$. But $G_R/B_R\simeq (G/U_G)_R/(p_R(B_R))$, where $p_R(B_R)\subset (G/U_G)_R$ is a Borel subgroup and thus $\RG(G_R/B_R, \mathcal O_{G_R/B_R})\simeq R$ by \Cref{Kempf}. Consequently, $\RG(BG_R, \mc F) \simeq \RG(BB_R, j^*\mc F)$ for any sheaf $\mc F\in \QCoh(BG_R)$. We now apply this as follows: there is a fibered square
$$\xymatrix{
[X_R/B_R]\ar[r]^{f}\ar[d]&[X_R/G_R]\ar[d]^{}\\
BB_R\ar[r]^{j}&BG_R 
}$$
and, given a coherent sheaf $\mc F\in \Coh([Y_R/G_R])$, its pull-back $f^*\mc F\in \Coh([Y_R/B_R])$ is also coherent. Applying base change and the above isomorphism we get that $\RG([Y_R/G_R],\mc F)\simeq \RG([Y_R/B_R],f^*\mc F)$; in particular, one complex is bounded below coherent if and only if the other one is. The statement of the lemma follows.
\end{proof}

It remains to cover the case of a disconnected $G$.
We can write $[X/G]\simeq [[X/G^0]/\pi_0(G)]$, where $G^0$ is the connected component of $e\in G$ and $\pi_0(G)$ is the finite group of components. 
The homomorphism $p\colon G\ra \pi_0(G)\simeq G/G^0$ can be spread out to $p_R\colon G_R\ra \pi_0(G)_R$ where $G_R$ is some spreading out of $G$ and $\pi_0(G)_R$ is the constant group $R$-scheme associated to $\pi_0(G)$. Moreover the kernel $G^0_R$ of $p_R$ is a spreading of $G^0$.

We have just shown that the quotient stack $[X_R/G^0_R]$ is cohomologically proper over $R$. We also have $[X_R/G_R]\simeq [[X_R/G^0_R]/\pi_0(G)_R]$. It follows that for any $\mc F\in \Coh([X/G_R])$
$$
\RG([X_R/G_R], \mc F) \simeq \RG\left(B\pi_0(G)_R, \RG([X_R/G^0_R], \mc F)\right).
$$
Replacing $R$ with $R[1/|\pi_0(G)|]$ we can assume that $|\pi_0(G)|$ is invertible in $R$, and so the functor of $\pi_0(G)$-invariants is $t$-exact. Then we get $H^q([X_R/G_R], F) \simeq H^q([X_R/G^0_R], \mc F)^{\pi_0(G)}$. In particular, $\RG([X_R/G_R], \mc F)$ is $t$-bounded and its cohomology are finitely generated over $R$, so $\RG([X_R/G_R], \mc F)\in \Coh^+(R)$.
\end{proof}
\end{thm}

We end this subsection by giving some examples to which \Cref{thm: conical examples} does apply:
\begin{ex}\label{ex:mpre examples}
\begin{enumerate}[leftmargin=0pt,itemindent=*]
\item $X$ is proper. In this case by the valuative criterion of properness every $\mathbb G_m$-orbit of an $F$-point has a limit as $t\ra 0$, so the map $X^+\ra X$ is surjective on $F$-points. It follows that any $\mathbb G_m$-action on $X$ is BB-complete. Moreover $X^0\subset X$ is a closed subscheme and so is proper. Thus, the only condition to check is on $G$: namely there should exist $h\in X_*(T)$ such that $h(\Phi^+)>0$ (since all Borel subgroups of $G$ are conjugate to each other this does not depend on the choice of $B$). Here is a list of linear algebraic groups $G$ which satisfy this:

\begin{itemize}
\item $G$ reductive. Then one can take $h\in X_*(T)$ given by any dominant coweight. This case is also covered by \Cref{thm:proper coarse moduli};
\item $G=P\subset H$ is a parabolic subgroup of a reductive group $H$. Any $h$ that is dominant with respect to some Borel subgroup $B\subset P$ applies;
\item More or less tautologically, any $G$ with a 1-dimensional subtorus $\mathbb G_m\subset G$ such that the adjoint action of $\mathbb G_m$ on the  the Lie algebra $\mf u_G$ of the unipotent radical $U_G\subset G$ has strictly positive weights and such that the projection of $\mathbb G_m$ to $G/U_G$ gives a regular coweight (meaning its centralizer is given by a maximal torus). Then one picks $B$ as the preimage of a Borel subgroup of $G/U_G$, with respect to which the $\mathbb G_m$ above gives a dominant coweight, under the projection $G\twoheadrightarrow G/U_G$ and take $h$ given by any lifting $\mathbb G_m\ra B$. As a non-parabolic example of such $G$ one can take any semidirect product $\mathbb G_m\ltimes U$ where $U$ is unipotent and $\mathbb G_m$ acts on $\mf u$ with strictly positive weights.
\end{itemize}

\item There are also natural examples that are more in the spirit of \Cref{thm:proper coarse moduli}. Namely, let $\pi\colon X\ra Y$ be a proper $G$-equivariant morphism where $Y$ is not necessarily smooth. Then, given a cocharacter $h\colon\mathbb G_m\ra G$ that satisfies $h(\Phi^+)>0$ for some $B\subset G$, we have that if $Y^+\twoheadrightarrow Y$ is a surjection and $Y^0$ is proper, $X$ satisfies the conditions of \Cref{thm: conical examples}. Indeed the induced map $X^0\ra Y^0$ is proper and so $X^0$ is proper. Also, given any point $x\in X(F)$, the image  $\pi(\mathbb G_m\cdot x)$ of its orbit is the orbit $\mathbb G_m\cdot \pi(x)$. Since $Y^+\twoheadrightarrow Y$ is a surjection, the limit $\lim_{t\ra 0} t\circ \pi(x)$ exists. This gives a diagram
$$
\xymatrix{\mathbb G_m \ar[r]^{\cdot\circ x}\ar[d]& X\ar[d]\\
\mathbb A^1\ar@{-->}[ru]\ar[r]^{\cdot \circ \pi(x)} & Y}
$$
and, due to the valuative criterion of properness, the lifting $\xymatrix{\mathbb A^1\ar@{-->}[r]& X}$, this way producing the limit of $t\circ x$ as $t\ra 0$. We get that $X^+\ra X$ is a surjection on $F$-points and that the $\mathbb G_m$-action given by $h$ is BB-complete.

This applies, in particular, to the case when $Y^0\simeq \Spec F$ and $G=\mathbb G_m$ (where we can assume $h=\mr{id}\colon \mathbb G_m\ra \mathbb G_m$). In this case we basically arrive at the definition of a conical resolution (see e.g. \cite{KubrakTravkin}). Namely, we have $Y\simeq Y^+ \ra \Spec F$ is  affine, so $Y\simeq \Spec A$; the induced $\mathbb Z$-grading on $A$ is such that $A\simeq A^{\le 0}$ and $A^0\simeq F$\footnote{Note the change of sign in the grading compared to \cite{KubrakTravkin}. In the case of a commutative group action there are two natural left actions on the space of functions on $Y$, induced either by the action of $g$ or $g^{-1}$ on $Y$. This is exactly the difference we are facing here.}. The map $\pi\colon X\ra \Spec A$ is proper, $X$ is smooth and the $\mathbb G_m$-action on $X$ agrees with the grading on $A$.  The geometry of such $X$ is the following: it is not proper itself, but it has a proper $\mathbb G_m$-equivariant map to $\Spec A$ so that the $\mathbb G_m$-action contracts it to the central fiber $\pi^{-1}(Y^0)$ which is proper over $F$. Note that even if $X$ is smooth, $\pi^{-1}(Y^0)$ can be singular (for example in the case of the minimal resolution of the $A_n$-singularity for $n>2$).
\end{enumerate}
\end{ex}

\subsection{$\Theta$-stratified stacks and the relation to the work of Teleman}
\label{sec:Teleman}

The example of $BB$-complete quotients by $\mbb G_m$ can be vastly generalized by the notion of a $\Theta$-stratified stack introduced recently by Halpern-Leistner (and studied in great detail in \cite{HL-prep}). All the stacks in this section are assumed to be derived and we also assume the base ring $R$ to be Noetherian and regular. Let $\mstack X$ be a derived stack over $R$ and assume that it is locally almost of finite presentation with affine diagonal. We also let $\Theta\coloneqq {[\mbb A^1/\mbb G_m]}$. One can define two mapping stacks associated to $\mstack X$: the stack of \textit{filtered objects} $\Filt(\mstack X)\coloneqq \ul\Map_{R}(\Theta, \mstack X)$ and the stack of \textit{graded objects} $\Grad(\mstack X)\coloneqq \ul\Map_{R}(B\mbb G_m, \mstack X)$. We have a stacky version of maps defined in \Cref{sec:G_m-actions}
$$
\xymatrix{\Grad(\mstack X)\ar@/^0.4pc/[r]^{\sigma}&\Filt(\mstack X) \ar@/^0.4pc/[l]^{\ev_0}\ar[r]^(.57){\ev_1}& \mstack X}
$$
induced by evaluations at $0\colon B\mbb G_m \hookrightarrow \Theta$ and $1\colon \Spec R\simeq [\left(\mbb A^1\setminus 0\right)/\mbb G_m]\ra \Theta$, and the natural projection $\Theta \ra B\mbb G_m$. Note that if $\mstack X$ is smooth (and thus classical) by \cite[Corollary 1.3.2.1]{HL-prep} the stack $\Filt(\mstack X)$ is also smooth and classical. A \textit{derived $\Theta$-stratum} $\mstack S$ is by definition a union of connected components of $\Filt(\mstack X)$ with the condition that $\ev_1 |_{\mstack S}\coloneqq \mstack S \ra \mstack X$ is a closed embedding. Let $\mstack Z\coloneqq \sigma^{-1}(\mstack S)\subset \Grad(\mstack X)$ be the centrum of $\mstack S$; $\ev_0$ restricts to a map $\mstack S \ra \mstack Z$.

\begin{defn}[A particular case of {\cite[Definition 1.10.1]{HL-prep}}]
	A finite \textit{$\Theta$-stratification} of $\mstack X$ indexed by a totally ordered finite set $I$ with a minimal element $0\in I$ is given by: 
	\begin{enumerate} \item  A collection of open substacks $\mstack X_{\le \alpha}\subset \mstack X$ with $\alpha\in I$ such that $\mstack X_{\alpha}\subset \mstack X_{\alpha'}$ if $\alpha< \alpha'$.
		
		\item  For each $\alpha>0$ a $\Theta$-stratum $\mstack S_\alpha\subset \Filt(\mstack X_\alpha)$ such that $\mstack X_{\le \alpha}\setminus\left( \cup_{\alpha'<\alpha} \mstack X_{\le \alpha'}\right)= \ev_1(\mstack S_\alpha)$. \item One should have $\mstack X=\cup_{\alpha\in I} \mstack X_\alpha$.
\end{enumerate}	
The minimal open stratum $\mstack X^{ss}\subset \mstack X$ is called the \textit{semistable locus}.
\end{defn}

Let $i_\alpha\colon \mstack S_\alpha \ra \mstack X_{\alpha}$ be the embedding induced by $\ev_1$. The pushforward $i_{\alpha*}$ has left $i^*_\alpha\colon \QCoh(\mstack X_{\le \alpha})\ra \QCoh(\mstack S_\alpha)$ and right $i_\alpha^!\colon \QCoh(\mstack X_{\le \alpha})\ra \QCoh(\mstack S_\alpha)$ adjoints. Also let $\Coh^-(-)$ denote the bounded above category of coherent sheaves. 

\begin{prop}\label{prop:Theta-stratified stacks are cohomologically proper}
	Let $\mstack X$ be a smooth Artin stack of finite type over $R$ with affine diagonal endowed with a finite $\Theta$-stratification. Assume that $\mstack X^{ss}$ and the centra $\mstack Z_\alpha$ are cohomologically proper over $R$. Then $\mstack X$ is cohomologically proper over $R$.
\end{prop} 
\begin{proof}
	By induction on $|I|$ we can reduce to the case of a single $\Theta$-stratum $\mstack S$ with the complement given by $\mstack X^{ss}$. The stratum $\mstack S$ is a connected component of $\Filt(\mstack X)$ and thus is also smooth over $R$. Since both $\mstack X$ and $\mstack S$ are smooth, we get that $i^*$ restricts to a functor between $\Coh(\mstack X)\simeq \QCoh(\mstack X)^\perf$ and $\Coh(\mstack S)\simeq \QCoh(\mstack S)^\perf$. Also by smoothness the direct image $i_*\mc O_{\mstack S}\in \QCoh(\mstack X)$ is perfect. Indeed, by descent this is enough to check after taking a pull-back on a smooth cover by a smooth $R$-scheme of finite type where we get a regular embedding $i'\colon S'\ra X'$ which is automatically a locally complete intersection in $X'$. After refining the cover in Zariski topology, we can assume the intersection is actually complete and resolve $i'_*\mc O_{S'}$ by the Koszul complex. This shows that the functor $i^!$ also restricts to a functor from $\QCoh(\mstack X)^\perf$ to $\QCoh(\mstack S)^\perf$. Indeed, one has a formula $i^!\mc F\simeq \ul\Hom_{\QCoh(\mstack X)}(i_*\mc O_{\mstack S},\mc F)$ where the latter has support on $\mstack S$ and is considered as an object of $\QCoh(\mstack S)$. By smooth descent for $\ul \Hom$ (\cite[Lemma A.1.1]{Preygel}) we can reduce to the case  $\Spec A \simeq S\ra X\simeq \Spec B$ of smooth affine schemes over $R$, where the sheaf $\ul \Hom_{\QCoh(X)}(i_*\mc O_S,\mc O_X)$ is computed by the dual to the Koszul complex and thus is bounded and has finitely generated cohomology modules (and this way belongs to $\Coh(A)$). 
	
	Given all this, by a similar argument to \Cref{prop:KN-stratum induction step} it is enough to get a suitable semiorthogonal decomposition of $\Coh(\mstack X)\simeq \QCoh(\mstack X)^\perf$ in terms of $\Coh_{\mstack S}(\mstack X)$ and $\Coh(\mstack X^{ss})$ and show that all $\Hom$'s in $\Coh(\mstack S)$ lie in $\Coh^+(R)$. Since $\Hom_{\Coh(\mstack S)}(\mc F,\mc G)\simeq \RG(\mstack S,\ul \Hom_{\QCoh(\mstack S)}(\mc F,\mc G))$ with $\ul \Hom_{\QCoh(\mstack S)}(\mc F,\mc G)\in \Coh(\mstack S)$ and $\mstack Z$ is cohomologically proper for the latter point it is enough to show that $\ev_0\colon \mstack S \ra \mstack Z$ is cohomologically proper. By descent and base change this can be checked on a (suitable) smooth cover of $\mstack X$; namely we can use \cite[Lemma 6.11]{Alper-HL-Heinloth} to produce a smooth cover $[X/\mbb G_m]\ra \mstack X$ where $X$ is an affine scheme and such that the preimage of $\mstack S$ is given by a union of connected components of $[X^+/\mbb G_m]$ (in the terminology of Sections \ref{sec:G_m-actions} and \ref{sec:BB-complete}). The centrum $\mstack Z$ is then given by a union of the corresponding components of $[X^0/\mbb G_m]$ and the needed statement follows from \Cref{lem: from + to 0}.
	
	It remains to deal with the semiorthogonal decomposition. By \cite[Theorem 1.9.2]{HL-prep} for each integer $w\in \mbb Z$ we have a decomposition for the bounded above category $\Coh^-(\mstack X)$
	$$\Coh^-(\mstack X)=\langle \ \! \Coh^-_{\mstack S}(\mstack X)_{< w}, \mr{G}_w^- , \Coh_{\mstack S}^-(\mstack X)_{\ge w} \ \!\rangle$$
	in terms of certain subcategories $\Coh_{\mstack S}(\mstack X)_{< w},\Coh_{\mstack S}(\mstack X)_{\ge w} \subset  \Coh_{\mstack S}(\mstack X)$ forming a semiorthogonal decomposition of $ \Coh_{\mstack S}(\mstack X)$ on its own and with $\mr{G}_w^-$ being isomorphic to $\Coh^-(\mstack X^{ss})$ via the restriction $\mc F\mapsto \mc F|_{\mstack X^{ss}}$. This decomposition holds without extra assumptions on $\mstack X$, however if we assume $\mstack X,\mstack S$ are smooth (and thus in particular the embedding $i\colon \mstack S \ra \mstack X$ is regular) the proof of \cite[Proposition 2.1.2]{HL-prep} goes through without any changes, giving the analogous decomposition for $\QCoh^\perf$ (and thus also for $\Coh$).
\end{proof}

Let's now assume that we have a smooth finite type stack $\mstack X$ over a characteristic 0 field $F$ endowed with a $\Theta$-stratification. Filtering $F$ by regular Noetherian rings $R\subset F$ as in \Cref{sec:Examples of spreadable Hodge-proper stacks} we get a smooth spreading $\mstack X_R$; we can also spread the open substacks $\mstack X_{\le \alpha}$ to get open substacks $\mstack X_{\le \alpha, R}\subset \mstack X_R$. Applying the following lemma inductively we can in fact assume that this gives a $\Theta$-stratification. Before stating the lemma note that one has a natural monoid structure on $\Theta$ induced by the multiplication map $\mbb A^1\times \mbb A^1 \ra \mbb A^1$. Having a stack $\mstack Y$ with an action of $\Theta$ in the homotopy category one gets a baric structure on $\QCoh(\mstack Y)$ (see \cite[Section 1.1]{HL-prep}); in particular for each weight $w\in \mbb Z$ one has a semiorthogonal decomposition $\langle\ \! \! \QCoh(\mstack Y)^{\ge w}, \QCoh(\mstack Y)^{<w}\ \! \! \rangle$. Moreover if $\mstack Y$ is smooth over a regular Noetherian base ring $R$ this also defines a decomposition for coherent sheaves: $\Coh(\mstack Y) =\langle\ \! \! \Coh(\mstack Y)^{\ge w}, \Coh(\mstack Y)^{<w}\ \! \! \rangle$ (\cite[Proposition 1.2.1(3)]{HL-prep}); we will denote by $\beta^{\ge w},\beta^{<w}$ the corresponding truncation functors. We note that any $\Theta$-stratum $\mstack S\subset \Filt(\mstack X)$ comes with a natural $\Theta$-action.

\begin{lem}\label{lem: Theta-strata spread out}
	Let $\mstack X_F$ be a smooth Artin stack of finite type over $F$ with affine diagonal and let $i_F\colon \mstack S_F\hookrightarrow \mstack X_F$ be a $\Theta$-stratum. Then one has a spreading $i_R\colon \mstack S_R\hookrightarrow \mstack X_R$ which is a $\Theta$-stratum as well.
\end{lem}
\begin{proof}
	The key step here is to use the intrinsic description of $\Theta$-strata (\cite[Section 1.4]{HL-prep}). Namely, over any Noetherian base a closed substack $i\colon \mstack S\hookrightarrow \mstack X$ with an action $a\colon \Theta \times \mstack S\ra \mstack S$ gives a map $\phi\colon \mstack S \ra \Filt(\mstack X)$ defined as a composition $$\mstack S\xra{a^*} \ul{\Map}(\Theta, \mstack S)\xra{\circ i} \ul{\Map}(\Theta, \mstack X).$$ By \cite[Proposition 1.4.1]{HL-prep} if $\mstack X$ is locally of finite presentation with affine diagonal the map $\phi$ is also  a closed embedding; moreover, $\phi\colon \mstack S \hookrightarrow \Filt(\mstack X)$ defines a $\Theta$-stratum if and only if $\mbb L_{\mstack S/\mstack X}\in \QCoh(\mstack S)^{\ge 1}$. 
	
	The stack $\mstack X_F$ is smooth, thus by \cite[Corollary 1.3.2.1]{HL-prep} $\Filt(\mstack X_F)$ is smooth and so $\mstack S_F$ is smooth. The stack $\mstack S_F$ is also a closed substack of a stack of finite type and so is of finite type over $F$ as well. Using \Cref{fp_over_filtered_limit} we can spread the natural action $a_F\colon \Theta_F\times \mstack S_F\ra \mstack S_F$ and the closed embedding $i_F\colon \mstack S_F\hookrightarrow \mstack X_F$ to get an action $a_R\colon \Theta_R\times \mstack S_R\ra \mstack S_R$ and a closed embedding $i_R\colon \mstack S_R\hookrightarrow \mstack X_R$. Moreover we can assume $\mstack S_R,\mstack X_R$ are smooth and of finite type over $R$. By the above description $\mstack S_R$ is a $\Theta$-stratum if and only if $\mbb L_{\mstack S_R/\mstack X_R}\in \QCoh(\mstack S_R)^{\ge 1}$; this is equivalent to $\beta^{<1}(\mbb L_{\mstack S_R/\mstack X_R})\simeq 0$. Note that by smoothness $\mbb L_{\mstack S_R/\mstack X_R}$ and thus also $\beta^{<1}(\mbb L_{\mstack S_R/\mstack X_R})$ are coherent. Since the restriction to $\mstack X_F$ of $\beta^{<1}(\mbb L_{\mstack S_R/\mstack X_R})$ is given by $\beta^{<1}(\mbb L_{\mstack S_F/\mstack X_F})$ which is zero, we get that $\beta^{<1}(\mbb L_{\mstack S_R/\mstack X_R})\simeq 0$ after a finite localization of $R$ (indeed this is enough to check for a pull-back to a smooth cover by a scheme, where this follows from the Chevalley's consructibility theorem).
\end{proof}
\begin{rem}
	We needed to use the intrinsic description of the $\Theta$-strata in \Cref{lem: Theta-strata spread out} because the stack $\Filt(\mstack X)$ is only locally finitely presentable; thus we can't directly apply \Cref{fp_over_filtered_limit} to spread $\mstack S_F\ra \Filt(\mstack X_F)$ or compare $\Filt(\mstack X_R)$ with some other spreading $\Filt(\mstack X)_R$.
\end{rem}

From the discussion above we deduce the following result: 

\begin{cor} \label{cor:degeneration for Theta-stratified stacks}
	Let $\mstack X$ be a smooth Artin stack of finite type with affine diagonal over $F$ and let $\{\mstack X_{\le \alpha},\mstack S_{\alpha}\}$ be a finite $\Theta$-stratification of $\mstack X$. If the centra $\mstack Z_\alpha$ of $\Theta$-strata and the semistable locus $\mstack X^{ss}$ are cohomologically properly spreadable then $\mstack X$ is also cohomologically properly spreadable. In particular the Hodge-de Rham spectral sequence degenerates for $\mstack X$.
\end{cor}
\begin{proof}
 By induction on $|I|$ using \Cref{lem: Theta-strata spread out} we can spread out the $\Theta$-stratification (with the properties as in the proof of the latter). The centrum $\mstack Z_{\alpha,R}$ is a closed substack of $\mstack S_{\alpha, R}$, thus it is also of finite type and is a spreading of $\mstack Z_\alpha$. Enlarging $R\subset F$ so that all $\mstack Z_{\alpha,R}$ and $\mstack X^{ss}_R$ become cohomologically proper we then use \Cref{prop:Theta-stratified stacks are cohomologically proper} to get that so is $\mstack X_R$.
\end{proof}

In \cite{HL-instability} various ways of constructing a $\Theta$-stratifications on a stack are discussed in great detail. We will stop on a single example given by a KN-stratification of a global quotient stack. 

\begin{ex}\label{ex:Kempf-Ness}
	In  \cite{Teleman} Teleman showed the degeneration of the Hodge-to-de Rham spectral sequence for the quotient stacks $[X/G]$ with $G$ reductive under the condition that the action on $X$ is KN-complete (see Section 1 of \textit{loc.cit.} for the definition of a KN-complete action and Section 7 for the proof of degeneration). We comment on how to deduce his results (in a slightly more general form) from \Cref{Hodge_degeneration} and \Cref{cor:degeneration for Theta-stratified stacks}. 
	
	A KN-stratification of a variety $X$ with a $G$-action is a stratification
$$
X=X^{\mr{ss}}\cup \bigcup_{\alpha\in I} S_\alpha
$$
by locally closed $G$-invariant subschemes satisfying the following properties:

\begin{itemize}
	\item 
For each $\alpha$ there should exist a one-parameter subgroup $\lambda_\alpha\colon \mathbb G_m\hookrightarrow G$; let $L_\alpha\subset G$ be the centralizer of $\lambda_\alpha(\mathbb G_m)$. The KN-strata $S_\alpha$ should come as follows: for each $\alpha\in I$ there should exist an open subvariety $Z_\alpha\subset X^{\lambda_\alpha(\mathbb G_m)}$ of the fixed locus of $\lambda_\alpha(\mathbb G_m)$ such that $S_\alpha$ is given by the $G$-span $G\cdot Y_\alpha$ of the corresponding attracting locus
$$
Y_\alpha\coloneqq \{x\in X  | \lim_{t\ra 0}\lambda_\alpha(t)\cdot x\in Z_\alpha\}.
$$
The subvariety $Z_\alpha$ is called the centrum of $S_\alpha$ and it is endowed with the natural action of the centralizer $L_\alpha$. The attracting locus $Y_\alpha$ is endowed with the natural action of the (automatically parabolic) subgroup $P_\alpha\subset G$ of elements $p\in G$ for which the limit of $\lambda_\alpha(t)p\lambda_\alpha(t)^{-1}$ in $G$ as $t\ra 0$ exists.

\item Each KN-stratum $S_\alpha$ should satisfy one further property: namely, the natural action map $G\times Y_\alpha\ra X$ should induce an isomorphism $S_\alpha\simeq G\overset{P_\alpha}{\times} Y_\alpha$. In the context of GIT the KN-stratification usually comes as follows: the centra $Z_\alpha\subset X^{\lambda_\alpha(\mathbb G_m)}$ are the semistable loci of the action of $L_\alpha'\coloneqq L_\alpha/\lambda_\alpha(\mathbb G_m)$ on $X^{\lambda_\alpha(\mathbb G_m)}$; we note that in this case the $L_\alpha$-action on $Z_\alpha$ is automatically locally linear and thus \Cref{thm:proper coarse moduli} applies to $\mstack Z_\alpha\coloneqq [Z_\alpha/L_\alpha]$. We will assume further on that we are in this setting.

\end{itemize}

A KN-stratification is called \textit{complete} if the categorical quotients  $Z_\alpha/\!/ L_\alpha$ and $X^{\mr{ss}}/\!/G$ are projective. We will call it \textit{locally linear} if the actions of $L_\alpha$ on $Z_\alpha$ and the action of $G$ on $X^{\mr{ss}}$) are locally.

Let $\mstack X\coloneqq [X/G]$, $\mstack X^{\mr{ss}}\coloneqq [X^{\mr{ss}}/G]$, $\mstack S_i\coloneqq[S_\alpha/G]\simeq [Y_i/P_i]$ and $\mstack Z_\alpha\coloneqq [Z_\alpha/L_i]$. Find a total ordering\footnote{It is not hard to see that such a total ordering exists for any stratification. In the GIT "projective-over-affine" case it usually comes via the values of the Hilbert-Mumford potential. Examples of more general potentials which apply to other situations can be found in \cite[Section 4]{HL-instability}.} on $I$ such that for $X_{\le \alpha}\coloneqq X\setminus \cup_{\alpha'>\alpha} S_{\alpha'}$ the embedding $i_\alpha\colon S_\alpha \rightarrow X_{\le \alpha}$ is closed. By the description of $\Theta$-strata in a quotient stack (\cite[Theorem 1.37]{HL-instability}) applied to $\mstack X_{\le \alpha}$ one can see that $\mstack S_\alpha$ is naturally a $\Theta$-stratum and that $\mstack Z_\alpha$ is its centrum. 
 Then, given the KN-stratification is complete and locally linear, by \Cref{thm:proper coarse moduli} we get that the stacks $\mstack Z_{\alpha}$ and $\mstack X^{ss}$ are cohomologically properly spreadable. Thus by \Cref{cor:degeneration for Theta-stratified stacks} $\mstack X$ is also cohomologically properly spreadable and the Hodge-de Rham spectral sequence degenerates for it. 
 
 Note that the same proof works if the categorical quotients $Z_\alpha/\!/ L_\alpha$ and $X^{\mr{ss}}/\!/G$ are proper but not necessarily projective.
\end{ex}

\begin{rem} In \cite{HLP_equiv_noncom} the non-commutative Hodge-to-de Rham degeneration was proved for a slightly more general definition of a KN-complete stratification: namely one does not need to assume that the $L_i$-action on $Z_i$ is locally linear, only that there exists a good quotient $q\colon [Z_i/L_i]\ra Z_i/\!/L_i$. In this case the strata are not necessarily covered by \Cref{thm:proper coarse moduli} (and the above strategy) but we still hope that they could be cohomologically properly spreadable.  
More generally, it would be interesting to answer the following question:
\begin{question}
Let $q\colon \mstack Y \ra Y$ be a good moduli space (e.g. see the Definition in Section 1.2 of \cite{Alper}) and assume that $Y$ is a proper algebraic space. Is it true that $\mstack Y$ is cohomologically properly spreadable?
\end{question}

The notion of a good moduli space does not spread out well unless the stabilizers are \textit{nice}, i.e. extension of a finite group by a torus. Thus we think it would be very interesting to understand if the property of having a good proper moduli space in characteristic 0 implies any cohomological properness for its mixed characteristic spreadings (as it happens in the case of $BG$ for a reductive $G$).
\end{rem}

Motivated by Questions 1.3.2 and 1.3.3 of \cite{RelaxedProperness}, one can also ask the following: 

\begin{question}\label{que: is formally proper spreadable?}
Let $\mstack X$ be a formally proper stack (in the sense of Definition 1.1.3 of \cite{RelaxedProperness}) over $F$. Is $\mstack X$ cohomologically properly spreadable?
\end{question}

\appendix 
\numberwithin{thm}{section}
\section{Computation of $H^*(B\mathbb G_a,\mc O_{B\mathbb G_a})$ over $\mathbb Z$}\label{appendix:BG_a}
In this section we compute cohomology (with coefficients in the structure sheaf) of the classifying stack $B\mathbb G_a$ of the additive group $\mathbb G_a$ over the ring of integers $\mathbb Z$. Unfortunately, were not able to locate this result in the literature, so we do the computation here based on the Jantzen's computation of cohomology of $B\mathbb G_{a,\mathbb F_p}$. This result is included for completeness only and will not be used anywhere else in the paper.

We start by constructing sufficiently many elements in the first few cohomology groups of $\mc O_{B\mathbb G_a}$, the rest of the computation will then unravel from there. We consider the action of $\mathbb G_m$ on $\mathbb G_a$ with $t\in \mathbb G_m$ acting on the variable $x$ in $\mc O(\mathbb G_a)\simeq \mathbb Z[x]$ by $t\colon x \mapsto t^2x$ (note the square in the formula). This makes  $H^*(B\mathbb G_a,\mc O_{B\mathbb G_a})$ into a $\mathbb G_m$-representation\footnote{More precisely we should take the corresponding semidirect product $\mathbb G_a \rtimes \mathbb G_m$ with the projection $p\colon \mathbb G_a \rtimes \mathbb G_m\ra \mathbb G_m$ and then consider the direct image $p_*\mc O_{B(\mathbb G_a \rtimes \mathbb G_m)}\in \QCoh(B\mathbb G_m)$; by base change its fiber over the point $\Spec \mathbb Z\ra B\mathbb G_m$ is given by $R\Gamma(B\mathbb G_a, \mc O_{B\mathbb G_a})$.}, thus providing an extra $\mathbb Z$-grading which we denote by ${H^*(B\mathbb G_a,\mc O_{B\mathbb G_a})}_*$ using the lower indexing.
	
	The cohomology ${H^*(B\mathbb G_a,\mc O_{B\mathbb G_a})}$ is the same as the cohomology of $\mathbb G_a$ with coefficients in the trivial module $\mathbb Z$. We can compute it via the standard complex $C^\bullet(\mathbb G_a,\mathbb Z)\coloneqq \mathbb Z[\mathbb G_a^\bullet]$ (see \cite[Section 4.14]{Jantzen}):
	$$
	0\ra \mathbb Z \xra{d_0} \mathbb Z[x] \xra{d_1} \mathbb Z[y,z] \xra{d_2} \ldots.
	$$
	The action of $\mathbb G_m$ extends to $C^\bullet(\mathbb G_a,\mathbb Z)$ giving a $\mathbb Z$-grading which on each term $\mathbb Z[\mathbb G_a^n]\simeq \mathbb Z[x_1,\ldots, x_n]$ is given by the doubled degree of a polynomial. This splits $C^\bullet(\mathbb G_a,\mathbb Z)$ as a direct sum of graded components $C^\bullet(\mathbb G_a,\mathbb Z)_{n}$ with $i\ge 0$. Note that all non-zero components have even weight.
	
	The 0-th component $C^\bullet(\mathbb G_a,\mathbb Z)_{0}$ has $\mathbb Z$ in every component and is just the complex associated to the constant simplicial set $\mathbb Z$; thus $C^\bullet(\mathbb G_a,\mathbb Z)_{0}\simeq \mathbb Z[0]$. The second graded component $C^\bullet(\mathbb G_a,\mathbb Z)_{2}$ looks as
	$$
	0\ra 0\ra \mathbb Z\cdot x \ra \mathbb Z\cdot x\oplus \mathbb Z\cdot y \ra \ldots 
	$$
	and is the complex associated to the simplicial interval $\Delta_1$ (or rather the corresponding free abelian group) shifted by 1; thus  $C^\bullet(\mathbb G_a,\mathbb Z)_{2}\simeq \mathbb Z[-1]$. Let $v_1$ be the corresponding generator of $H^1(\mathbb G_a,\mathbb Z)_{2}$. 
	
	For any $n\in \mathbb N$ consider $\Phi_n(y,z)\coloneqq d_1(x^n)=(y-z)^n-y^n+z^n$ as an element of $C^2(\mathbb G_a,\mathbb Z)_{2n}$. Note that $x^n$ is the generator of $C^1(\mathbb G_a,\mathbb Z)_{2n}$ and thus, first, $H^1(\mathbb G_a,\mathbb Z)_{2n}=0$ unless $n=1$ and, second, $\Phi_n(y,z)$ generates the group of coboundaries $B^2(\mathbb G_a,\mathbb Z)_{2n}\subset C^2(\mathbb G_a,\mathbb Z)_{2n}$ over $\mathbb Z$. In particular $H^1(\mathbb G_a,\mathbb Z)\simeq \mathbb Z v_1$. Note also that since ${p}\vert{p^i\choose{k}}$ if $0<k<p^i$, $\Phi_{p^i}(y,z)$ is divisible by $p$ for any $i\ge 1$; moreover $d_2\left(\frac{\Phi_{p^i}(x,y)}{p}\right)=0$ since $d_2(\Phi_{p^i}(x,y))=0$ and all terms in the complex are free $\mathbb Z$-modules. Thus for any prime $p$ and $i>0$ we get a class $v_{p_i}\coloneqq \left[\frac{\Phi_{p^i}(y,z)}{p}\right]\in H^2(\mathbb G_a,\mathbb Z)_{p^i}$ such that $p\cdot v_{p_i}=0$. This way, we get a map 
	\begin{equation}\label{eq:chi}
	\chi\colon (\mathbb Z\oplus \mathbb Zv_1) \otimes_{\mathbb Z}\Sym_{\mathbb Z}^*\left(\bigoplus_{p}\mathbb F_p v_p\oplus \mathbb F_pv_{p^2}\oplus \ldots\right)\ra H^*(\mathbb G_a,\mathbb Z)
	\end{equation}
	which is an isomorphism on $H^1$ and an injection on $H^2$. In the context of Hodge-properness we see that this is already very bad: for any $p$ the $p$-torsion in $H^2(B\mathbb G_a,\mc O_{B\mathbb G_a})$ is infinitely generated.
	
	To compute $H^*(\mathbb G_a,\mathbb Z)$ fully we will need a description of the cohomology of $\mathbb G_a$ over $\mathbb F_p$ (see e.g. \cite[Lemma 4.22 and Proposition 4.27]{Jantzen}). The first cohomology 
	$
	H^1({(\mathbb G_a)}_{\mathbb F_p},{\mathbb F_p})
	$ is a span $\mathbb F_pw_1 \oplus \mathbb F_p w_p\oplus \mathbb F_pw_{p^2}\oplus \ldots$ of classes $w_1,w_p,w_p^2 \ldots$ with the $\mathbb G_m$-weight of each $w_{p^i}$ given by $2p^i$. Moreover since by the universal coefficient formula the reduction map $H^1(\mathbb G_a,\mathbb Z)/p \ra H^1({(\mathbb G_a)}_{\mathbb F_p},{\mathbb F_p})$ is an injection and preserves the $\mathbb G_m$-weights we get that the class $w_1$ equals to the reduction $\ol v_1$ of $v_1$ (up to a scalar) for any $p$. From the computation in \cite[second paragraph on p.60]{Jantzen} it also follows the reductions $\ol v_{p^i}\in H^2({(\mathbb G_a)}_{\mathbb F_p},{\mathbb F_p})_{2p^i}$ are non-zero (namely in the notations of \cite{Jantzen} $\ol v_{p^i}$ is equal to $\beta(x^{p^i})$ up to a sign change in the second variable) and linearly independent (since they have different $\mathbb G_m$-weights). 
	
	\begin{lem}\label{lem: torsion in cohomology is elementary}
		For any prime $p$ the $p$-primary part of $H^*(\mathbb G_a,\mathbb Z)$ is elementary, i.e. it is killed by $p$. 
	\end{lem}
	\begin{proof}
		The statement will follow from the computation of the Bockstein differential $$\beta_p\colon H^*({(\mathbb G_a)}_{\mathbb F_p},{\mathbb F_p})\ra H^{*+1}({(\mathbb G_a)}_{\mathbb F_p},{\mathbb F_p}).$$ Namely we will use the fact that if we have a class $[c]\in H^*({\mathbb G_a},{\mathbb Z})$ then its reduction $[\ol c]\in H^*({(\mathbb G_a)}_{\mathbb F_p},{\mathbb F_p})$ is killed by $\beta_p$ and that if $[\ol c]\in \Im \ \!\beta_p$ then $p\cdot [c]=0$ (in other words if the class of $[\ol c]$ in the cohomology with respect to Bockstein is 0 then $[c]$ is killed by $p$).
		
		There are two cases. If $p=2$ there is an isomorphism
		$$
		H^*({(\mathbb G_a)}_{\mathbb F_2},{\mathbb F_2})\simeq \Sym_{\mathbb F_2}^*\left( H^1({(\mathbb G_a)}_{\mathbb F_2},{\mathbb F_2})\right)\simeq  \mathbb F_2[w_1,w_2,w_{4},\ldots].
		$$
		Consequently, all $\mathbb G_m$-weights in $H^2({(\mathbb G_a)}_{\mathbb F_2},{\mathbb F_2})$ are given as sums $2^i+2^j$ for $i,j\ge 0$. Since the reduction $\ol v_{2^i}\in H^2({(\mathbb G_a)}_{\mathbb F_2},{\mathbb F_2})$ is non-zero and its weight is $2^{i+1}$ the only option for it is $ w_{2^{i-1}}^2$.
		Also, by the universal coefficient formula the class $w_{2^i}$, $i>0$ should come as the only non-zero element of $\Tor_1(H^2(\mathbb G_a,\mathbb Z),\mathbb F_2)$ of weight $2^{i+1}$. The latter group is equal to the 2-torsion in $H^2(\mathbb G_a,\mathbb Z)$ and contains $v_{2^i}$ which has the correct weight. From the properties of the Bockstein operator it follows that $\beta_2(w_{2^i})$ is equal to $\ol v_{2^i}=w_{2^{i-1}}^2$ if $i>0$. Also $\beta_2(w_1)=0$ (since $w_1=\ol v_1$) and since  $\beta_2$ is a differentiation this, together with the above, defines it uniquely\footnote{In terms of the identification $H^*({(\mathbb G_a)}_{\mathbb F_2},{\mathbb F_2})\simeq \mathbb F_2[w_1,w_2,w_{4},\ldots]$, $\beta_2$ acts as the well-defined vector field $\sum_{i=1}^\infty w_{2^{i-1}}^2\frac{\partial}{\partial w_{2^i}}$.}. Note that $\beta_2$ is $\mathbb F_2[w_1, w_2^2, w_4^2, \ldots]$-linear and that $\mathbb F_2[w_1, w_2, w_4, \ldots]$ is free over $\mathbb F_2[w_1, w_2^2, w_4^2, \ldots]$ with basis given by $w_{2^I}\coloneqq w_{2^{i_1}}w_{2^{i_2}}\ldots w_{2^{i_k}}$ where $I=\{i_1,\ldots, i_k\}$ runs over finite subsets of $\mathbb Z_{>0}$. We turn $(\mathbb F_2[w_1, w_2, w_4, \ldots],\beta_2)$ into a complex of $\mathbb F_2[w_1, w_2^2, w_4^2, \ldots]$-modules be defining another (homological) grading, namely putting $\deg_* w_{2^I}$ to be equal to $|I|$. In fact this way we can identify $(\mathbb F_2[w_1, w_2, w_4, \ldots]_*,\beta_2)$ with the Koszul complex $\mr{Kos}_{\mathbb F_2[w_1, w_2^2, w_4^2, \ldots]}(w_1^2,w_2^2,w_4^2,\ldots)_*$ for the infinite sequence $(w_1^2,w_2^2,w_4^2,\ldots)$;\footnote{We warn the reader that this is only an isomorphism $\mathbb F_2[w_1, w_2^2, w_4^2, \ldots]$-dg-modules and not dg-algebras.} indeed, one can map the $k$-th term $\mr{Kos}_{\mathbb F_2[w_1, w_2^2, w_4^2, \ldots]}(w_1^2,w_2^2,w_4^2,\ldots)_k\simeq \Lambda^k_{\mathbb F_2[w_1, w_2^2, w_4^2, \ldots]}(\mathbb F_2[w_1, w_2^2, w_4^2, \ldots]^{\oplus \infty})$ with the basis $(e_1,e_2,e_3,\ldots)$ of $\mathbb F_2[w_1, w_2^2, w_4^2, \ldots]^{\oplus \infty}$, associated to the elements $(w_1^2,w_2^2,w_4^2,\ldots)$,  to the $k$-th graded component of $\mathbb F_2[w_1, w_2, w_4, \ldots]$  by sending $e_I\coloneqq e_{i_1}\wedge \ldots \wedge e_{i_k}$ to $w_{2^I}$. It is easy to see that the Koszul differential goes exactly to $\beta_2$. Since the sequence $(w_1^2,w_2^2,w_4^2,\ldots)$ is regular we get that the cohomology of $(H^2({(\mathbb G_a)}_{\mathbb F_2},{\mathbb F_2}),\beta_2)$ is given by 
		$$\mathbb F_2[w_1, w_2^2, w_4^2, \ldots]/(w_1^2, w_2^2, w_4^2,\ldots)\simeq \mathbb F_2[w_1]/w_1^2$$ and is spanned by 1 and $w_1$ over $\mathbb F_2$. But we know that $w_1=\ol v_1$ (and of course 1) is a reduction of a non-torsion class $v_1$ (resp. 1). Thus we get the statement.
		
		If $p$ is odd then
		$$
		H^*({(\mathbb G_a)}_{\mathbb F_p},{\mathbb F_p})\simeq \Lambda_{\mathbb F_p}^*(w_1, w_p,w_{p^2},\ldots)\otimes_{\mathbb F_p}\Sym_{\mathbb F_p}^*( \ol v_p,\ol v_{p^2},\ol v_{p^3},\ldots),
		$$
		By a similar reasoning to the $p=2$ case we get that $\beta_p(w_{p^i})=\ol v_{p^i}$ (at least up to a scalar) for $i>0$ and that $\beta_p(w_1)=0$. Similarly, $\beta_p$ is ${\mathbb F_p}[ v_p,v_{p^2},v_{p^3},\ldots]$-linear and $H^*({(\mathbb G_a)}_{\mathbb F_p},{\mathbb F_p})$ is a free module over ${\mathbb F_p}[ v_p,v_{p^2},v_{p^3},\ldots]$ with the basis given by $w_{p^I}\coloneqq w_{p^{i_1}}\wedge\ldots \wedge w_{p^{i_k}}$ where $I=\{i_1,\ldots, i_k\}$ with $i_1<\ldots i_k$ runs over finite subsets of $\mathbb Z_{\ge 0}$. Defining a new (homological) grading on  $(H^*({(\mathbb G_a)}_{\mathbb F_p},{\mathbb F_p})$ by putting $\deg_* w_{p^I}=|I|$ and $\deg_* v_{i}=0$ we view $(H({(\mathbb G_a)}_{\mathbb F_p},{\mathbb F_p})_*,\beta_p)$ as a complex of $\Sym_{\mathbb F_p}^*[ \ol v_p,\ol v_{p^2},\ol v_{p^3},\ldots]$-algebras, which in fact is identified with the product (now as dg-algebras) 
		$$ \Lambda_{\mathbb F_p}^*[w_1]\otimes_{\mathbb F_p} \mr{Kos}_{{\mathbb F_p}[ \ol v_p,\ol v_{p^2}, \ol v_{p^3},\ldots]}(\ol v_p, \ol v_{p^2},\ol v_{p^3},\ldots)\simeq H^*({(\mathbb G_a)}_{\mathbb F_p},{\mathbb F_p})
		$$
		where the differential on $w_1$ is 0. Indeed one can define a map by sending each generator $e_I\in \mr{Kos}_{|I|}$ to $w_{p^I}$ and we leave it to the reader to check that it is an isomorphism. Since $\ol v_p,\ol v_{p^2},\ol v_{p^3},\ldots\in {\mathbb F_p}[ \ol v_p,\ol v_{p^2},\ol v_{p^3},\ldots]$ is a regular sequence in the case of an odd $p$ we also get that the cohomology of $\beta_p$ is spanned by 1 and $w_1$ over $\mathbb F_p$, and they come as reductions of non-torsion classes. This finishes the proof.
	\end{proof} 
	
	We finish the description of $H^*(\mathbb G_a,\mathbb Z)$. By \Cref{lem: torsion in cohomology is elementary} the $p$-primary part of $H^*(\mathbb G_a,\mathbb Z)$ (as a non-unital algebra) is killed by $p$ and thus can be described as $\Im \ \! \beta_p\subset H^*((\mathbb G_a)_{\mathbb F_p},\mathbb F_p)$ via the reduction map. More explicitly $\Im \ \!\beta_p$ is freely  generated by elements $\beta_p(w_I)$ as a module over ${\mathbb F_p}[ \ol v_p,\ol v_{p^2},\ol v_{p^3},\ldots]$ in the notations of \Cref{lem: torsion in cohomology is elementary}. The elements $\beta_p(w_I)$ are not algebraically independent over ${\mathbb F_p}[ \ol v_p,\ol v_{p^2},\ol v_{p^3},\ldots]$ and it seems hard to describe all the relations between them; but still this description is somewhat nice, since there is only finite number of $\beta_p(w_I)$ of a given cohomological degree. To finish the computation over $\mathbb Z$ it only remains to see what happens with the powers of $v_1$. Since $v_1$ is of cohomological degree $1$, $v_1^2$ is $2$-torsion and we saw in the course of proof of \Cref{lem: torsion in cohomology is elementary} that in fact $v_1^2=v_2$. All in all this gives the following description of $H^*(B\mathbb G_a,\mc O_{B\mathbb G_a})$:
	
	\begin{prop}\label{prop: description of the cohomology of BG_a}
		We have
		$$
		H^*(B\mathbb G_a,\mc O_{B\mathbb G_a})\simeq \left(\mathbb Z[v_1]\oplus \left(\bigoplus_p \Im\ \! \beta_p\right)\right)\bigg/ v_1^2 =v_2.
		$$	
	\end{prop}
	
	Also, returning to the map $\chi$ (see \Cref{eq:chi}), we get a subalgebra 
	$$
	A=\left(\mathbb Z[v_1] \otimes_{\mathbb Z}\Sym_{\mathbb Z}^*\left(\bigoplus_{p}\mathbb F_p v_p\oplus \mathbb F_pv_{p^2}\oplus \ldots\right)\right) \bigg/ v_1^2=v_2 \subset H^*(B\mathbb G_a,\mc O_{B\mathbb G_a}),
	$$ and the algebra $H^*(B\mathbb G_a,\mc O_{B\mathbb G_a})$ is generated by 1 and $\beta_p(w_I)$ (for various $p$ and $I$) as an $A$-module. More precisely one can check that we have a direct sum decomposition
	$$
	H^*(B\mathbb G_a,\mc O_{B\mathbb G_a})\simeq A \oplus  \bigoplus_{p,I} A\cdot \beta_p(w_I),
	$$
	where for each $\beta_p(w_I)$ the submodule $A\cdot \beta_p(w_I)$ is just isomorphic to (non-derived) quotient $A/p$.

\begin{rem} 
	Quite remarkably the cohomology $H^*(B\mathbb G_a,\mc O_{B\mathbb G_a})$ turns out to be directly related to the cohomology of the Eilenberg-Maclane space $K(\mathbb Z,3)$\footnote{Recall that $K(\mathbb Z,n)$ for $n\ge 1$ is the unique (up to homotopy) space such that $\pi_n(K(\mathbb Z,n))=\mathbb Z$ and $\pi_i(K(\mathbb Z,n))=0$ for $i\neq n$.}:
	namely there is an isomorphism 
	$$
	H^{n}_{\mr{sing}}(K(\mathbb Z,3),\mathbb Z)\simeq \bigoplus_{i=0}^n H^i(B\mathbb G_a,\mc O_{B\mathbb G_a})_{n-2i},
	$$
	which in fact extends to the isomorphism of the graded algebras 
	$$
	\bigoplus_{n\ge 0} H^{n}_{\mr{sing}}(K(\mathbb Z,3),\mathbb Z)\simeq \bigoplus_{n\ge 0} \left(\bigoplus_{i=0}^n H^i(B\mathbb G_a,\mc O_{B\mathbb G_a})_{n-2i}\right).
	$$
	We comment more on this. Indeed, $K(\mathbb Z,3)\simeq B(K(\mathbb Z,2))\simeq B\mathbb C\mathbb P^\infty$. Realizing $K(\mathbb Z,3)$ as the colimit of the simplicial diagram
	$$
	\colim \left(\xymatrix{\ldots \ar@<0.8ex>[r]\ar[r]\ar@<-0.8ex>[r]& \mathbb C\mathbb P^\infty\times \mathbb C\mathbb P^\infty \ar@<0.4ex>[r]\ar@<-0.4ex>[r]& \mathbb C\mathbb P^\infty \ar[r]& {*} }\right)\xymatrix{\ar[r]^\sim&} K(\mathbb Z,3)
	$$
	we get a spectral sequence 
	$$
	E_1^{n,q}=H^q_{\mr{sing}}((\mathbb C\mathbb P^\infty)^n,\mathbb Z) \Rightarrow H^{n+q}_{\mr{sing}}(K(\mathbb Z,3),\mathbb Z).
	$$

	The cohomology $H^{*}_{\mr{sing}}(\mathbb C\mathbb P^\infty,\mathbb Z)\simeq \mathbb Z[x]$, $\deg x=2$ has a natural Hopf algebra structure with comultiplication induced by the addition $m\colon \mathbb C\mathbb P^\infty\times \mathbb C\mathbb P^\infty \ra \mathbb C\mathbb P^\infty$. It is easy to see that $m^*(x)=x\otimes 1 + 1\otimes x$ and so the corresponding affine group scheme is $\mathbb G_a$; moreover, the cohomological grading corresponds exactly to the $\mathbb G_m$-action on $\mathbb G_a$ which we considered before.
	Via this identification and the K\"unneth formula, the first page $E_1^{n,q}$ of the spectral sequence above is identified the standard complex $C^\bullet(\mathbb G_a,\mathbb Z)$:
	$$
	0\ra \mathbb Z \xra{d_0} \mathbb Z[x] \xra{d_1} \mathbb Z[x]\otimes_{\mathbb Z} \mathbb Z[x] \xra{d_2} \ldots. 
	$$
	Thus we also know the second page, namely we have $E_2^{n,q}=H^n(\mathbb G_a,\mathbb Z)_{q}$. Note that all odd rows are automatically zero.
	
	We claim that the spectral sequence degenerates at the second page. Since all terms $E_2^{n,q}$ are finitely generated over $\mathbb Z$ and by \Cref{lem: torsion in cohomology is elementary} all torsion they have is elementary, it is enough to check that all the differentials on the second page are zero modulo all primes $p$. Consider the analogous spectral sequence $E_{1,p}^{n,q}=H^q_{\mr{sing}}((\mathbb C\mathbb P^\infty)^n,\mathbb F_p) \Rightarrow  H^{n+q}_{\mr{sing}}(K(\mathbb Z,3),\mathbb F_p)$ for $\mathbb F_p$-cohomology; similarly to the above its second page looks as $E_{2,p}^{n,q}=H^n((\mathbb G_a)_{\mathbb F_p},\mathbb F_p)_{2q}$ and as we have seen the reduction map $E_2^{n,q}/p\ra E_{2,p}^{n,q}$ is an embedding. Thus it will be enough to show that $E_{*,p}^{*,*}$ degenerates at the second page for any $p$.  All differentials in $E_{*,p}^{*,*}$ are generated by maps between some $\mathbb F_p$-cohomology of some spaces and thus commute with the action of the Steenrod algebra $\mc A_p$. By \cite[Proposition in 4.27]{Jantzen} the algebra generators $w_{p^i}$ in fact are related by the Frobenius $F_{\mathbb G_a}\colon \mathbb G_a\ra \mathbb G_a$, namely $w_{p^i}=F_{\mathbb G_a}^*w_{p^{i-1}}$. From the topological point of view, if $x\in E_{1,p}^{1,2}\simeq H^2(\mathbb C\mathbb P^\infty,\mathbb F_p)$ is a generator whose class in $E_{2,p}^{1,2}=H^2((\mathbb G_a)_{\mathbb F_p},\mathbb F_p)_2$ is equal to $w_1\in E_{2,p}^{1,2}$, then such $w_{p^i}$ comes as a generator $x^{p^i}\in H^{2p^i}(\mathbb C\mathbb P^\infty,\mathbb F_p)$ and is expressed more functorially as $P^iP^{i-1}\ldots P^1 w_1$, where $P^i$ denotes $i$-th Steenrod power operation. Also recall that $\ol v_{p^i}=\beta_p(w_{p^i})$. Since $w_{p^i}$ and $v_{p^i}$ together generate $E_{*,p}^{*,*}$ and are obtained from $w_1$ by applying cohomological operations it is enough to show that $d_{n,p}(w_1)=0$ for any $n$. This is obvious for $n>2$ and for $n= $ we have $d_{2,p}(w_1)=0$ since $d_{2,p}(w_1)\in E_{2,p}^{3,1}=H^3((\mathbb G_a)_{\mathbb F_p},\mathbb F_p)_{1}=0$.

	Even though the degeneration of the spectral sequence a priori only gives the description of a certain associated graded of $H^{*}_{\mr{sing}}(K(\mathbb Z,3),\mathbb Z)$, we claim that there also exists a natural isomorphism of the latter with $E_2^{*+*}$. Indeed, let $c$ be the generator of $H^3(K(\mathbb Z,3),\mathbb Z)\simeq \mathbb Z$ which goes to $v_1\in E_{2}^{1,2}\simeq H^1(\mathbb G_a,\mathbb Z)_{2}$ under the natural map and consider its reduction $\ol c$ which generates $H^3(K(\mathbb Z,3),\mathbb F_p)\simeq \mathbb F_p$.  Then putting $c_{p^i}\coloneqq P^iP^{i-1}\ldots P^1\ol c$, and $d_{p^i}=\beta_p(c_{p^i})$ in the case of odd $p$, we get an isomorphism 
	$$
	\mathbb F_2[\ol c,c_2,c_{4},\ldots]\simeq H^*(K(\mathbb Z,3),\mathbb F_2)
	$$
	and 
	$$
	\Lambda_{\mathbb F_p}^*(\ol c, c_p,c_{p^2},\ldots)\otimes_{\mathbb F_p}\Sym_{\mathbb F_p}^*( d_p,d_{p^2}, d_{p^3},\ldots)\simeq H^*(K(\mathbb Z,3),\mathbb F_p)
	$$
	for $p$ odd. Moreover, the Bockstein $\beta_p$ is acting on $H^*(K(\mathbb Z,3),\mathbb F_p)$ by analogous formulas analogous to the ones we had in the course of the proof of \Cref{lem: torsion in cohomology is elementary}
	and by the same argument it follows that the  $p$-primary part in $H^{*}_{\mr{sing}}(K(\mathbb Z,3),\mathbb F_p)$ is killed by $p$. Sending $c$ to $v_1$, $c_{p^i}$ to $w_{p^i}$ and $d_{p^i}$ to $v_{p^i}$ defines an isomorphism between $H^*(K(\mathbb Z,3),\mathbb F_p)$ and $H^*((\mathbb G_a)_{\mathbb F_p},\mathbb F_p)$, which, moreover, respects Bocksteins on both sides. Finally, describing $H^*(K(\mathbb Z,3),\mathbb Z)$ in terms of $\Im \ \! \beta_p$ for various $p$ and the class $c$ as in \Cref{prop: description of the cohomology of BG_a} this extends to the isomorphism of graded algebras
	$$
	\bigoplus_{n\ge 0} H^{n}_{\mr{sing}}(K(\mathbb Z,3),\mathbb Z)\simeq \bigoplus_{n\ge 0} \left(\bigoplus_{i=0}^n H^i(B\mathbb G_a,\mc O_{B\mathbb G_a})_{n-2i}\right)
	$$
	as we claimed.
\end{rem}

\numberwithin{thm}{section}

\bigskip

\noindent Dmitry~Kubrak, {\sc Department of Mathematics, Massachusetts Institute of Technology,  Cambridge, United States}
\href{mailto:dmkubrak@mit.edu}{dmkubrak@mit.edu}

\smallskip

\noindent 
Artem~Prikhodko, {\sc Department of Mathematics, National Research University Higher School of Economics, Moscow; Center for Advanced Studies, Skoltech, Moscow,}
\href{mailto:artem.n.prihodko@gmail.com}{artem.n.prihodko@gmail.com}

\begin{bibdiv}
\addcontentsline{toc}{section}{\protect\numberline{}References}
\begin{biblist}

\bib{Alper}{article}{
author={Alper, Jarod},
title={Good moduli spaces for Artin stacks},
journal={Annales de l'Institut Fourier},
volume={63},
number={6},
date={2013},
pages={2349-2402},
}

\bib{Alper-HL-Heinloth}{article}{
	author={Alper, Jarod},
	author={Halpern-Leistner, Daniel},
	author={Heinloth, Jochen},
	title={Existence of moduli spaces for algebraic stacks.},
	eprint={https://arxiv.org/abs/1812.01128},
	year={2018}
}

\bib{Antieau-Bhatt-Mathew}{article}{
author={Antieau, Benjamin},
author={Bhatt, Bhargav},
author={Mathew, Akhil},
title={Counterexamples to Hochschild--Kostant--Rosenberg in characteristic $p$},
year={2019},
eprint={https://arxiv.org/pdf/1909.11437.pdf}
}

\bib{BarthelotOgus}{book}{
   title =     {Notes on crystalline cohomology},
   author =    {Berthelot, Pierre},
   author =    {Ogus, Arthur},
   publisher = {Princeton Univ. Pr.},
   isbn =      {9780691082189,0691082189},
   date =      {1978},
   series =    {Mathematical Notes},
}

\bib{Bhatt_derivedDeRham}{article}{
author={Bhatt, Bhargav},
title={Completions and derived de Rham cohomology},
date={2012},
eprint={https://arxiv.org/abs/1207.6193},
}
\bib{Bhatt_pAdicDerivedDeRham}{article}{
title={$p$-adic derived de Rham cohomology},
author={Bhatt, Bhargav},
year={2012},
eprint={https://arxiv.org/abs/1204.6560},
}

\bib{BhattdeJong}{article}{
	title={Crystalline cohomology and de Rham cohomology},
	author={Bhatt, Bhargav},
	author={de Jong, Aise},
	year={2011},
	eprint={https://arxiv.org/abs/1110.5001},
}

\bib{BMS2}{article}{
author={Bhatt, Bhargav},
author={Morrow, Mathew},
author={Scholze, Peter},
title={Topological Hochschild homology and integral $p$-adic Hodge theory},
journal={Publ. Math. de l'IHES Sci.},
volume={129},
date={2019}
}

\bib{Bial-Bir}{article}{
author={Bialynicki-Birula, Andrzej},
title={Some theorems on actions of algebraic groups},
journal={Annals of mathematics},
date={1973},
pages={480--497}
}

\bib{MathewBrantner_LLambda}{article}{
author={Brantner, Lukas},
author={Mathew, Akhil},
title={Deformation Theory and Partition Lie Algebras},
date={2019},
eprint={https://arxiv.org/abs/1904.07352}
}

\bib{HodgeIII}{article}{
author = {Deligne, Pierre},
title = {Th\'eorie de Hodge: III},
journal = {Publications Math\'ematiques de l'IH\'ES},
volume = {44},
pages = {5-77},
year = {1974}
}

\bib{DeligneIllusie}{article}{
author={Deligne, Pierre},
author={Illusie, Luc},
title={Rel\`evements modulo $p^2$ et d\'ecomposition du complexe de de Rham},
journal={Invent. Math.},
volume={89},
date={1987},
pages={247-270}
}

\bib{Drinfeld}{article}{
author={Drinfeld, Vladimir},
title={On algebraic spaces with an action of $\mathbb G_m$},
year={2013},
eprint={https://arxiv.org/abs/1308.2604}
}


\bib{FontaineMessing_padciPeriods}{article}{
author={Fontaine, Jean Marc},
author={Messing, William},
title={$p$-adic periods and $p$-adic \'etale cohomology},
journal={Current trends in arithmetical algebraic geometry (Arcata, Calif., 1985), Contemporary math. (67), Amer. Math. Soc., Providence, RI,},
date={1987}
}
\bib{Franjou}{article}{
author={Franjou, Vincent},
title={Cohomologie de de Rham entiere (Integral de Rham cohomology)},
eprint={https://arxiv.org/abs/math/0404123}
}

\bib{FranjouKallen}{article}{
author={Franjou, Vincent},
author={van der Kallen, Wilberd},
title={Power reductivity over an arbitrary base},
journal={Doc.Math. Extra vol: Suslin},
volume={},
date={2010},
pages={171-195}
}

\bib{GaitsRozI}{book}{
title =     {A Study in Derived Algebraic Geometry, Volume I: Correspondences and Duality},
author =    {Gaitsgory, Dennis},
author =    {Rozenblyum, Nick},
publisher = {American Mathematical Society},
isbn =      {1470435691,9781470435691},
date =      {2017},
series =    {Mathematical Surveys and Monographs},
}

\bib{EGA_III1}{article}{
author={Grothendieck, Alexander},
title={\'El\'ements de g\'eom\'etrie alg\'ebrique III. \'Etude cohomologique des faisceaux coh\'erents, Premi\`ere partie},
journal={Publ. Math. de l'Inst. des Hautes \'Etudes Sci.},
year={1961},
volume={11},
}

\bib{EGA_IV3}{article}{
author={Grothendieck, Alexander},
title={\'El\'ements de g\'eom\'etrie alg\'ebrique IV. \'Etude locale des sch\'emas
et des morphismes de sch\'emas. Troisi\`eme partie},
journal={Publ. Math. de l'Inst. des Hautes \'Etudes Sci.},
year={1966},
volume={28},
}
\bib{EGA_IV4}{article}{
author={Grothendieck, Alexander},
title={\'El\'ements de g\'eom\'etrie alg\'ebrique IV. \'Etude locale des sch\'emas
et des morphismes de sch\'emas. Quatri\`eme partie},
journal={Publ. Math. de l'Inst. des Hautes \'Etudes Sci.},
year={1967},
volume={32},
}

\bib{GAGA}{article}{
author={Hall, Jack},
title={GAGA Theorems},
eprint={https://arxiv.org/pdf/1804.01976},
year={2018}
}

\bib{HL-GIT}{article}{
author={Halpern-Leistner, Daniel},
title={The derived category of a GIT quotient},
journal={Journal of the American Mathematical Society},
year={2015},
volume={28 (3)},
issue={3},
pages={871-912},
}


\bib{HL-instability}{article}{
	author={Halpern-Leistner, Daniel},
	title={On the structure of instability in moduli theory},
	date={2018},
	eprint={https://arxiv.org/abs/1411.0627}
}

\bib{HL-prep}{article}{
	author={Halpern-Leistner, Daniel},
	title={Derived $\Theta$-stratifications and the $D$-equivalence conjecture},
	date={2020},
	eprint={https://arxiv.org/abs/2010.01127},
}
\bib{HLP_equiv_noncom}{article}{
author={Halpern-Leistner, Daniel},
author={Pomerleano, Daniel},
title={Equivariant Hodge theory and noncommutative geometry},
date={2015},
eprint={https://arxiv.org/abs/1507.01924}
}

\bib{RelaxedProperness}{article}{
author={Halpern-Leistner, Daniel},
author={Preygel, Anatoly},
title={Mapping stacks and categorical notions of properness},
date={2019},
eprint={https://arxiv.org/abs/1402.3204}
}

\bib{Illusie_Cotangent}{book}{
author={Illusie, Luc},
title={Complexe cotangent et d\'eformations. I},
series={Lecture notes in mathematics},
volume={223},
date={1971},
publisher={Springer-Verlag}
}
\bib{Illusie_CotangentII}{book}{
author={Illusie, Luc},
title={Complexe cotangent et d\'eformations. II},
series={Lecture notes in mathematics},
volume={283},
date={1972},
publisher={Springer-Verlag}
}

\bib{Jantzen}{book}{
  author={Jantzen, Jens Carsten},
  title={Representations of Algebraic Groups},
  isbn={9780821843772},
  lccn={2003058381},
  year={2007},
  series={Mathematical surveys and monographs},
  publisher={American Mathematical Society}
}

\bib{Kaledin}{article}{
author={Kaledin, Dmitry},
title={Non-commutative Hodge-to-de Rham degeneration via the method of Deligne-Illusie},
journal={Pure Appl. Math. Quot.},
volume={4},
date={2008},
}
\bib{Kaledin2}{article}{
author={Kaledin, Dmitry},
title={Spectral sequences for cyclic homology},
journal={Algebra, Geometry, and Physics in the 21st Century},
volume={324},
date={2017},
}

\bib{KubrakTravkin}{article}{
author={Kubrak, Dmitry},
author={Travkin, Roman},
title={Resolutions with conical slices and descent for the Brauer group classes of certain central reductions of differential operators in characteristic $p$},
date={2016},
}

\bib{MoretLaumon}{book}{
title={Champs alg\'ebriques},
author={Laumon, G\'erard},
author={Moret-Bailly, Laurent},
publisher={Springer},
year={2000},
volume={39},
}

\bib{LurHTT}{article}{
      author={Lurie, Jacob},
       title={Higher Topos Theory},
       date={2009},
       journal={Princeton Univ. Pr.},
}

\bib{LurHA}{article}{
      author={Lurie, Jacob},
       title={Higher Algebra},
       date={2017},
      eprint={http://www.math.harvard.edu/~lurie/papers/HA.pdf},
}

\bib{Lur_SAG}{article}{
author={Lurie, Jacob},
title={Spectral Algebraic Geometry },
date={2018},
eprint={https://www.math.ias.edu/~lurie/papers/SAG-rootfile.pdf},
}

\bib{MumfordFogartyKirvan}{book}{
title={Geometric invariant theory},
author={Mumford, David},
author={Fogarty, John},
author={Kirwan, Frances Clare},
publisher={Springer-Verlag},
year={1994},
series={Ergebnisse der Mathematik und ihrer Grenzgebiete. 2. Folge 34},
edition={3rd enl. ed},
}

\bib{Olsson}{article}{
author = {Olsson, Martin C.},
title = {On proper coverings of Artin stacks},
journal = {Advances in Mathematics},
volume = {198},
pages = {93-106},
year = {2005}
}

\bib{Poonen}{book}{
author={Poonen, Bjorn},
title={Rational Points on Varieties},
isbn={9781470443153},
series={Graduate studies in mathematics},
url={https://books.google.ru/books?id=TjdFtAEACAAJ},
year={2017},
publisher={Amer. Mathematical Society}
}

\bib{PortaYu_StacksGAGA}{article}{
title={Higher analytic stacks and GAGA theorems},
author={Porta, Mauro},
author={Yu, Tony Yue},
year={2014},
journal={Advances in Mathematics},
volume={302},
}

\bib{Preygel}{article}{	
	title={Thom-Sebastiani and Duality for Matrix Factorizations},
	author={Preygel, Anatoly},
	date={2011},
	eprint={https://arxiv.org/abs/1101.5834}
}

\bib{Pridham}{article}{
author = {Pridham, Jon P.},
title = {Presenting higher stacks as simplicial schemes},
journal = {Advances in Mathematics},
volume = {238},
pages = {184-245},
year = {2015},
issn = {0001-8708},
}

\bib{Rydh_NoetherianApprox}{article}{
author={Rydh, David},
title={Noetherian approximation of algebraic spaces and stacks},
volume={422},
journal={Journal of Algebra},
publisher={Elsevier BV},
year={2015},
pages={105--147}
}

\bib{Nick_FilteredCats}{article}{
      author={Rozenblyum, Nick},
       title={Filtered colimits of $\infty$-categories},
       date={2012},
      eprint={http://www.math.harvard.edu/~gaitsgde/GL/colimits.pdf},
}

\bib{Satriano}{article}{
author={Satriano, Matthew},
title={de Rham Theory for Tame Stacks and Schemes with Linearly Reductive Singularities},
journal={Ann. de l'Inst. Fourier},
date={2012},
volume={62 (6)},
issue={6},
}

\bib{Serre_GACC}{book}{
author={Serre, Jean Pierre},
title={Groupes alg\'ebriques et corps de classes},
series={Actualit\'es scientifiques et industrielles},
publisher={Hermann},
date={1975}
}

\bib{StacksProj}{article}{
author={The Stacks Project Authors},
title={Stacks Project},
eprint={https://stacks.math.columbia.edu/},
date={2020},
label={SP20}
}

\bib{Sumihiro}{article}{
author={Sumihiro, Hideyasu},
title={Equivariant completion},
journal={J. Math. Kyoto Univ.},
year={1974},
volume={14(1)},
issue={1},
pages={1-28},
}

\bib{Teleman}{article}{
      author={Teleman, Constantin},
       title={The quantization conjecture revisited},
   journal={Ann. of Math.},
        isbn={0003-486X},
        date={2000},
      volume={152 (2)},
      issue={2},
}

\bib{Toen_AutoFlat}{article}{
title={Descente fid\'element plate pour les $n$-champs d’Artin},
author={To\"en, Bertrand},
journal={Compositio Mathematica},
volume={147},
number={5},
publisher={London Mathematical Society},
pages={1382--1412},
year={2011},
}

\bib{TV_HAGII}{article}{
      author={Toen, Bertrand},
      author={Vezzosi, Gabriele},
       title={Homotopical Algebraic Geometry II: Geometric Stacks and Applications},
   publisher={Amer. Mathematical Society},
        isbn={0821840991,9780821840993},
        date={2008},
      series={Memoirs of the American Mathematical Society (v. II)},
}

\bib{Totaro}{article}{
     author={Totaro, Burt},
     title={Hodge theory of classifying stacks},
     journal={Duke Math. J.},
     volume={167},
     date={2018}
}
\end{biblist}
\end{bibdiv}
\end{document}